\documentclass[11pt,reqno]{amsart}

\usepackage[T1]{fontenc}
\usepackage[utf8]{inputenc}

\usepackage{multirow}
\usepackage[hypcap=false]{caption}
\usepackage{mathrsfs}
\usepackage{amsmath}
\usepackage{amssymb}
\usepackage{amsfonts}
\usepackage{amsthm}
\usepackage{a4wide}
\usepackage{graphicx}
\usepackage{color}
\usepackage[sans]{dsfont}
\usepackage{linearA}
\usepackage{longtable}
\usepackage{pdflscape}
\usepackage{float}
\usepackage{cancel}
\usepackage{dsfont}
\usepackage{setspace}
\usepackage{array}
\usepackage{hyperref}
\hypersetup{colorlinks=true,allcolors=[rgb]{0,0,0.6}}
\usepackage[english]{babel}
\usepackage{enumitem}
\usepackage[square,numbers,sort&compress]{natbib}
\usepackage[shadow,textwidth=22mm,textsize=tiny]{todonotes}
\usepackage{cellspace} 
\newtheorem{theorem}{Theorem}[section]
\newtheorem{lemma}[theorem]{Lemma}
\newtheorem{proposition}[theorem]{Proposition}
\newtheorem{corollary}[theorem]{Corollary}

\newtheorem{remark}[theorem]{Remark}

\newtheorem{openproblem}[theorem]{Open Problem}

\numberwithin{equation}{section}

\DeclareMathOperator{\e}{e}
\renewcommand{\diamond}{\vartriangle}
\renewcommand{\d}{\mathrm{d}}
\renewcommand{\Re}{\mathrm{Re}}

\newcommand{\cI}{\mathcal{I}}
\newcommand{\cH}{\mathcal{H}}
\newcommand{\cN}{\mathcal{N}}
\newcommand{\cB}{\mathcal{B}}
\newcommand{\cK}{\mathcal{K}}
\newcommand{\cD}{\mathcal{D}}
\newcommand{\loc}{\mathrm{loc}}
\newcommand{\Dom}{\mathcal{D}}
\newcommand\C{\mathbb{C}}

\newcommand\Z{\mathbb{Z}}
\newcommand\Ker{\mathcal{N}}
\def\bbbone{{\mathchoice {\rm 1\mskip-4mu l} {\rm 1\mskip-4mu l}
{\rm 1\mskip-4.5mu l} {\rm 1\mskip-5mu l}}}
\def\one{\bbbone}
\renewcommand{\i}{\mathrm{i}}
\newcommand{\Wr}{\mathscr{W}}

   \newcommand{\smallO}[1]{\ensuremath{\mathop{}\mathopen{}o\mathopen{}\left(#1\right)}}

\author[J. Derezi{\'n}ski]{Jan Derezi{\'n}ski}
\address[J. Derezi{\'n}ski]{Department of Mathematical Methods in Physics \\
Faculty of Physics, University of Warsaw \\
Pasteura 5, 02-093 Warszawa, Poland}
\email{jan.derezinski@fuw.edu.pl}

\author[J. Faupin]{J{\'e}r{\'e}my Faupin}
\address[J. Faupin]{Institut Elie Cartan de Lorraine \\
Universit{\'e} de Lorraine,
57045 Metz Cedex 1, France}
\email{jeremy.faupin@univ-lorraine.fr}

\date{\today}

\begin{document}
\bibliographystyle{abbrv} \title[Perturbed Bessel operators]
{Perturbed Bessel operators}

\begin{abstract}
We study perturbed Bessel operators $L_{m^2}=- \partial^2_x + ( m^2 -
  \frac14 )\frac{1}{x^2} + Q(x)$ on $L^2]0,\infty[$, where
  $m\in\mathbb{C}$ and $Q$ is a complex locally integrable
  potential. Assuming that $Q$ is integrable near $\infty$ and
  $x\mapsto x^{1-\varepsilon}Q(x)$ is integrable near $0$, with
  $\varepsilon\ge0$, we construct solutions to $L_{m^2} f = - k^2 f$
  with prescribed behaviors near $0$. The special cases $m=0$ and
  $k=0$ are included in our analysis. Our proof relies on mapping
  properties of various Green's operators of the unperturbed Bessel operator.
 Then we determine all
 closed realizations of $L_{m^2}$ and show that they can be organized as holomorphic families of closed operators.
\end{abstract}

\maketitle

\tableofcontents

\section{Introduction}

One of the most important families of exactly solvable  1-dimensional
Schr\"odinger operators
is the family of {\em Bessel operators}
\begin{align}
 - \partial^2_x + \frac{\mathrm{c}}{x^2}. \label{eq:diff_expr}
\end{align}
As is well-known, it is  convenient to set $\mathrm{c} = m^2 - 1/4$, so that \eqref{eq:diff_expr} is rewritten as
\begin{align}
L_{m^2}^0 := - \partial^2_x + \Big ( m^2 - \frac14 \Big )
  \frac{1}{x^2} .
  \label{eq:defL_00}
\end{align}
There exists large literature devoted to Bessel operators, mostly
restricted to the case $m^2\in\mathbb{R}$ (see e.g. \cite{DeRe65_01,DS3,EE,Kato,ReedSimon} and references therein).
They are also interesting for complex $m^2$.
Their closed realizations  on $L^2]0,\infty[$
were  studied  in \cite{BuDeGe11_01, DeRi17_01}.

Many operators  in mathematics and physics
can be reduced to  Bessel operators. Here are a few examples:
\begin{enumerate}
  \item the
    usual {\em Laplacian in  dimension $d\geq3$},
    $m=\frac{d}{2}-1+\ell$,
$\ell=0,1,2,\dots$, see e.g.  \cite[Section 3]{DeRi17_01};
\item the {\em 2d Aharonov-Bohm Hamiltonian with magnetic flux $\theta$},
  $m=\frac{\theta}{2\pi}+n$, $n\in\Z$,
see e.g. \cite{AThe98,CoFe21,Pankrashkin} and \cite[Appendix B]{BuDeGe11_01};
\item the {\em Laplacian on a conical surface of angle $\alpha$},
  $m=\frac{2\pi n}{\alpha}$, $n\in\Z$;
  \item the {\em Laplacian on a wedge of angle $\alpha$ with Dirichlet
    or Neumann boundary conditions},
  $m=\frac{\pi n}{\alpha}$, $n\in\Z$;
\item perturbed Bessel operators with $m$ complex are used to define
  {\em Regge poles}, e.g. \cite{DeRe65_01,Bottino};
  \item three-body systems with contact interactions.  
\end{enumerate}

In this paper we would like to investigate   Bessel operators with
complex $m$
perturbed by complex valued locally integrable potentials $Q(x)$.
Our goal is to show that under some  assumptions on $Q$ boundary
conditions for
perturbed Bessel operators can be described in a similar way as for
 unperturbed ones.

Before
describing our results, let us
review  general Schr\"odinger operators on the half-line, and then
 unperturbed Bessel operators.

\subsection{Basic facts about Schr\"odinger operators on the half-line}
\label{basic1}
We follow mostly
\cite{DeGe19_01}.
Suppose that $]0,\infty[\ni x\mapsto V(x)$ is a function in
$L_\loc^1]0,\infty[$, possibly complex valued. Consider the expression
\begin{equation}L:=-\partial_x^2+V(x).\end{equation}
The basic meaning of $L$ used in our paper
will be that of
a linear map from $AC^1]0,\infty[$ to $L^1_{\mathrm{loc}}]0,\infty[$. Recall that  $AC]0,\infty[$ denotes the set of absolutely continuous functions from $]0 , \infty[$ to $\mathbb{C}$, that is, functions whose distributional derivative belong to $L^1_{\mathrm{loc}}]0,\infty[$, and  $AC^1]0,\infty[$ is the set of functions from $]0 , \infty[$ to $\mathbb{C}$ whose distributional derivative belongs to $AC]0,\infty[$.

Let  $\mathcal{N}(L)$ denote all functions in $AC^1[0,\infty[$
annihilated by $L$. The space
$\mathcal{N}(L)$ is always 2-dimensional.

Let $h_1$, $h_2$ be two linearly independent elements of $\cN(L)$.
The {\em canonical bisolution} of $L$, denoted $G_\leftrightarrow$, is
defined by the integral kernel
\begin{equation}
  G_{\leftrightarrow}(x,y) =\frac{1}{\Wr(h_1,h_2)}\big ( h_1(x) h_2( y ) - h_2( x ) h_1( y ) \big ),\label{canoni}
\end{equation}
where $\Wr(h_1,h_2)$ denotes the {\em Wronskian} of $h_1$ and $h_2$.
Note that \eqref{canoni} does not depend on the choice of
$h_1,h_2$. The operator \eqref{canoni} is usually unbounded on $L^2]0,\infty[$,
however it is very useful in the study of $L$.

We will use the term {\em Green's operator} as a synonym for a right inverse of
$L$.
 In other words, the integral kernel
  $G_\bullet(x,y)$ of Green's operator $G_\bullet$ satisfies
\begin{equation}\big(-\partial_x^2+V(x)\big) G_\bullet(x,y)=\delta(x-y).
    \end{equation}
Again, we do not insist on the boundedness of $G_\bullet$  on $L^2]0,\infty[$.

We have various types of Green's operators:
 \begin{enumerate}
 \item the {\em forward Green's operator} $G_{\to}$:
\begin{equation}
G_\to(x,y):=\theta(x-y ) G_\leftrightarrow(x,y),\label{green1}\end{equation}
 \item
   the {\em backward Green's
   operator} $G_{\leftarrow}$:
\begin{equation}
G_\leftarrow(x,y):=-\theta(y-x ) G_\leftrightarrow(x,y).\label{green2}\end{equation}
\end{enumerate}   These two Green's operators are forward,
resp. backward Volterra operators: when they
     act on a function,  they do not extend its 
     support  to the left, resp. to the right.
They   are uniquely defined given $L$: they   do not depend on the choice of $h_1,h_2$.

We also have
\begin{enumerate}[resume]
   \item the {\em two-sided Green's operator corresponding to the
       boundary condition given by $h_1$ near $0$ and $h_2$ near
       $\infty$}:
\begin{equation}    G_\bullet(  x , y ) :=
  \frac{1}{ \Wr(h_1,h_2)    }
\left \{
\begin{array}{lr}
  h_1(x)h_2(y)
  , & x < y , \\
 h_2(x)h_1(y) , & y<x.
\end{array}
\right.
\label{green-bowtie+}
\end{equation}
   \end{enumerate}
The expression   \eqref{green-bowtie+}
   depends only on  the choice of the 1-dimensional subspaces $\C
h_1$, $\C h_2$.

Let us now discuss  realizations of $L$ as closed densely defined operators on
$L^2]0,\infty[$.
There are two obvious choices: the minimal realization $L^{\min}$ and
the maximal realization $L^{\max}$. Their domains are given by
    \begin{align*}
      \cD(L^{\max})
& := \big \{ f \in L^2]0,\infty[ \, \cap \, AC^1]0,\infty[\quad \mid\ L f \in L^2]0,\infty[ \big \} ,\\
  \cD(L^{\min})&:=\text{the closure of }\{f\in\cD(L^{\max})\ \mid\
  f=0\text{ near }0\text{ and }\infty\},
  \end{align*}
the closure being taken with respect to the graph norm of $L^{\mathrm{max}}$.  In general, there may exist other
operators
$L^\bullet$ such that $  L^{\min}\subset L^\bullet\subset L^{\max}$ defined
by boundary conditions at $0$
and $\infty$.

Potentials $V$ considered in this paper usually vanish at $\infty$. In this
case
we do not need to specify boundary
conditions for $L$ near $\infty$. This implies that either 
\begin{align}L^{\min}=L^{\max}&,\label{domains0}\\
\text{or }\qquad
\label{domains}
  \dim\cD(L^{\max})/\cD(L^{\min})&=2.\end{align}
If \eqref{domains} is true, there exists a one-parameter family of
operators $L^\bullet$
that
satisfy
\begin{equation}\label{domains1}
  L^{\min}\subsetneq L^\bullet\subsetneq L^{\max}.
\end{equation}

Suppose that $L^\bullet$ satisfies \eqref{domains1} or coincides with
\eqref{domains0}. Let $\lambda$ belong to the resolvent set of
$L^\bullet$. Then the integral kernel of $(L^\bullet-\lambda)^{-1}$ has the
form of a two-sided Green's operator 
  \eqref{green-bowtie+} with appropriate $h_1$ and $h_2$.

\subsection{Basic facts about unperturbed Bessel operators}

We mostly follow
\cite{BuDeGe11_01, DeRi17_01}.
Let $m\in\C$ and $k\in \C$.
Consider the space $\cN(L_{m^2}^0+k^2)$, that is, solutions to the eigenvalue equation
\begin{equation}
L_{m^2}^0 f = - k^2 f .
\label{eigen0}
\end{equation}
Solving \eqref{eigen0} for $k=0$ is 
easy:
$\cN(L_{m^2}^0)$ is spanned by
\begin{align}
  x^{\frac12+m},\quad x^{\frac12-m},&\quad m\neq0; 
\\                                      x^{\frac12},\quad 
                                      x^{\frac12}\mathrm{ln}(x),&\quad m=0. 
   \end{align}
For $k\neq0$, \eqref{eigen0} can be reduced  to the Bessel
equation.
This justifies the name {\em Bessel operator} for \eqref{eq:defL_00}. 
 Here is a pair of solutions of \eqref{eigen0} for $k\neq0$:
\begin{align}\label{solq1}
  u_m^0(x,k)&:=\Big(\frac2k\Big)^m\sqrt{x}I_m(kx),\\
  v_m^0(x,k)&:=\Big(\frac{k}2\Big)^m\sqrt{x}K_m(kx),\label{solq2}
\end{align}
where $I_m$ is the modified Bessel function and $K_m$ the Macdonald
function (see Subsection \ref{subsec:whittaker} below). Note that  $u_{m}^0(\cdot,k)$
 behaves as $\frac{x^{\frac12+m}}{\Gamma(m+1)}$ near zero and $v_m^0$ decays exponentially
 at infinity.
 Their normalization is chosen in such a way that
 their Wronskians
 are $1$
and they have a limit at $k=0$:
\begin{align}
 u_m^0(x,0)&=\frac{x^{\frac12+m}}{\Gamma(m+1)},\\
  v_m^0(x,0)&=\frac{\Gamma(m)x^{\frac12-m}}{2},\quad \Re(m)\ge0, \, m \neq 0.
  \end{align}

  It is convenient to  introduce another
   solution
of the unperturbed eigenequation \eqref{eigen0}, which differs from $v^0_m(x,k)$ only
by a different normalization:
\begin{equation}
  w_m^0(x,k)=\sqrt{\frac{2k}{\pi}}\Big(\frac{2}{k}\Big)^mv_m^0(x,k)=
  \sqrt{\frac{2xk}{\pi}}K_m(kx).
\end{equation}
Note that $w_m^0(x,k)=w_{-m}^0(x,k)\sim \e^{-kx}$  for $x\to\infty$.

The Bessel operator for $m=0$ often needs a separate treatment.
Note, for instance, that
   $v_0^0(\cdot,k)$ does not have a limit at $k=0$. To treat the case
   $m=0$ in a satisfactory way it is useful to introduce a family of eigenfunctions of $L_0^0$:
\begin{align}
p_0^0(x,k)&
:=-v_0^0(x,k)-\Big(\mathrm{ln}\Big(\frac{k}{2}\Big)+\gamma\Big) u_0^0(x,k),
  \end{align}
where $\gamma$ denotes Euler's constant. At $k=0$ it coincides with the logarithmic solution:
\begin{align*}
p_0^0(x,0)&=x^{\frac12}\mathrm{ln}(x).
  \end{align*}

 We will often assume that
$\Re(m)\geq0$, because $L_{m^2}^0$ depends only on $m^2$.  Based on  the
behavior near zero of its  eigenfunctions,   one can
distinguish 3 regimes:

\begin{enumerate}
\item $\Re(m)>0$. Eigensolutions of $L_{m^2}^0$ can be divided into 
  {\em principal}, that means proportional to $u_m^0$, and {\em
    non-principal}, all the others. Principal solutions
  behave as $x^{\frac12+m}$ and are more regular than  non-principal
  ones, which behave as $x^{\frac12-m}$.
  \item $\Re(m)=0$, $m\neq0$. 
    Eigensolutions of $L_{m^2}^0$ 
    are spanned by $ u_m^0$ and  $u_{-m}^0$,
    with a comparable behavior $x^{\frac12+m}$ and $x^{\frac12-m}$
    near zero.
  \item $m=0$.
        Eigensolutions of $L_{m^2}^0$ 
    are spanned by $u_0^0$ and  $p_0^0$. Those proportional to
    $u_0^0$ are again called {\em principal}, the remaining
    ones are called {\em non-principal}. Principal
    solutions behave as $x^{\frac12}$ and are more regular than
    non-principal ones behaving as $x^{\frac12}\mathrm{ln}(x)$.
\end{enumerate}

As explained in the previous subsection, with
 $L_{m^2}^0+k^2$
 one can associate various
 Green's operators
and the  canonical bisolution. The most important are
 \begin{enumerate}
   \item The  canonical bisolution 
$G_{m^2\leftrightarrow}^0(-k^2)$;
 \item the {\em forward Green's operator} $G_{m^2,\to}^0(-k^2)$;
 \item
   the {\em backward Green's
   operator} $G_{m^2,\leftarrow}^0(-k^2)$;
   \item the {\em two-sided Green's operator with homogeneous boundary
     conditions}
     $G_{m^2,\bowtie}^0(-k^2)$.
   \end{enumerate}
   Additionally, for $m=0$ we will use
\begin{enumerate}[resume]
\item the {\em two-sided  Green's
   operator logarithmic near zero}  $G_{0,\diamond}^0(-k^2)$.
\end{enumerate}

Green's operators $G_{m^2\leftrightarrow}^0(-k^2)$,  $G_{m^2,\to}^0(-k^2)$,
 $G_{m^2,\leftarrow}^0(-k^2)$ and    $G_{m^2,\bowtie}^0(-k^2)$ are defined as  in \eqref{canoni},
\eqref{green1}, \eqref{green2}, resp. \eqref{green-bowtie+}
where  we  put
$h_1(x)=u_m^0(x,k)$, $h_2(x)=v_m^0(x,k)$, and use the fact that their Wronskian is
$1$.

$G_{0,\diamond}^0(-k^2)$ is defined as in \eqref{green-bowtie+} where
we put $h_1(x)=p_0^0(x)$, $h_2(x)=u_0^0(x)$ and again 
  replace the Wronskian by 
$1$.

For the sake of brevity, we will  often abuse terminology, calling 
$G_{m^2,\bowtie}^0(-k^2)$  "two-sided" and
$G_{0,\diamond}^0(-k^2)$ "logarithmic".
However,  both  
are two kinds of two-sided Green's operators according to the
terminology of Subsection \ref{basic1}.

Let us now sketch the theory of closed realizations
 $L_{m^2}^0$ on $L^2]0,\infty[$.
 First of all, we can define  the minimal and maximal  
realization of $L_{m^2}^0$ denoted by $L_{m^2}^{0,\min}$  and
$L_{m^2}^{0,\max}$, respectively. They satisfy
\begin{align}|\Re(m)|\geq1 \text{ implies }&\quad L_{m^2}^{0,\min}=L_{m^2}^{0,\max},\label{domains0.}\\
\label{domains.}
|\Re(m)|<1 \text{ implies }& \quad \dim\cD(L_{m^2}^{0,\max})/\cD(L_{m^2}^{0,\min})=2.\end{align}
Thus for $|\Re(m)|<1$ there exists a 1-parameter family of closed  
realisations of $L_{m^2}^0$  between  
$L_{m^2}^{0,\min}$  and $L_{m^2}^{0,\max}$ defined  
by boundary conditions at zero.  To describe these realizations
one can introduce
the following three families of Bessel operators
\begin{align}\label{closed1}
  \{-1<\Re(m)\}\ni m&\mapsto H_m^0,\\
  \{-1<\Re(m)<1
  \}\times\big(\C\cup\{\infty\}\big)\ni (m,\kappa)&\mapsto
                                                    H_{m,\kappa}^0,
  \label{closed2}\\
 \big(\C\cup\{\infty\}\big)\ni \nu&\mapsto H_0^{0,\nu}.  \label{closed3}
\end{align}

 The family $H_m^0$ is the most basic one.
It  is holomorphic on $\{-1<\Re(m)\}$ (see Appendix \ref{Holomorphic families of closed operators} for the definition of holomorphic families of closed operators).  For $1\leq\Re(m)$ it is
 the unique closed realization of $L_{m^2}^0$. Then it is extended to the
 strip $-1<\Re(m)<1$ by analytic continuation. Its domain is defined
 by the boundary condition $\sim x^{\frac12+m}$  at zero, called {\em
   homogeneous} or {\em pure}. In other words, functions in the domain of $H_m^0$ belong to the domain of the maximal operator $L_{m^2}^{0,\max}$ and behave as $x^{\frac12+m}$ near $0$.

 The operator $ H_{m,\kappa}^0$
is defined by the boundary condition $\sim x^{\frac12+m}+\kappa
x^{\frac12-m}$ at zero. For $m=0$ and all $\kappa$ we simply have
$H_{0,\kappa}^0=H_0^0$. The map \eqref{closed2} is holomorphic except for a
singularity at $(m,\kappa)=(0,-1)$ (see
 Proposition 3.11(ii) in \cite{DeFaNgRi20_01}).

Finally, for the special case $m=0$, 
$H_0^{0,\nu}$ is defined by the boundary condition
$\sim x^{\frac12}\mathrm{ln}(x)+\nu x^{\frac12}$ at zero. The map \eqref{closed3} is
holomorphic.

For $\Re(m)>-1$ and $\Re(k)>0$ the two-sided Green's operator (with
pure boundary conditions) is
bounded on $L^2]0,\infty[$ and coincides with the resolvent of $H_m^0$:
\begin{equation*}
G_{m^2,\bowtie}^0(-k^2)=\big(H_m^0+k^2\big)^{-1}. 
\end{equation*}
It should, however, be remarked that
the integral kernel $G_{m^2,\bowtie}^0(-k^2,x,y)$ is well
defined and useful also for other values of $k$ and $m$, when it 
does not define a bounded operator.

    \subsection{Overview of main results}

    Our paper is devoted to {\em perturbed Bessel operators}, that is,
to    operators of the form
\begin{align*}
L_{m^2} := - \partial^2_x +  \Big ( m^2 - \frac14 \Big ) \frac{1}{x^2} + Q(x). 
\end{align*}
We  allow $m$ to be complex and $Q$ to be complex-valued. Throughout the paper, we will assume that $Q\in L^1_{\mathrm{loc}}]0,\infty[$.

Note that the condition $\Re(m)>-1$ which we saw e.g in 
\eqref{closed1}
appears in several places in our analysis. One can argue that the 
case $\Re(m)\leq-1$ is less important for applications, because then 
 $x^{\frac12+m}$ is not square integrable at zero. 
Nevertheless,  if possible we keep $m$ arbitrary, without restricting
it to $\Re(m)>-1$.

Our first concern in this paper is  the construction of solutions
in $AC^1]0,\infty[$ to the equation
    \begin{equation}
      L_{m^2} f = - k^2 f
      \label{eigen}
    \end{equation}
    with a prescribed behavior near $0$ or near $\infty$.
    We will show that under some conditions on perturbations these
    solutions are quite similar to  solutions of the unperturbed
    eigenequation \eqref{eigen0}.

First of all we show that if the perturbation is  slightly weaker than $1/x^2$ near zero, then
there exists  a solution of the perturbed equation approximating the 
principal solution, as described
in the following proposition:
\begin{proposition}
\label{pr:opo0}
Let $\Re(m)\geq0$, $k\in\C$
and suppose that
 \begin{align}\label{ass1}
   \int_{0}^1 x|Q(x)| \mathrm{d}x& < \infty,\quad
    \text{if } m\neq0;\\\label{ass2}
    \int_{0}^1 x(1+|\mathrm{ln}(x)|) |Q(x)| \mathrm{d}x &< \infty,
    \quad \text{if } m=0.
    \end{align}
    Suppose that $g^0$ is a solution of \eqref{eigen0} such that
   $g^0(x) =\mathcal{O}(x^{\frac12+\Re(m)}) $ near $0$. 
    Then, 
 there exists a  unique solution $g
 \in AC^1]0,\infty[$ to     \eqref{eigen} such that, 
\begin{align*}
  g(x)-g^0(x)&=o(x^{\frac12+\Re(m)}),\\
\partial_x  g(x)-\partial_x g^0(x)&=o(x^{-\frac12+\Re(m)}),\quad x\to0.
\end{align*}
\end{proposition}
 In order to be able to well approximate all unperturbed solutions,
including the more singular ones, we need to strengthen the assumption
on the perturbation. 
    \begin{proposition}
      \label{prop:reg1!}
Let $\Re(m)\geq0$, $k\in\C$
and suppose that
 \begin{align}\label{cono3}
   \int_{0}^1 x^{1-2\Re(m)} |Q(x)| \mathrm{d}x& < \infty,\quad
    \text{if } m\neq0;\\
    \int_{0}^1 x\big(1+(\mathrm{ln}(x))^2\big) |Q(x)| \mathrm{d}x &< \infty,
    \quad \text{if } m=0.
    \end{align}
 Then for any
 $g^0 \in AC^1]0,\infty[$ 
solving \eqref{eigen0}
 there exists a  unique  $g
 \in AC^1]0,\infty[$ solving
     \eqref{eigen}
    such that
\begin{align*}
g(x)-g^0(x)&=o(x^{\frac12+\Re(m)}),\\
\partial_x  g(x)-\partial_x g^0(x)&=o(x^{-\frac12+\Re(m)}),\quad x\to0.
\end{align*}
\end{proposition}

Here are consequences of Propositions  \ref{pr:opo0} and \ref{prop:reg1!}:
\begin{corollary}\label{coro1}
  Let $m\in\C$,  $k\in\C$
and suppose that
 \begin{align}\label{alig1-}
   \int_{0}^1 x^{1-\varepsilon} |Q(x)| \mathrm{d}x& < \infty,\quad
\varepsilon\geq0,\quad
  \Re(m)\geq-\frac\varepsilon2,\quad
                                                    m\neq0;\\
   \int_{0}^1 x(1+|\mathrm{ln}(x)|) |Q(x)| \mathrm{d}x &< \infty,
                                                     \quad m=0.\label{alig10}
\end{align}
Then 
there exists a unique   $ u_m(\cdot,k)  \in AC^1]0,\infty[$ that solves      \eqref{eigen}
 and satisfies
\begin{align*}
   u_m(x,k)- u_m^0(x,k)&=o(x^{\frac12+|\Re(m)|}),\\
\partial_x   u_m(x,k)-\partial_x  u_m^0(x,k)&=o(x^{-\frac12+|\Re(m)|}),\quad x\to0.
\end{align*}
\end{corollary}

\begin{corollary}\label{coro2}  Let $k\in\C$ and suppose that
 \begin{align}\label{alig2log}
   \int_{0}^1 x(1+(\mathrm{ln}(x))^2) |Q(x)| \mathrm{d}x &< \infty,\quad
m=0.
\end{align}
Then 
there exists a unique   $p_0(\cdot,k)  \in AC^1]0,\infty[$ that solves      \eqref{eigen}
 and satisfies
\begin{align*}
 p_0(x,k)-p_0^0(x,k)&=o(x^{\frac12}),\\
\partial_x  p_0(x,k)-\partial_x p_0^0(x,k)&=o(x^{-\frac12}),\quad x\to0.
\end{align*}
\end{corollary}

Note that if $|Q(x)|\lesssim |x|^{\alpha}$ near $0$, then Condition 
\eqref{alig1-} is satisfied for $\alpha>-2+\varepsilon$.

Conditions \eqref{ass1} and \eqref{ass2}
are the minimal assumptions near zero in the context of our paper.
They are enough to guarantee that
the behavior near zero of non-principal solutions 
is roughly as in  the unperturbed case:

\begin{proposition}\label{cor:u_not_L2a-}
  Let $\Re(m)\ge0$, $\Re(k)\ge0$. Under the assumptions \eqref{ass1} and \eqref{ass2},
for all $g\in\Ker(L_{m^2}+k^2)$, we have
\begin{align}
  g(x)=\mathcal{O}(x^{\frac12-\Re(m)}),\qquad   \partial_x g(x)=\mathcal{O}(x^{-\frac12-\Re(m)}) ,& \\
  g(x)=\mathcal{O}(x^{\frac12}\mathrm{ln}(x)),\qquad   \partial_x g(x)=\mathcal{O}(x^{-\frac12}\mathrm{ln}(x)) ,& \qquad 
x\to0. 
 \end{align}
\end{proposition}

As described in Proposition \ref{pr:opo0}, 
the above assumptions are enough
for the existence of $ u_m$ with
$\Re(m)\geq0$.
However, it seems that to have {\em distinguished} non-principal solutions 
 one 
needs to impose stronger conditions on $Q$, as described in Corollaries 
\ref{coro1} and \ref{coro2}:
In particular, if $\varepsilon\geq0$ and Condition \eqref{alig1-} holds, then
$ u_m$ is constructed  in the region
$\Re(m)\geq-\frac\varepsilon2$.
This suggests the following question, which we believe  is open and interesting:

\begin{openproblem}
Let 
 $Q$ satisfy condition \eqref{alig1-} with $\varepsilon>0$. 
Does it imply that the  function $m\mapsto u_m(x,k)$  extends   
holomorphically (or at least meromorphically) to the whole $\mathbb{C}$? (This is true for the 
 Coulomb potential \cite{DeRi18_01}.). 
\label{op1}\end{openproblem}

Let us now consider the behavior near infinity.
To prove the existence of solutions well approximating  exponentially
decaying solutions, called {\em
  Jost solutions}, we need the so-called short-range
condition on the potential.
\begin{proposition}\label{prop:reginfty0}
Let $m \in \C$. Suppose that
\begin{align*}
\int_{1}^\infty|Q(x)| \mathrm{d}x < \infty .
\end{align*}
Let $k\neq0$ be such that  $\Re(k)\geq0$. Then there exists a 
 unique solution $w_{m}( \cdot , k ) = w_{-m}( \cdot , k ) \in
 AC^1]0,\infty[$ to
 \eqref{eigen} such that
\begin{align*}
  w_{m} ( x ,k) -
  w_{m}^0 ( x ,k) &=o(\e^{-x\Re(k)}),\\
\partial_x      w_{m} ( x ,k) -
\partial_x    w_{m}^0 ( x ,k) &=o(\e^{-x\Re(k)})
  , \qquad x \to \infty .
\end{align*}
\end{proposition}

Similarly as in unperturbed case, it is often convenient to use differently
normalized Jost solutions $v_m(x,k):=\sqrt{\frac{\pi}{2k}}\Big(\frac{k}2\Big)^m
  w_{m} ( x ,k) $.

Proposition \ref{prop:reginfty0} does not cover the zero energy, that
is, $k=0$. 
To handle this case we need to impose stronger conditions
on the decay of perturbations, as described in the following two
propositions.

\begin{proposition}\label{prop:reginfty+intro}
Let $m\in\C$, $k=0$.  Suppose that
\begin{align*}  
\int_{1}^\infty x^{\delta}|Q(x)| \mathrm{d}x& < \infty ,\quad 
                                              \text{ if } m\neq 0, \quad \text{with } 
                                          \delta=1+
                                              2\max\big(\Re(m),0\big);\\
\int_{1}^\infty x(1+\mathrm{ln}(x))|Q(x)| \mathrm{d}x& < \infty ,\quad 
 \text{ if }                                                 m=0.
\end{align*}
Then there exists a unique
   $q_m
 \in AC^1]0,\infty[$ solving   \eqref{eigen}  at $k=0$ such that, 
\begin{align*}
  q_{m}(x)-x^{\frac12+m} &=o(x^{\frac12-\Re(m)}),\\
\partial_x q_{m}(x)-\partial_x x^{\frac12+m}&=o(x^{-\frac12-\Re(m)}),\quad x\to\infty. 
\end{align*}
\end{proposition}

\begin{proposition}\label{prop:reginftylogintro}
  Let $m=0$, $k=0$. Suppose that
  \begin{align*} 
\int_{1}^\infty x(1+(\mathrm{ln}(x))^2)|Q(x)| \mathrm{d}x& < \infty .
\end{align*}
 Then there exists a unique
   $q_{0,\mathrm{ln}}
 \in AC^1]0,\infty[$ solving     \eqref{eigen}  for $k=0$ such that, 
\begin{align*}
  q_{0,\mathrm{ln}}(x)- x^{\frac12}\mathrm{ln}(x)&=o(x^{\frac12}),\\
  \partial_x q_{0,\mathrm{ln}}(x)-\partial_x x^{\frac12}\mathrm{ln}(x)&=o(x^{-\frac12}),\quad x\to\infty. 
\end{align*}
\end{proposition}

 The zero energy eigenequation near infinity 
is equivalent
to the zero energy eigenequation near zero. This follows from the identity
\begin{align}
-\partial_x^2+\Big(m^2-\frac14\Big)\frac1{x^2}+Q(x)&=
y^3\Big(
-\partial_y^2+\Big(m^2-\frac14\Big)\frac1{y^2}+\tilde Q(y)\Big)y,\\
y=\frac1x,\qquad&\quad \tilde Q(y):= y^{-4}Q(y^{-1}).
\end{align}
Note also a simple relationship between  the integral conditions 
near zero on $Q$ and near infinity on $\tilde Q$:
\begin{align}
\int_0^1x^{1-\varepsilon}|Q(x)|\d
  x&=\int_1^\infty y^{1+\varepsilon}|{\tilde Q}(y)|\d y,
  \\
  \int_0^1x(1+|\ln(x)|^\alpha)|Q(x)|\d x&=\int_1^\infty
                                          y(1+\ln(y)^\alpha)|{\tilde
                                          Q}(y)|\d y.
\end{align}
Thus one can derive Propositions \ref{prop:reginfty+intro}
and \ref{prop:reginftylogintro} from
the $k=0$ case of Corollaries \ref{coro1} and \ref{coro2}.

The main tools used in the construction of eigenfunctions are various
Green's operators for the unperturbed Bessel operator.
The forward Green's operator is used in Propositions
\ref{pr:opo0}, \ref{prop:reg1!} and their corollaries. For instance,
\begin{align}
   u_m(\cdot,k)&=\big(\one+G_{m^2,\to}^0(-k^2)Q\big)^{-1} u_m^0(\cdot,k),\\
    p_0(\cdot,k)&=\big(\one+G_{0,\to}^0(-k^2)Q\big)^{-1}p_0^0(\cdot,k).\label{pzero0}
\end{align}
The   backward Green's operator
is used in Propositions \ref{prop:reginfty0},
\ref{prop:reginfty+intro}
and \ref{prop:reginftylogintro}:
\begin{align}\label{lippmann}
   w_m(\cdot,k)&=\big(\one+G_{m^2,\leftarrow}^0(-k^2)Q\big)^{-1}
                 w_m^0(\cdot,k),\\
   q_m&=\big(\one+G_{m^2,\leftarrow}^0(0)Q\big)^{-1}
        u_m^0(\cdot,0),\\
     q_{0,\ln}&=\big(\one+G_{0,\leftarrow}^0(0)Q\big)^{-1}
                 p_0^0(\cdot,0).  
                 \end{align} 
                 In quantum physics the equation for the Jost solution
                 \eqref{lippmann} is called the {\em
  Lippmann--Schwinger Equation.}

 If \eqref{alig1-} holds and
$\frac\varepsilon2<\Re(m)$, then  Corollary \ref{alig10}
guarantees the existence only of $u_m(\cdot,k)$, but not of $u_{-m}(\cdot,k)$.
 Therefore, in this case it is more
complicated to describe non-principal solutions. 
One way to do this is to use
the two-sided
Green's operator with pure boundary conditions $G_{m,\bowtie}^0$
(where we assume that $\Re(m)\geq0$).
Unfortunately, 
 $\one+G_{m,\bowtie}^0(-k^2)Q$ may be not invertible. In
order to guarantee the invertibility we can \emph{compress}
$G_{m,\bowtie}^0(-k^2)$ to a sufficiently small interval $]0,a[$.
The corresponding compressed Green's operator is denoted
$G_{m,\bowtie}^{0(a)}(-k^2)$ (see
Appendix \ref{volterra}).

In the case $m=0$ one may prefer to use the logarithmic Green's
operator $G_\diamond^{0}(-k^2)$, or actually its compressed version
 $G_\diamond^{0(a)}(-k^2)$.

\begin{proposition}\label{prop:u_po}
Suppose the assumptions of Proposition
  \ref{pr:opo0} hold.  If $a$ is small enough, the following
  functions are well defined and solve \eqref{eigen} on $]0,a[$:
\begin{align}\label{cono}
   u_{-m}^{\bowtie(a)}(\cdot,k)&:=\big(\one+G_{m,\bowtie}^{0(a)}(-k^2)Q\big)^{-1} u_{-m}^0(\cdot,k),\\\label{cono2}
 p_0^{\diamond(a)}
  (\cdot,k)&:=\big(\one+G_{0,\diamond}^{0(a)}(-k^2)Q\big)^{-1}p_0^0(\cdot,k).
\end{align}
\end{proposition}

 In the case  $m=0$ the function $ p_0^{\diamond(a)}$
is well defined by \eqref{cono2} under
slightly less 
  restrictive  condition on $Q$ than $p_0 $ defined in Corollary \ref{coro2}: the difference is just one power of the logarithm less in \eqref{ass2} than in 
  \eqref{cono3}, which is not much. 
  However the difference between the assumptions for $
  u_{-m}^{\bowtie(a)}(\cdot,k)$ defined in \eqref{cono} and
$u_m(\cdot,k)$ defined in Corollary \ref{coro1} is quite substantial.

 Unfortunately, the construction \eqref{cono} and \eqref{cono2} has obvious drawbacks. It is  not very explicit: it involves inverting a complicated
integral operator.  It also depends on an
arbitrary parameter $a$ even if the dependence on $a$ is actually quite weak -- if
we change $a$, \eqref{cono} and \eqref{cono2} are shifted by a
multiple of the corresponding principal solution. Note that we
  cannot fix the value of $a$ once for all, because the invertibility
  of $\one+G_{m,\bowtie}^{0(a)}(-k^2)Q$ and
  $\one+G_{0,\diamond}^{0(a)}(-k^2)Q$ depends on $Q$
  and other parameters.

We will describe below
 an alternative approach, which leads to a simpler
description of non-principal solutions for $\Re(m)>0$.
We choose  a non-negative integer $n$.
We 
expand the denominator  \eqref{cono}
into a formal power series,
retaining $n$ first terms. For definiteness, we fix $a=1$ (quite arbitrarily)
and set
\begin{align}\label{defrn0}
   u_{{-}m}^{0[n]}(x,k)&=\sum_{j=0}^n
                                 (-G_{\bowtie}^{0(1)}(0)Q)^j
                                                                  u_{{-}m}^0(x,k).
\end{align}

 \begin{proposition}   \label{boundcon-}
   Let $\Re(k)\geq0$.
Let  $n$ be a nonnegative integer such that Condition 
\eqref{alig1-} is satisfied for $-\frac{\varepsilon}2(n+1)\leq \Re(-m)\leq0$.
Then
there exists a unique solution $ u_{-m}^{[n]}(\cdot,k)$ in $AC^1]0,\infty[$ of
\eqref{eigen} such that
\begin{align}
    u_{-m}^{[n]}(x,k)-u_{-m}^{0[n]}(x,k)
  &=o(x^{\frac12+\Re(m)}),\\
   \partial_x u_{-m}^{[n]}(x,k)-\partial_xu_{-m}^{0[n]}(x,k)
&=o(x^{-\frac12+\Re(m)}), \quad x\to 0.
  \end{align}
 \end{proposition}

 Thus for sufficiently large $n$ the function $u_{-m}^{[n]}(\cdot,k)$
 determines uniquely  a non-principal element of $\cN(L_{m^2}+k^2)$
 under much weaker 
assumptions than before.

Boundary conditions determined by
$   u_{{-}m}^{0[n]}(\cdot,k)$
still have  an unpleasant feature -- they depend on $k$.
If we want
to
have boundary conditions independent of $k$ we need to assume that
$|\Re(m)|<1$. Then it is reasonable to choose $k=0$, which we do setting
\begin{equation}\label{eq:boundary_intro}
u_{-m}^{0[n]}(x) :=u_{-m}^{0[n]}(x,0).
\end{equation}
In particular,  under the condition $| \Re(m)|<1$ in Proposition \ref{boundcon-}
we can replace $   u_{{-}m}^{0[n]}(\cdot,k)$ with $   u_{{-}m}^{0[n]}(\cdot)$.
This condition
is also  important in the
$L^2$ theory of perturbed Bessel operators as we explain below.

 In concrete cases, the function $u_{-m}^{0[n]}$ can be easily computed.
 For instance, if $Q$ has a Coulomb
singularity at $0$, such as $Q(x) = - \frac{ \beta }{ x
}\one_{]0,1]}(x)$ with $\beta\in\C$, then we need to take $n=1$ to
cover $|\Re(m)|<1$.
Then  in the generic case $m\neq\frac12$ we have
\begin{equation*}
u_{-m}^{0[1]}(x)=\frac{j_{\beta,-m}(x)}{\Gamma(1-m)}+\mathcal{O}(x^{\frac12+\Re(m)}), \quad j_{\beta,-m}(x):=x^{\frac12-m}\big ( 1 -  \frac{\beta x}{1-2m} \big ),
\end{equation*}
the function $j_{\beta,-m}$ being precisely the function used to define  Whittaker operators in
\cite{DeRi18_01,DeFaNgRi20_01}.

An important object of our analysis is the {\em Jost function} $\Wr_m(k)$, that
is the Wronskian of the two main solutions $u_m(\cdot,k)$ and 
$v_m(\cdot,k)$.
We prove that it is well-behaved as a function of
$k$:
\begin{proposition}\label{prop:wronsk_intro}
  Assume $\Re(m)>-1$, as well as \eqref{alig1-} if $m\neq 0$, or \eqref{alig2log} if $m=0$. Then
  \begin{align}\label{asim0}
 \lim_{|k|\to\infty} \Wr_{m}(k) = 1,
  \quad \Re(k)\geq0.
\end{align}
  \end{proposition}
Note the assumption $\Re(m)>-1$ that appears in the above proposition -- which again anticipates
 the basic condition needed in the $L^2$ analysis.

The last section of our paper,  Section \ref{closed_operators}, is devoted to closed realizations of $L_{m^2}$
on the Hilbert space
$L^2]0,\infty[$.  First we prove that
under the assumptions of Propositions \ref{pr:opo0} and \ref{prop:reginfty0},
we have
\begin{align}|\Re(m)|\geq1 \text{ implies }&\quad L_{m^2}^{\min}=L_{m^2}^{\max},\label{domains0..}\\
\label{domains..}
|\Re(m)|<1 \text{ implies }& \quad \dim\cD(L_{m^2}^{\max})/\cD(L_{m^2}^{\min})=2.\end{align}
Thus the basic picture is the same as in the
unperturbed case described in \eqref{domains0.} and \eqref{domains.}

In particular, for $|\Re(m)|<1$, beside the minimal and maximal
realizations, there exists a 1-parameter family of closed 
realizations of $L_{m^2}$
defined by boundary conditions
at zero.  Boundary conditions can be fixed by specifying continuous linear
functionals on $\cD(L_{m^2}^{\max})$ vanishing on
$\cD(L_{m^2}^{\min})$, called {\em boundary functionals}. 
  The method to descibe boundary functionals
which seems to work the best in our context uses the Wronskian at zero, that is
$\Wr(f,\cdot;0):=\lim\limits_{x\to0}\Wr(f,\cdot;x)$,
for appropriately chosen functions $f$. In practice the most
convenient  $f$ are approximate zero energy eigenfunctions
of $L_{m^2}$.

One can ask about distinguished bases of the boundary space
\begin{equation*}
  \cB_{m^2}:= \big(\cD(L^{\max}_{m^2})/\cD(L^{\min}_{m^2})\big)',
  \end{equation*}
where the prime denotes the dual.
Under the assumptions of Proposition \ref{pr:opo0}
one can always distinguish the {\em principal   boundary functional}.
For $0\le\Re(m)<1$ it can be defined as
$\Wr( x^{\frac12+m},\cdot;0)$. There are also
``non-principal boundary functional'', which lead to boundary conditions
roughly of the type
$x^{\frac12-m}$ for $m\neq0$, or $x^{\frac12}\mathrm{ln}(x)$ for $m=0$. In general their
choice is less canonical: under the assumptions of Proposition \ref{pr:opo0}, a basis of $\cB_{m^2}$, $0\le\Re(m)<1$, $m\neq 0$, is given by 
\begin{equation*}
\big(\Wr(x^{\frac12+m},\cdot;0),\Wr(u_{-m}^{\bowtie(a)}(\cdot,k),\cdot;0)\big),
\end{equation*}
with $a$ small enough as in Proposition \ref{prop:u_po}. Likewise, if $m=0$, then 
\begin{equation*}
\big(\Wr(x^{\frac12},\cdot;0),\Wr( p_0^{\diamond(a)},\cdot;0)\big),
\end{equation*}
is a basis of $\cB_0$.

Let us now impose the assumption
 \begin{align}\label{alig}
   \int_{0}^1 x^{1-\varepsilon} |Q(x)| \mathrm{d}x& < \infty.\end{align}
If $2>\varepsilon > 0$, then
 for $0\le\Re(m)\le\varepsilon/2$ we have a distinguished
non-principal boundary functional   given by 
$\Wr(x^{\frac12-m},\cdot;0)$ if $m\neq0$ and $\Wr(x^{\frac12}\ln(x),\cdot;0)$ if $m=0$.
  Thus we obtain three families of perturbed Bessel operators
\begin{align}\label{fami1}
  \Big\{-\frac\varepsilon2<\Re(m)\Big\}\ni m&\mapsto H_m,\\
  \Big\{|\Re(m)|<\frac{\varepsilon}{2}
  \Big\}\times\big(\C\cup\{\infty\}\big)\ni (m,\kappa)&\mapsto H_{m,\kappa}, \label{eq:fami2}\\
 \big(\C\cup\{\infty\}\big)\ni \nu&\mapsto H_0^{\nu},
\end{align}
fully analogous to the families of the unperturbed case.
All three families are holomorphic except for
a singularity of \eqref{eq:fami2} at $(m,\kappa)=(0,-1)$.
They are defined as the restrictions of $L_{m^2}$ to the domains:
\begin{align*}
\Dom(H_{ m}) & := \Big\{f\in \Dom(L_{m^2}^{\max})\mid
\, \Wr(x^{\frac12+m},f;0)=0\Big\},\\
  \Dom(H_{ m,\kappa}) & := \Big\{f\in \Dom(L_{m^2}^{\max})\mid
\Wr\big(x^{\frac12+m} + \kappa x^{\frac12-m},f;0
                        \big)=0\Big\},\quad\kappa\in\C,
                        \nonumber \\
\Dom(H_{ m,\infty}) &: = \Big\{f\in \Dom(L_{m^2}^{\max})\mid
\Wr\big( x^{\frac12-m},f;0
                      \big)=0\Big\},
  \\
\Dom(H_{0}^{\nu}) &: = \Big\{f\in \Dom(L_{0}^{\max})\mid
                    \Wr\big(\nu x^{\frac12} + x^{\frac12}\mathrm{ln}(x)
                    ,f;0
\big)=0\Big\},\quad\nu\in\C, \nonumber \\
\Dom(H_{ 0}^{\infty}) & := \Dom(H_{ 0}).
\end{align*}
The maps $m\mapsto H_m$ and $(m,\kappa)\mapsto H_{m,\kappa}$ are also continuous on $\{-\frac{\varepsilon}{2}\le\Re(m)\}$, respectively $\{|\Re(m)|\le\frac{\varepsilon}{2},\kappa\in\C\cup\{\infty\},(m,\kappa)\neq(0,-1)\}$ (continuous families of closed operators are defined similarly as holomorphic families of closed operators, see Appendix \ref{Holomorphic families of closed operators}).

 The holomorphic family \eqref{fami1} for $\Re(m)\geq1$ coincides with
$L_{m^2}^{\min}=L_{m^2}^{\max}$.  It involves the boundary conditions
that can be viewed as ``the most natural'', and which we call  {\em pure}. Note that \eqref{fami1} is restricted to $\{\Re(m)>-\frac{\varepsilon}{2}\}$. This leaves the following open question.

\begin{openproblem}
  Under the minimal conditions of Proposition \ref{pr:opo0},
 does  $m\mapsto H_m$ extend to $\{\Re(m)>-1\}$ holomorphically, or at least meromorphically?
\end{openproblem}

Let us now consider a nonnegative integer $n$. Under the assumption \eqref{alig},
$2(n+1)>\varepsilon>0$ and $0\le\Re(m)\le\frac{\varepsilon}2(n+1)$ we
can use
 the function  $u_{-m}^{0[n]}$ that was defined in \eqref{eq:boundary_intro}.
 Then every
non-principal boundary functional can be written as
\begin{equation}
  \Wr(\Gamma(1-m)u_{-m}^{0[n]}+\kappa x^{\frac12+m},\cdot;0)\label{nonpri}\end{equation}
for some $\kappa\in\C$.
Clearly, \eqref{nonpri} 
is proportional to $\Wr(x^{\frac12-m}+\kappa x^{\frac12+m},\cdot;0)$ for $n=0$.
For $n\ge1$ \eqref{nonpri}   is less canonical. If $n,n'$ are two 
integers and \eqref{nonpri} are 
well defined for $n$ and $n'$, then their difference is proportional 
to the principal boundary functional  $\Wr(x^{\frac12+m},\cdot;0)$.
Thus the set of non-principal boundary conditions can be viewed as a
1-dimensional affine space, where we can use
$  \Wr(\Gamma(1-m)u_{-m}^{0[n]},\cdot;0)$ as a possible ``reference point''.

The boundary functional \eqref{nonpri} can be
used to define a family of perturbed Bessel operators  which includes
all possible
boundary conditions at $0$:
\begin{align}
  \Big\{|\Re(m)|<\frac{\varepsilon}{2}(n+1)\Big\}\times\big(\C\cup\{\infty\}\big)\ni (m,\kappa)&\mapsto H_{m,\kappa}^{[n]}. \label{eq:fami2a}
\end{align}
Note that \eqref{eq:fami2a} is
less canonical than \eqref{eq:fami2}, however 
it is defined on a wider region.

The distinguished solutions to \eqref{eigen}
can be used to write down the resolvent of 
$H_m$ and its cousins with mixed boundary conditions.
For instance, the integral kernel of $(H_m+k^2)^{-1}$ coincides with
\eqref{green-bowtie+} with $h_1(x)=u_m(x,k)$ and $h_2(x)=v_m(x,k)$.

One of the main difficulties of the analysis
comes from the  need to consider separately the case $m=0$, because
generic estimates are not true due to logarithmic terms. This case is
actually very important -- it corresponds to the 2-dimensional
Laplacian in the {\em s-wave sector}.

The case $k=0$  also requires special care and is particularly important.
One can argue that the most natural way to define
boundary conditions at zero involves zero-energy eigenfunctions
\cite{DeFaNgRi20_01}.
Moreover,
the behavior of  zero energy eigenfunctions at large distances
described by the so-called {\em scattering length} is responsible for
large scale properties of quantum systems, see Subsection
\ref{scatlen} and \cite{LSSY}.

\subsection{Comparison with the literature}

The present paper can be viewed as a continuation of a series of
related 
papers devoted to 1d Schr\"odinger operators. This series includes
\cite{BuDeGe11_01,DeRi17_01, DeGe21,DeGr22} about holomorphic families of Bessel operators,
\cite{DeRi18_01,DeFaNgRi20_01} about holomorphic families of Whittaker operators 
and \cite{DeGe19_01} devoted to the general theory.

Of course, the literature devoted to  Schr\"odinger
operators on the half-line is vast and goes back several decades.
Here is a selection of classical sources: Edmunds-Evans \cite{EE}, Reed-Simon
vol. II \cite {ReedSimon}, Titchmarsh \cite{Titchmarsh}, 
Coddington-Levinson \cite{CoLe}, Dunford-Schwartz \cite{DS3}, Yafaev
\cite{Yafaev}, Levitan-Sargsjan \cite{LevSar}, Weidmann
  \cite{Weidmann}, Derkach-Malamud \cite{DeMa}.
 See also  Gesztesy-Zinchenko \cite{Gesztesy-Zinchenko}.

Most of this literature is restricted to
real potentials and to self-adjoint realizations. 
The theory
of general closed realizations of Schr\"odinger 1d operators with
complex potentials is actually a relatively straightforward extension of
the real theory and  also  has a long tradition. (However it has a
rather different terminology: e.g. ``self-adjoint extensions of
symmetric operators'' and ``limit point/limit circle case''  replace
``closed realizations of the formal
operator'', and ``trivial/nontrivial boundary space'').
Here the number of sources is
much smaller, but includes some of the classics, such as Titchmarsh \cite{Titchmarsh}, 
 Edmunds-Evans \cite{EE} and Dunford-Schwartz \cite{DS3}.

Most of these sources start from  the 1-dimensional Laplacian on the half-line
 with Dirichlet or Neumann
conditions. This
corresponds to the Bessel operator  $H_{\frac12}$ (Dirichlet) and
$H_{-\frac12}$ (Neumann)  
in the terminology of our paper.
Usually the potential is assumed to be integrable near zero. Note that
this excludes not only the $1/x^2$ potential, but even the $1/x$ potential, and makes the
theory of boundary conditions very straightforward.

Self-adjoint extensions for potentials $1/x^2$ and $1/x$ are
of course also discussed in the literature by many authors, e.g. in \cite{Gesztesy,BullaGesztesy,GTV,AlAn,AnBu,Kovarik}.

There are also many treatments of $d$-dimensional Schr\"odinger
operators. They are closely related to the Bessel operators for
$m=\frac{d}{2}-1$,
 especially in the spherically
symmetric case.
We are convinced that for many readers our analysis of perturbed Bessel operators can
serve as a good introduction to the subject of Schr\"odinger operators
in various dimensions.

Perturbed Bessel operators with complex $m$ were  considered to be an important 
  subject already in the 70's, especially in view of applications to 
  the so-called Regge poles \cite{DeRe65_01}.

There exists large literature about defining  boundary conditions
with the help of the
so-called {\em boundary triplets}, see e.g. \cite{BBMNP}. In order to
define a boundary triplet one needs to select a
 transversal pair of Lagrangian subspaces inside the boundary space. In the
case of perturbed Bessel 
operators this amounts to selecting two complementary 1-dimensional
subspaces,
such as (if possible) those defined by $\Wr(x^{\frac12+m},\cdot,0)$ and $\Wr(x^{\frac12-m},\cdot;0)$. Thus the analysis of
our paper can be treated as a preparation for an application of the
boundary triplets formalism.

The concept of a holomorphic family of closed operators goes back to
\cite{Kato}. 
The usefulness of organizing perturbed Bessel operators
in holomorphic families was recognized by
Kato \cite{Kato,Kato83}. 

The behavior of zero energy eigensolutions near infinity and the
related concept of the scattering length is a standard tool of
contemporary physics, at least in dimension $3$ (sometimes  also
$2$). Mathematical treatment of this concept for all dimensions can be
found in \cite{LSSY}.

Many elements and partial results of our paper can be found in
the literature.
We are not aware, however, of previous work
  about all closed realizations of $L_{m^2}$, their pure point
  spectra and their holomorphic properties under the general (and
  rather weak) assumptions on $Q$ that we consider. In this respect,
  we believe that our results are not far from being sharp, at least
  concerning the behavior of $Q$ near zero. As we see in our paper,
 the full picture is quite complex. Note in particular that the cases $m=0$ and
 $k=0$ are quite subtle,  both near $0$ and $\infty$.
 We have also  never seen a systematic
 analysis involving the boundary conditions given by $u_m^{0[n]}$, see
 \eqref{eq:fami2a}, 
 which shows how to deal with a  perturbation where the most
 straightforward approach fails.

  The direction where our results could be
  somewhat strengthened is the regularity of $Q$. This can be
  done e.g. using the method
  of Shkalikov and Savchuk \cite{Shk1,Shk2}, however it would  introduce an extra layer of technical complication
in our analysis.

\subsection{Notations}
On $L^2]0,\infty[$, the notation $\langle \cdot | \cdot \rangle$ stands for the bilinear form defined by
\begin{equation}\label{bili}
\langle f |g \rangle := \int_0^\infty f(x) g(x) \mathrm{d} x , \quad f,g \in L^2]0,\infty[.
\end{equation}
More generally, we will use the notation 
\begin{equation*}
\langle f |g \rangle = \int_0^\infty f(x) g(x) \mathrm{d} x ,
\end{equation*}
for any measurable functions $f,g$ such that  $fg \in L^1]0,\infty[$.

The {\em transpose} of an operator $A$, that is, the adjoint with respect to
\eqref{bili} will be denoted $A^\#$, as in
\cite{DeGe19_01}.

The Wronskian of two differentiable functions $f,g$ is denoted by
\begin{equation}
\Wr(f,g;x) := f(x) g'(x) - f'(x) g(x), \quad x \in ]0,\infty[.
\label{wronskian}\end{equation}
Moreover,
\begin{equation*}
\Wr( f , g ; 0 ) := \lim_{x\to 0} \Wr( f , g ; x ) , \quad \Wr( f , g ; \infty ) := \lim_{x \to \infty} \Wr( f , g ; x ) ,
\end{equation*}
if these limits exist. If $f,g\in AC^1]0,\infty[$ are two solutions to the equation $(L^\bullet_{m^2} +k^2)u=0$, where $L^\bullet_{m^2}$ stands for $L^0_{m^2}$ or $L_{m^2}$, then their Wronskian is constant and is denoted by $\Wr(f,g)$.

To shorten notations below, for $b,c \in \mathbb{R} \cup \infty$, we use the shorthand 
\begin{equation*}
(m , k ) \in \{ \mathrm{Re}(m) > b , \mathrm{Re}( k ) > c \} ,
\end{equation*}
with the obvious meaning that $(m,k) \in \mathbb{C}^2$ are such that $\mathrm{Re}(m) > b$, $\mathrm{Re}( k ) > c$, and likewise if $\mathrm{Re}( m ) > b$ is replaced by $\mathrm{Re}( m ) \ge b$ and so on.

Let $\Omega\subset\C\times\C$. We will say that a function 
\begin{equation*}
  \Omega\ni(m,k)\mapsto f(m,k)
  \end{equation*}
is {\em regular} on $\Omega$ if it is continuous 
 and 
 for any $m_0,k_0\in\C$ it is analytic 
 on 
 \begin{align*}
   \{k\in\C\ |\ (m_0,k)\in\Omega\}^\circ \ni k&\mapsto f(m_0,k),\\
   \{m\in\C\ |\ (m,k_0)\in\Omega\}^\circ \ni m&\mapsto f(m,k_0), 
 \end{align*}
 where $K^\circ$ denotes the interior of a set $K\subset\C$. Note that if $\Omega$ is open then by Hartog's theorem $f$ is regular if and only if it is analytic.

In several proofs, $C$ will stand for a positive constant depending on the parameters and which may vary from line to line.  Moreover the notation $a \lesssim b$ stands for $a \le C b$ where $C$ is a positive constant depending on the parameters.

If $A$ is an operator, then $\cD(A)$ will denote its domain and $\Ker(A)$ its
kernel (nullspace).

\subsection{Hypotheses}
Recall that throughout the paper, we assume that $Q\in L^1_{\mathrm{loc}}]0,\infty[$. Depending on the results, we will require further integrability conditions near $0$ and/or $\infty$. Our minimal conditions will be
 \begin{equation*}
\begin{split}
Q \in \mathscr{L}^{(0)}_0 := \Big \{  Q\in L^1_{\mathrm{loc}}]0,\infty[ \, \Big | \,  \int_{0}^1 x|Q(x)| \mathrm{d}x < \infty \Big \}&,\quad
    \text{if } m\neq0;\\
Q \in \mathscr{L}^{(0)}_{0,\mathrm{ln}} := \Big \{ Q\in L^1_{\mathrm{loc}}]0,\infty[ \, \Big | \,   \int_{0}^1 x (1+|\mathrm{ln}(x)|) |Q(x)| \mathrm{d}x < \infty \Big \}&,
    \quad \text{if } m=0,
\end{split}
    \end{equation*}
near $0$ and
 \begin{equation*}
 \begin{split}
Q \in \mathscr{L}^{(\infty)}_0 := \Big \{ Q\in L^1_{\mathrm{loc}}]0,\infty[ \, \Big | \,  \int_{1}^\infty |Q(x)| \mathrm{d}x < \infty \Big \},
   \end{split}
    \end{equation*}
near $\infty$. We will sometimes strengthen these conditions to
 \begin{equation*}
\begin{split}
Q \in \mathscr{L}^{(0)}_\varepsilon := \Big \{  Q\in
L^1_{\mathrm{loc}}]0,\infty[ \, \Big | \,  \int_{0}^1
x^{1-\varepsilon}|Q(x)| \mathrm{d}x < \infty \Big \}&;
\\
Q \in \mathscr{L}^{(0)}_{\varepsilon,\mathrm{ln}^\beta} := \Big \{
Q\in L^1_{\mathrm{loc}}]0,\infty[ \, \Big | \,   \int_{0}^1
x^{1-\varepsilon} (1+|\mathrm{ln}(x)|^\beta) |Q(x)| \mathrm{d}x < \infty
\Big \}.&
\end{split}
    \end{equation*}
with $\varepsilon\ge0$, $\beta\geq0$ and
 \begin{equation*}
 \begin{split}
Q \in \mathscr{L}^{(\infty)}_\delta := \Big \{ Q\in 
L^1_{\mathrm{loc}}]0,\infty[ \, \Big | \,  \int_{1}^\infty x^\delta 
|Q(x)| \mathrm{d}x < \infty \Big \},\\
Q \in \mathscr{L}^{(\infty)}_{\delta,\mathrm{ln}} := \Big \{ Q\in 
L^1_{\mathrm{loc}}]0,\infty[ \, \Big | \,  \int_{1}^\infty x^\delta (1+\mathrm{ln}(x))
|Q(x)| \mathrm{d}x < \infty \Big \}.
   \end{split}
    \end{equation*}
    with $\delta\ge0$.

Obviously, if $0 \le \varepsilon < \varepsilon'$, $0 \le \beta < \beta'$, then 
\begin{equation*}
\mathscr{L}^{(0)}_{\varepsilon'} \subset \mathscr{L}^{(0)}_{\varepsilon,\mathrm{ln}^{\beta'}}  \subset \mathscr{L}^{(0)}_{\varepsilon,\mathrm{ln}^{\beta}} \subset \mathscr{L}^{(0)}_\varepsilon.
\end{equation*}
 Likewise, if $0 \le \delta < \delta'$ then $\mathscr{L}^{(\infty)}_{\delta'} \subset \mathscr{L}^{(\infty)}_{\delta,\mathrm{ln}} \subset \mathscr{L}^{(\infty)}_\delta$. Moreover, since $Q\in L^1_{\mathrm{loc}}]0,\infty[$, the integrability conditions on $]0,1[$ are equivalent to the same integrability conditions on $]0,a[$ for any $a>0$ and the integrability conditions on $]1,\infty[$ are equivalent to the same integrability conditions on $]a,\infty[$ for any $a>0$. 

\section{Solutions of the unperturbed eigenequation}\label{sec:unperturbed}

In this section we describe
solutions to the unperturbed eigenequation
\begin{align}\label{eq:a100}
L_{m^2}^0 g = - k^2 g.
\end{align}

\subsection{Bessel equation}\label{subsec:whittaker}

The eigenequation  \eqref{eq:a100} with the eigenvalue $-k^2=-1$ has the
form
\begin{align}
\Big ( -\partial_{z}^2 + \Big ( m^2 - \frac14 \Big ) \frac{1}{z^2}  +1\Big)g=0. 
\label{whi0}
\end{align}
We will call \eqref{whi0} the {\em hyperbolic Bessel equation for dimension
$1$}, or  the {\em hyperbolic 1d Bessel equation}.

Eq. \eqref{whi0} is fully equivalent to the usual modified Bessel
equation, which  corresponds
to dimension $2$,
\begin{align}
\Big ( -\partial_{z}^2 -\frac1z\partial_z+ \frac{m^2}{z^2}  +1\Big)g=0. 
\label{whi00}
\end{align}
We  use the name the {\em hyperbolic 2d Bessel equation} for
\eqref{whi00}. In general, we will prefer to use
\eqref{whi0} as our standard form of the Bessel equation.

In this subsection we briefly describe solutions of the hyperbolic 1d Bessel equation,
following mostly \cite{DeRi18_01,DeFaNgRi20_01}. There are two kinds of standard solutions to the hyperbolic 1d Bessel equation \eqref{whi0}.

The {\em hyperbolic 1d Bessel function} $\mathcal{I}_{m}$ is defined by
\begin{align*}
\cI_m(z)& =
          \sum_{n=0}^\infty\frac{\sqrt\pi\left(\frac{z}{2}\right)^{2n+m+\frac12}}{n!\Gamma(m+n+1)}=\sqrt{\frac{\pi
          z}{2}}
          I_m(z)=\sqrt\pi\Big(\frac{z}2\Big)^{\frac12+m}\mathbf{F}_m\Big(\frac{z^2}{4}\Big) .
\end{align*}
Here $I_m$ is the usual modified Bessel function, which solves the
hyperbolic 2d Bessel equation and $\mathbf{F}_m$ is the appropriately normalized
$(0,1)$-hypergeometric function
\begin{equation*}
  \mathbf{F}_m(w):=\frac{{}_0F_1(m+1;w)}{\Gamma(m+1)}=\sum_{n=0}^\infty\frac{w^n}{n!\Gamma(m+n+1)}.
  \end{equation*}
Note that for $m\in\mathbb{Z}$
\begin{equation}
  \cI_m(z)=\cI_{-m}(z),\qquad \mathbf{F}_{-m}(w)=w^{2m}\mathbf{F}_m(w).
  \label{refle}
\end{equation}
The analytic continuation around $0$ by the angle $\pm\pi$ multiplies
$\cI_m$ by a phase factor, namely
\begin{equation*}
\cI_m(\e^{\pm\i\pi}z)=\e^{\pm\i \pi (m+\frac12)}\cI_m(z).
\end{equation*}

The {\em 1d Macdonald function} $\cK_m$ is defined by
\begin{align*}
  \cK_m(z)&=\frac{\sqrt{z}}{\sqrt{2\pi}}\int_0^\infty\exp
            \left(-\frac{z}{2}(s+s^{-1})\right)s^{-m-1}\d s
=\sqrt{\frac{2 z}{\pi}} K_m(z)\\
  &=\frac{1}{\sin(\pi m)}\big(\cI_{-m}(z)-\cI_{m}(z)\big).
\end{align*}
Here $K_m$ is the usual Macdonald function, which solves the
hyperbolic 2d Bessel equation.

For any $m\in\C$ we have $\cK_m(z)=\cK_{-m}(z)$.

For any fixed $m \in \mathbb{C}$, the maps $z\mapsto\cI_{m}(z)$ and $z \mapsto \cK_{m}(z)$ are analytic except for a branch point at $z=0$.
Thus the natural domain for these solutions is the Riemann surface of
the logarithm.
One can parametrize this surface by $|z|\in]0,\infty[$ and
    $\arg(z)\in\mathbb{R}$.  It is often convenient to 
 restrict the domain  to $ \mathbb{C} \setminus ] - \infty , 0 ] $, that is, to $|\mathrm{arg}(z)|<\pi$. One can also include
    two copies of    $ ] - \infty , 0 ] $, from above and from below, that is $\mathrm{arg}(z)=\pm\pi$. For any fixed $z$ in this domain, the maps $m \mapsto \cI_{m}(z)$ and $m\mapsto \cK_{m}(z)$ are analytic on $\mathbb{C}$.

    The functions
    $z \mapsto \cK_{m}(\e^{\pm\i\pi}z)$, obtained from $\cK_m$ by
analytic continuation,
    are also solutions of 
         \eqref{whi0}. 
Typically, it is natural to consider the pairs of functions 
\begin{equation*}
z\mapsto\cK_{m}(z), \quad z \mapsto \cK_{m}(\e^{\pm\i\pi}z),
\end{equation*}
on $0\leq\mp\mathrm{arg}(k)\leq\pi$. Both pairs are linearly independent.
In particular, $  \cI_{m}(z) $ can be expressed in terms of these functions:
\begin{align*}
  \cK_m(\e^{\pm\i\pi m}z)&=
                         \frac{\mp\i}{\sin(\pi m)}\left(\e^{\pm\i\pi 
                         m}\cI_m(z)-
                            \e^{\mp\i\pi 
                         m}\cI_{-m}(z)\right),\\
  \cI_m(z)&=
            \frac12\left(\cK_m(\e^{\pm\i\pi}z)\mp\i\e^{\mp\i\pi m}\cK_m(z)\right).  
  \end{align*}

Here are the asymptotics of the solutions $\mathcal{I}_m$ and $\mathcal{K}_m$ near $0$:
\begin{small}
\begin{align}
  \cI_m(z)&=\frac{\sqrt\pi}{\Gamma(m+1)}\Big(\frac{z}{2}\Big)^{m+\frac12}+\mathcal{O}(|z|^{\frac52+\Re(m)}),&m\neq-1,-2,\dots;\\\label{asym-k}
\cK_m(z)&=
\frac{\Gamma(m)}{\sqrt\pi}\Big(\frac{z}{2}\Big)^{-m+\frac12}+\mathcal{O}(|z|^{\frac52-\Re(m)}),&
                                                                             \Re(m)>1;\\
  \cK_m(z)&=
\frac{\Gamma(m)}{\sqrt\pi}\Big(\frac{z}{2}\Big)^{-m+\frac12}
            -\frac{\Gamma(-m)}{\sqrt\pi}\Big(\frac{z}{2}\Big)^{m+\frac12}
            +\mathcal{O}(|z|^{\frac52-\Re(m)}),& |\Re(m)|<1, \ m\neq0;\\
  \cK_0(z)&=-\frac{\sqrt{2z}}{\sqrt\pi}\Big(\mathrm{ln}\Big(\frac{z}{2}\Big)+\gamma\Big)+\mathcal{O}\big(|z|^{\frac52}\mathrm{ln}|z|\big),&
                                                                             m=0;\\
                                                                                                                       \cK_1(z)&=\frac{1}{\sqrt\pi}
                                                                                                                                 \Big(\frac{z}{2}\Big)^{-\frac12}+\mathcal{O}(|z|^{\frac32}\mathrm{ln}|z|),&m=\pm1.
\end{align}
\end{small}
Recall that $\gamma$ denotes Euler's constant.

Using 
 the integral representation of
 $ \mathcal{K}_{m}(z)$ one can prove the following asymptotics at infinity: for any $\epsilon>0$,
\begin{align}\label{eq:equivKbm_infty}
  \mathcal{K}_{m}(z)& = \e^{-z}\big ( 1 + \mathcal{O}( z^{-1} ) \big ) , \quad|\mathrm{arg}(z)|\leq\frac32\pi-\epsilon,\quad |z| \to \infty .
\end{align}
Note that the sector $|\mathrm{arg}(z)|<\frac32\pi$ is maximal for the estimate \eqref{eq:equivKbm_infty}. Beyond this sector the estimate no longer holds.

The following estimates near zero, say, for $|z|\leq1$, follow from the
series expansions:
\begin{align*}
  |\cI_{m}(z)|&\lesssim |z|^{\frac12+\Re(m)};\\
  |\cK_{m}(z)|&\lesssim |z|^{\frac12-|\Re(m)|},\quad m\neq0;\\
      |\cK_{0}(z)|&\lesssim |z|^{\frac12}(1+|\mathrm{ln}(z)|),\quad m=0.
\end{align*}

We also have the following estimates near $\infty$, say, for $|z|\geq1$ (and any $\epsilon>0$):
 \begin{align*}
   \left|  \mathcal{K}_{m}(z)\right|& \lesssim \e^{-\Re(z)}  \quad|\mathrm{arg}(z)|\leq \frac32\pi-\epsilon;\\
   |\cK_{m}(\e^{\pm\i\pi}z)|&\lesssim 
\e^{\Re(z)}
,\quad    |\mathrm{arg}(z)\mp\pi|\leq\frac32\pi-\epsilon;\\
\left|  \mathcal{I}_{m}(z)\right|& \lesssim \e^{\Re(z)}
+\e^{-\Re(z)}.
\end{align*}

Here are global estimates:
\begin{small}
\begin{align}
  |\cK_{m}(z)|&\lesssim
  \min(1,|z|)^{\frac12-|\Re(m)|}\e^{-\Re(z)},\quad
    |\mathrm{arg}(z)|\leq\frac32\pi-\epsilon,\quad m\neq0;\\
  |\cK_{m}(\e^{\pm\i\pi}z)|&\lesssim
  \min(1,|z|)^{\frac12-|\Re(m)|}\e^{\Re(z)},
\quad
   |\mathrm{arg}(z)\mp\pi|\leq\frac32\pi-\epsilon,\quad m\neq0;
    \\
  |\cK_{0}(z)|&\lesssim \min(1,|z|)^{\frac12}(1+|\mathrm{ln}\min(1,|z|)|)\e^{-\Re(z)},\quad
    |\mathrm{arg}(z)|\leq\frac32\pi-\epsilon,\quad m=0;\\
  |\cK_{0}(\e^{\pm\i\pi}z)|&\lesssim \min(1,|z|)^{\frac12}(1+|\mathrm{ln}\min(1,|z|)|)\e^{\Re(z)},
\quad
|\mathrm{arg}(z)\mp\pi|\leq\frac32\pi-\epsilon,\quad m=0;
    \\
    |  \cI_{m}(z)|&
    \lesssim \min(1,|z|)^{\frac12+\Re(m)}\big(\e^{-\Re(z)}
    +\e^{\Re(z)}\big) \label{estim_Im}
  .
\end{align}
\end{small}

 Here are the Wronskians of 
various solutions of the hyperbolic 1d Bessel equation: 
\begin{align*}
\Wr(\cI_{m},\cI_{-m})& = - \sin(\pi m) ,\\
\Wr(\cK_{m},\cI_{m})&  = 1 ,\\
\Wr\big(\cK_{m},\cK_{m}(\e^{\pm\i\pi}\cdot)\big) & =
2, \\ 
\Wr\big(\cI_{m},\cK_{m}(\e^{\pm\i\pi}\cdot)\big) &=
                                                          \mp\i\e^{\pm\i\pi 
                                                          m}. 
\end{align*}

\subsection{Equation \texorpdfstring{$L_{m^2}^0g=-k^2g$}{Lg}}\label{subsec:eq_whit}
Let us now analyze the eigenvalue equation for $L_{m^2}^{0}$ and
an arbitrary  eigenvalue $-k^2\neq0$:
\begin{align}
\Big ( -\partial_{x}^2 + \Big ( m^2 - \frac14 \Big ) \frac{1}{x^2}  \Big ) g = -k^2 g.
\label{whi0-}
\end{align}

A direct computation shows that
for $k\neq0$
\eqref{whi0-}
is solved by the following functions
\begin{align}
    u_m^0(x,k)&:=\sqrt{\frac{2}{\pi k}}\Big(\frac2{k}\Big)^{m}\cI_m(kx),\label{uu1}\\
   v_m^0(x,k)&:=\sqrt{\frac\pi{2k}}\Big(\frac{k}2\Big)^{m}\cK_m(kx).\label{uu2}
\end{align}
For $m=0$ we also introduce the solution
\begin{align}
  \label{pzero}
  p_0^0(x,k)&:=-\sqrt{\frac\pi{2k}}\cK_0(kx)-\Big(\mathrm{ln}\Big(\frac{k}{2}\Big)+\gamma\Big)\sqrt{\frac2{\pi
              k}}\cI_0(kx) \\  &=- v_0^0(x,k)-\Big(\mathrm{ln}\Big(\frac{k}{2}\Big)+\gamma\Big)u_0^0(x,k). \notag
  \end{align}
Here are their Wronskians: 
\begin{align}                                \Wr \big(
  u_m^0(\cdot,k) , u_{-m}^0(\cdot,k)  \big)&=-\frac{2\sin(\pi m)}{\pi}, \label{eq:wronsk_m-m}\\
                                 \Wr \big( v_m^0(\cdot,k) , u_m^0(\cdot,k)  \big)&=1, \quad
                       \Wr\big( u_0^0(\cdot,k),p_0^0(\cdot,k)\big)=1. 
                                                         \end{align}
We define these functions also for $k=0$:
\begin{align}
\label{zero1}   u_m^0(x,0)&=\frac{x^{\frac12+m}}{\Gamma(m+1)},\\
\label{zero2}   v_m^0(x,0)&=\frac{\Gamma(m)x^{\frac12-m}}{2}, \quad \Re(m)\ge0 , \quad m \neq 0 ,\\
\label{zero3}
 p_0^0(x,0)&=x^{\frac12}\mathrm{ln}(x).
\end{align}
Clearly \eqref{zero1}, \eqref{zero2} and \eqref{zero3} are annihilated 
by $L_{m^2}^0$. 
Note that for any fixed $x>0$, $ u_m^0(x,k)$ and $p_0^0(x,k)$ are continuous in
$k$ down to $k=0$. If $\Re(m)\ge0$, $m\neq 0$, the same is true for 
$ v_m^0(x,k)$.

\begin{proposition}\label{uua0} Let $x>0$. Then
  \begin{enumerate}[label=(\roman*)]
    \item The function $(m,k)\mapsto u_m^0(x,k)$ is analytic on
      $\C\times\C$.
      \item The function $(m,k)\mapsto v_m^0(x,k)$ is
  regular on
  \begin{equation*}
\C {\times} \{ \Re(k)\ge0 \}\,\backslash\, 
   \big( \{\Re(m)<0 \} {\times} \{k=0\}\cup\{m=0\}{\times}\{k=0\} \big). 
  \end{equation*}
\item
  The function $k\mapsto p_0^0(x,k)$ is regular on
  \begin{equation*}
   \{ \Re(k)\ge 0\}.
  \end{equation*}
  \end{enumerate}
\end{proposition}
    \begin{proof}
    We can rewrite the definitions \eqref{uu1} and \eqref{uu2} as
    \begin{align}
    u_m^0(x,k)&:=x^{\frac12+m}\mathbf{F}_m\Big(\frac{k^2x^2}{4}\Big),\label{uua1}\\
       v_m^0(x,k)&:=\frac{\pi x^{\frac12-m}}{2\sin(\pi m)}
                  \Bigg(\mathbf{F}_{-m}\Big(\frac{k^2x^2}{4}\Big)-\Big(\frac{k^2x^2}{4}\Big)^m \mathbf{F}_{m}\Big(\frac{k^2x^2}{4}\Big)\Bigg)
                  .\label{uua2}
\end{align}
From \eqref{uua1} the reguarity
of $u_m^0(x,k)$ is obvious.
\eqref{uua2} directly shows the regularity of
$ v_m^0(x,k)$ on the considered domain at $m\not\in\mathbb{Z}$. To see the regularity of
$ v_m^0(x,k)$ at $m\in\mathbb{Z}$ it suffices to use
\eqref{refle} and the de l'Hopital rule.

(iii) follows from (i) and (ii) at $k\neq0$. At $k=0$, as mentioned above, a direct computation shows that $p_0^0(x,k)\to p_0^0(x,0)$, as $k\to0$.
 \end{proof}

\begin{remark}     Note that $v_m^0(x,k)$ can be continued across the cut
     $\arg(k)=\pm\frac\pi2$, so that $k=0$ becomes its branch point. The
     restriction to $\{\Re(k)\ge0\}=\{|\arg(k)|\leq\frac\pi2\}$ is convenient in view of
     applications. It is however interesting to note that for
     $m\in\mathbb{Z}+\frac12$, after this continuation, one obtains an
     univalent function of $k$.
 \end{remark} 

For $\Re(k)\ge0$, we have the following estimates:
\begin{align}\label{pwo1}
  | u_m^0(x,k)|&\lesssim  \min(|k|^{-1},x)^{\frac12+\Re(m)}\e^{\Re(k)x};\\\label{pwo2}
  | v_m^0(x,k)|&
    \lesssim \min(|k|^{-1},x)^{\frac12-|\Re(m)|} \e^{-\Re(k)x},\quad
                       \Re(m)\geq0,\quad m\neq0;\\\label{pwo3}
  | v_0^0(x,k)|&\lesssim
                       \min(|k|^{-1},x)^{\frac12}(1-\mathrm{ln}\min(1,|k|x))\e^{-\Re(k)x},\quad
                       m=0
                       ;\\\label{pwo4}
  |p_0^0(x,k)|&\lesssim \min(|k|^{-1},x)^{\frac12}(1-\mathrm{ln}\min(|k|^{-1},x))\e^{\Re(k)x},\quad m=0.
\end{align}

The reason for complicated prefactors in 
\eqref{uu1} and  \eqref{uu2} is the good behavior near $k=0$. When we are
interested in $k$ large, we usually prefer to replace $v_m^0$ with a
differently normalized solution
\begin{equation}\label{eq:w_m^0}
  w_m^0(x,k):=\sqrt{\frac{2k}{\pi}}\Big(\frac{2}{k}\Big)^mv_m^0(x,k)=
\cK_m(kx)
\end{equation}
behaving as $\e^{-kx}$  for $x\to\infty$.

\subsection{Canonical bisolution}

 In this and the following subsection we introduce a few
   integral kernels naturally associated with $L_{m^2}^0+k^2$.  The
   corresponding operators are not always bounded on
   $L^2]0,\infty[$, however, they 
   will play an important role in our paper.

Let $h_1^0, h_2^0$ be any pair of solutions to  $L_{m^2}^0 f = - k^2 f$
satisfying $\Wr(h_1^0, h_2^0)\neq0$. Then, following \cite{DeGe19_01}, we introduce the operator 
\begin{equation*}
 G^0_\leftrightarrow =G^0_{m^2,\leftrightarrow}(-k^2)
  :=\frac{1}{\Wr(h_1^0,h_2^0)}\Big(h^0_1 \langle h^0_2 | \cdot \rangle  - h^0_2\langle h^0_1 | \cdot \rangle \Big) .
\end{equation*}
It has the kernel
\begin{equation*}
G^0_\leftrightarrow(x,y) = G^0_{m^2,\leftrightarrow}(-k^2;x,y) =\frac{1}{\Wr(h_1^0,h_2^0)}\big ( h^0_1(x) h^0_2( y ) - h^0_2( x ) h^0_1( y ) \big ).
\end{equation*}
Note that $G^0_\leftrightarrow$ does not depend on the choice
of the functions  $h_1^0, h_2^0$, which justifies the adjective {\em canonical}.
The operator $G^0_\leftrightarrow$ is called the {\em canonical bisolution of} $L^0_{m^2} + k^2$. It satisfies
\[( L^0_{m^2} + k^2 ) G^0_\leftrightarrow = G^0_\leftrightarrow (
  L^0_{m^2} + k^2 ) = 0,\]
which justifies calling it  {\em bisolution}.
In particular, since the Wronskian
of $v_m^0$ and $u_m^0$ is $1$, one has
\begin{align*}
  G^0_{m^2,\leftrightarrow}(-k^2;x,y)
  &=  v_m^0(x,k)u_m^0( y,k ) -
                          u_m^0( x ,k) v_m^0( y,k ).
\end{align*}
For $m=0$ we also have 
\begin{align*}
  G^0_{0,\leftrightarrow}(-k^2;x,y) 
  &=- p_0^0(x,k)u_0^0( y,k ) +
                          u_0^0( x ,k) p_0^0( y,k ).
\end{align*}
 The canonical bisolution is defined also for $k=0$:
 \begin{align*}
   G_{m^2,\leftrightarrow}^0(0;x,y)&=\frac{1}{2m}\big(x^{\frac12-m}y^{\frac12+m}-
                           x^{\frac12+m}y^{\frac12-m}\big),\quad m\neq0;\\
      G_{0,\leftrightarrow}^0(0;x,y)&=
                           x^{\frac12}y^{\frac12}\big(\mathrm{ln}(y)-\mathrm{ln}(x)\big),\quad m=0.
   \end{align*}

\begin{proposition}
  Let $x,y>0$. Then  the map 
\begin{equation*}
( m^2 , k^2 ) \mapsto G^0_{m^2,\leftrightarrow}(-k^2;x,y) ,
\end{equation*}
is analytic on $\C\times\C$. \end{proposition}

\begin{proof}
We can write
\begin{align*}\notag
  G_\leftrightarrow^0(x,y) =\frac{\pi\sqrt{xy}}{\sin(\pi m)}&  x^{m}y^{-m}
                        \Bigg(  
\mathbf{F}_m\Big(\frac{k^2x^2}{4}\Big) 
                            \mathbf{F}_{-m}\Big(\frac{k^2y^2}{4}\Big)\\-&
                              x^{-2m}y^{2m}\mathbf{F}_{-m}\Big(\frac{k^2x^2}{4}\Big) \mathbf{F}_{m}\Big(\frac{k^2y^2}{4}\Big)\Bigg).
\end{align*}
For $m\not\in\mathbb{Z}$, the analyticity in $m,k^2$ is obvious from this
expression. For
 $m\in\mathbb{Z}$, we apply first  the de l'Hopital
 rule. Then  we obtain a function analytic in $m,k^2$.

 $G_\leftrightarrow^0$ is invariant with respect to the change $m\to-m$. Together
 with the analyticity in $m$, it implies the analyticity in $m^2$. 
 \end{proof}

\subsection{Green's operators}

We will need several kinds of Green's operators.
 The \emph{forward Green's
  operator} $G^0_{m^2,\rightarrow}(-k^2)$ and the \emph{backward Green's operator}
$G^0_{m^2,\leftarrow}(-k^2)$
are defined by
\begin{align*}
G^0_\rightarrow(x,y)=G^0_{m^2,\rightarrow}(-k^2;x,y)& : = 
\theta(x-y) G^0_{m^2,\leftrightarrow}(-k^2; x , y ),\\
G^0_\leftarrow(x,y) =G^0_{m^2,\leftarrow}(-k^2;x,y) & := 
  -\theta(y-x) G^0_{m^2,\leftrightarrow}( -k^2;x , y ).
\end{align*}
Here $\theta$ is the Heaviside function:
\begin{align*}
\theta(x)
&:=\left \{ 
\begin{array}{ll}0
 & \text{if } \, x<0 , \vspace{0,1cm} \\
1 & \text{if } \, x\geq0 .
\end{array}
\right.
\end{align*}
Many properties of forward and backward Green's operators can be
directly deduced
from those of the canonical bisolution.
\begin{proposition}
Let $x,y>0$.    Then the  maps
    \begin{align*}
      ( m^2 , k^2 ) \,\mapsto\, &G^0_{m^2,\to}(-k^2;x,y) ,\quad G^0_{m^2,\leftarrow}(-k^2;x,y) 
\end{align*}
are analytic on $\C\times\C$. \end{proposition}

$L_{m^2}^0$ possesses various Green's operators defined by
imposing boundary conditions at $0$ and $\infty$, which can be called
{\em two-sided}. Among them  we
should distinguish $G_{m,\bowtie}^0(-k^2) $
given for $\Re(k)\geq0$, $k\neq0$, by its
integral kernel
\begin{align*}
  G_{\bowtie}^0 (x,y) = \, &G_{m,\bowtie}^0 (-k^2;x,y) \\
   :=\, & \theta(x-y) v_m^0(x,k)u_m^0( y,k )
    +\theta(y-x)u_m^0( x ,k) v_m^0( y ,k). 
    \end{align*}
  We will call it
the {\em two-sided Green's operator with pure boundary conditions},
often abusing the terminology and shortening the name to just the {\em two-sided Green's operator}.
Note that it depends on $m$ and not on $m^2$.
Note also that at the moment we do not insist on the
  conditions $\Re(m)>-1$ and $\Re(k)>0$, which will be needed to make it a bounded operator.

             Note the connection between the forward and two-sided Green's 
operators: 
\begin{align}\label{conne}
  G_{m,\bowtie}^0(-k^2;x,y)&= u_m^0(x,k) v_m^0(y,k)+ G_{m^2,\to}^0(-k^2;x,y),\\
&=v_m^0(x,k) u_m^0(y,k)+ G_{m^2,\leftarrow}^0(-k^2;x,y). 
  \end{align}

For $m\neq0$, $  G_{m,\bowtie}^0(-k^2)$ can be also defined for $k=0$, when its integral
kernel is
\begin{align*}
  G_{m,\bowtie}^0(0;x,y)&:=\frac{1}{2m}\big(x^{\frac12-m}y^{\frac12+m}\theta(x-y) 
                      +x^{\frac12+m}y^{\frac12-m}\theta(y-x)\big)
                      .
  \end{align*}

         For $m=0$ and   $k=0$  Green's operator $G_{\bowtie}^0$ is not well 
           defined. This motivates us to introduce another kind of
           Green's operator at $m=0$, which will be called
 {\em  Green's operator logarithmic near zero}: 
\begin{equation*}
 G_\diamond^0 (x,y) =G_{0,\diamond}^0(-k^2;x,y):=
 - u_0^0(x,k)p_0^0(y,k)
  \theta(x-y) 
  -  p_0^0(x,k) u_0^0(y,k)
  \theta(y-x).
   \end{equation*}
It has a limit at $k=0$:
\begin{equation*}
G_{0,\diamond}^0 (0;x,y):=-x^{\frac12}y^{\frac12}\big(\mathrm{ln}(x)\theta(x-y) 
  +\mathrm{ln}(y)\theta(y-x)\big).
   \end{equation*}
Observe that $ G_\diamond^{0}$ and $ G_{\bowtie}^{0}$ differ by a
term that diverges as $k\to0$:
\begin{equation*}
  G_{0,\diamond}^0(-k^2;x,y)= G_{0,\bowtie}^{0}(-k^2;x,y)
+  \Big(\mathrm{ln}\Big(\frac{k}2\Big) +\gamma\Big) u_0^0(x,k)  u_0^0(y,k).
  \end{equation*}

For $m=0$ it is also natural to introduce {\em Green's operator
  logarithmic near infinity}:
\begin{equation*}
 G_{0,\triangledown}^0 (x,y) =G_{0,\triangledown}^0(-k^2;x,y):=
    p_0^0(x,k) u_0^0(y,k)
  \theta(x-y)  + u_0^0(x,k)p_0^0(y,k) 
  \theta(y-x) .
   \end{equation*}
It also has a limit at $k=0$:
\begin{equation*}
G_{0,\triangledown}^0 (0;x,y):=x^{\frac12}y^{\frac12}\big(\mathrm{ln}(y)\theta(x-y) 
  +\mathrm{ln}(x)\theta(y-x)\big).
\end{equation*}
One could compare
$ G_{0,\triangledown}^{0}$ and $ G_{0,\leftarrow}^{0}$:
\begin{equation*}
  G_{0,\triangledown}^0(-k^2;x,y)= G_{0,\leftarrow}^{0}(-k^2;x,y)
+ p_0^0(x,k)  u_0^0(y,k).
  \end{equation*}

  \begin{proposition}\label{uua5} 
  Let $x,y>0$.
    \begin{enumerate}[label=(\roman*)]
    \item
The function $(m,k)\mapsto G_{m,\bowtie}^0(-k^2;x,y)$
 is regular on 
  \begin{equation*}
    \C\times\{ \Re(k)\ge0\}\,\backslash\, 
    \{\Re(m)\leq0\}{\times}\{k=0\}. 
  \end{equation*}
\item
The function $k\mapsto G_{0,\diamond}^0(-k^2;x,y)$ is
  regular on 
  $
   \{ \Re(k)\ge0\}
$.
  \end{enumerate}
\end{proposition}
    \begin{proof}
This is a direct consequence of Proposition \ref{uua0}. 
\end{proof}

Let $a>0$.
If $G_\bullet^0$ is one of Green's operators, then we introduce the
corresponding Green's operator {\em compressed to the interval
  $]0,a[$}
by\begin{equation}\label{bowtie}
  G_\bullet^{0(a)}(x,y):=
  G_\bullet^{0}(x,y)\theta(a-x)\theta(a-y). 
  \end{equation}

    For
$  m\in\C,\kappa\in\C\cup\{\infty\} $ and $\nu\in\C\cup\{\infty\}$
one can also introduce Green's operators with mixed boundary
conditions
\begin{align}\label{mixed1}
  G_{m,\kappa}^0(-k^2;x,y)&:=
\frac{\frac{1}{\Gamma(-m)}  (k/2)^{-m}G_{m,\bowtie}^0(-k^2)-\frac{\kappa}{\Gamma(m)}(k/2)^mG_{-m,\bowtie}^0(-k^2)}
                          {\frac{1}{\Gamma(-m)}(k/2)^{-m}-\frac{\kappa}{\Gamma(m)}(k/2)^m},\quad
                            m\neq0;\\
    G_{0,\kappa}^0(-k^2;x,y)&:=G_{0,\bowtie}^0(-k^2;x,y);\\
  \label{mixed2}
  G_0^{0,\nu}(-k^2;x,y)&:=\frac{\big(\nu-\gamma-\ln(k/2)\big)G_{0,\bowtie}^0(-k^2)+G_{0,\bowtie}^{0\prime}(-k^2)}
{\nu-\gamma-\ln(k/2)};
\end{align}
where $G_{0,\bowtie}^{0\prime}(-k^2)$ denotes
$\partial_mG_{m,\bowtie}^0(-k^2)\big|_{m=0}$.
Eq. \eqref{mixed1} is (6.3) of \cite{DeRi17_01} (generalized to
$m\in\C$). 
Eq. \eqref{mixed2}  follows from \eqref{mixed1} by the
de l'Hopital method if we set $\kappa=\frac{\nu m-1}{\nu m+1}$ as in
Remark 2.5   of \cite{DeRi17_01}.
Note that \eqref{mixed2}  is consistent with (7.1)  of \cite{DeRi17_01}.

  \section{Solutions of the perturbed eigenequation 
  with prescribed behavior near origin}\label{section:solutions}

In this section we construct solutions  to the equation 
\begin{align}\label{eq:a1}
L_{m^2} g = - k^2 g , 
\end{align}
and study their properties.
     We will try to find solutions whose behavior near origin is 
     similar to solutions of the unperturbed equation
\begin{equation}
  L_{m^2}^0g^0=-k^2g^0.\label{eq:a1x0}
\end{equation}
To shorten notations below, we will often write
\begin{equation*}
u^0(x)=u^0_m(x)=u^0_m(x,k), \qquad v^0(x)=v^0_m(x)=v^0_m(x,k),
\end{equation*}
where $u^0_m(x,\cdot)$, $v^0_m(x,\cdot)$ are the solutions of \eqref{eq:a1x0} introduced in \eqref{uu1} and \eqref{uu2}. Recall that the space of solutions 
in $AC^1]0,\infty[$ to \eqref{eq:a1}, respectively to \eqref{eq:a1x0}  is denoted 
     ${\Ker(L_{m^2}+k^2)}$, respectively    $\Ker(L_{m^2}^0+k^2)$.

\subsection{Weights}

One of our main tools  will be various weighted
$L^\infty$ spaces. Let us introduce notation that we will 
use to denote such spaces.

Let $]a,b[ \, \subset \, ]0,\infty[$. For any positive measurable function $\phi$ on $]a,b[$, we define the following  Banach space of (equivalence classes of) measurable functions on $]a,b[$:
\begin{align}
L^\infty(]a,b[,\phi)&:=
     \left\{f : ]a,b[ \to \mathbb{C} \ \mid \ 
   \left\|\frac{f}{\phi}\right\|_\infty :=
  \underset{x\in]a,b[}{\textrm{ess\,sup}} \left|\frac{f(x)}{\phi(x)}\right|
  <\infty\right\}, \label{eq:defLinfty}\\
  L_0^\infty(]0,b[,\phi)&:=
  \left\{f\in L^\infty(]0,b[,\phi)\ \mid \
                          \lim_{x\to0}\frac{f(x)}{\phi(x)}=0\right\}, \label{eq:defLinfty1}\\
  L_\infty^\infty(]a,\infty[,\phi)&:=
  \left\{f\in L^\infty(]a,\infty[,\phi)\ \mid \
  \lim_{x\to\infty}\frac{f(x)}{\phi(x)}=0\right\}. \label{eq:defLinfty12}
\end{align}

We will use the following convention. Suppose that the operator $\one + G_{\bullet}^{0}Q$ is invertible on $L^\infty(]0,a[,\phi)$ for some $a>0$ and some positive measurable function $\phi$ on $]0,a[$, where $G^0_\bullet$ is a Green's operator. If $f :]0,\infty[\to\mathbb{C}$ is such that its restriction to $]0,a[$ belongs to $L^\infty(]0,a[,\phi)$, then $(\one + G_{\bullet}^{0}Q)^{-1}f$ should be understood as $(\one + G_{\bullet}^{0}Q)^{-1}$ applied to the restriction of $f$ on $]0,a[$. Clearly, if in addition $f\in\cN(L_{m^2}^0+k^2)$, then $(\one + G_{\bullet}^{0}Q)^{-1}f$ is a solution to \eqref{eq:a1} on $]0,a[$. The unique solution on $]0,\infty[$ which coincides with $(\one + G_{\bullet}^{0}Q)^{-1}f$ on $]0,a[$ will be denoted by the same symbol.

In order to make the notation more compact, we
introduce the following 
$k$-dependent weights on $]0,\infty[$:
\begin{align}\label{shorthand}
  \mu_k (x)&: = \min(|k|^{-1},x), 
          \quad \eta_{\pm k}(x):=\e^{\pm\Re(k)x},\\
  \quad 
  \lambda_k (x)&
              :=1-\mathrm{ln}\big(|k|\mu_k (x)\big). \label{shorthand2}
\end{align}
Note that for $k=0$ we have $\mu_k (x)=x$ and $\lambda_k $ is
ill defined. { For $x>|k|^{-1}$, we have $\mu_k (x)=|k|^{-1}$ and $\lambda_k (x)=1$.

With these shorthands we can concisely rewrite our basic estimates on
unperturbed eigenfunctions \eqref{pwo1}--\eqref{pwo4}: 
\begin{align}\label{pio1}
  | u_m^0(x,k)|&\lesssim  \mu_k (x)^{\frac12+\Re(m)}\eta_k
                       (x);\\
  \label{pio2}
    | v_m^0(x,k)|&
    \lesssim \mu_k (x)^{\frac12-\Re(m)} \eta_{-k}(x),\quad
                         \Re(m)\geq0,\quad m\neq0;\\
  \label{pio3}
  | v_0^0(x,k)|&\lesssim 
                       \mu_k (x)^{\frac12}\lambda_k (x)\eta_{-k}(x),\quad 
                       m=0,\quad k\neq0;\\
\label{pio4}                         |p_0^0(x,k)|&\lesssim \mu_k (x)^{\frac12}\big(1+|\mathrm{ln}\,\mu_k (x)|\big)\eta_k(x),\quad m=0.
                       \end{align}
Note that estimates \eqref{pio1}, \eqref{pio2}, \eqref{pio3} and
\eqref{pio4} are
uniform in $x\in]0,\infty[$ and $\Re(k)\ge0$.

\subsection{The forward Green's operator}

In this subsection we study the boundedness of the operator $G_\rightarrow^0Q$ between
suitable weighted $L^\infty$-spaces. The forward Green's operator is insensitive to the change of the sign 
at $m$. Therefore, without limiting generality, we can assume that
$\Re(m)\geq0$.

The first lemma is devoted to global properties of
$G_\rightarrow^0Q$ on the whole $]0,\infty[$. Note that, if $\varepsilon\ge0$ and $k\neq0$, the condition \eqref{bound1} is equivalent to $Q \in \mathscr{L}^{(0)}_\varepsilon\cap\mathscr{L}^{(\infty)}_0$, while \eqref{bound2} is equivalent to $Q \in \mathscr{L}^{(0)}_{\varepsilon,\mathrm{ln}^\beta}\cap\mathscr{L}^{(\infty)}_0$.
   \begin{lemma}\label{lemma1}
     Let
     $\mathrm{Re}(k)\ge0$ and $Q\in L^1_\mathrm{loc}]0,\infty[$.
     \begin{enumerate}[label=(\roman*)]
     \item
       Let $\Re(m)\geq0$, $m\neq0$
       and
       $ \varepsilon_1+\varepsilon\geq\frac12+\Re(m)$.
       Suppose that
       \begin{equation}
         \int_0^\infty\mu_k (y)^{1-\varepsilon}|Q(y)|\d y<\infty.
\label{bound1}         \end{equation}
Then  
              \begin{equation*}
         G_\rightarrow^0Q:
       L^\infty\big(]0,\infty[,\mu_k ^{\varepsilon_1}\eta_k \big)\to
     L_0^\infty\big(]0,\infty[,\mu_k ^{\varepsilon_1+\varepsilon}\eta_k \big)
       \end{equation*}
       is bounded by $C\times$\eqref{bound1} uniformly in $k$.
     \item
       Let $m=0$, $k\neq0$,
       $1\leq\beta-\alpha$ and
 $ \varepsilon_1+\varepsilon\geq\frac12$.  Suppose that
       \begin{equation}
         \int_0^\infty \mu_k (y)^{1-\varepsilon}\lambda_k (y)^\beta|Q(y)|\d y<\infty.
\label{bound2}         \end{equation}
Then  
              \begin{equation*}
         G_\rightarrow^0Q:
       L^\infty\big(]0,\infty[,\mu_k ^{\varepsilon_1}\lambda_k ^\alpha\eta_k \big)\to
     L_0^\infty\big(]0,\infty[,\mu_k ^{\varepsilon_1+\varepsilon}\lambda_k ^{\alpha+1-\beta}\eta_k \big)
       \end{equation*}
       is bounded by $C\times$\eqref{bound2} uniformly in $k$.
     \end{enumerate}
   \end{lemma}
    
     \begin{proof} 
For $m\neq0$ we use
     \begin{align*}
       ( G_\rightarrow^0Qf)(x)=
        v^0(x)\int_0^x  u^0(y)Q(y)f(y)\mathrm{d} y-
  u^0(x)\int_0^x  v^0(y)Q(y)f(y)\mathrm{d} y.
     \end{align*}
By \eqref{pio1}--\eqref{pio2}, we have
     \begin{align*}\notag &
       \left| u^0(x)\int_0^x v^0(y)Q(y)f(y)\mathrm{d} y\right| \\&\lesssim \mu_k (x)^{\frac12+\Re(m)}\eta_k (x) 
       \int_0^x\mu_k (y)^{\frac12-\Re(m)+\varepsilon_1}\eta_{-k}(y)|Q(y)|\eta_k (y)\mathrm{d}
                                                                          y\left\| \frac{f}{\mu_k^{\varepsilon_1}\eta_k} \right\|_{\infty},
\end{align*}
and
\begin{align*}                                                                          
     \notag &
       \left| v^0(x)\int_0^x u^0(y)Q(y)f(y)\mathrm{d} y\right|\\ &\lesssim \mu_k (x)^{\frac12-\Re(m)}\eta_{-k}(x) 
       \int_0^x\mu_k (y)^{\frac12+\Re(m)+\varepsilon_1}\eta_k (y)|Q(y)|\eta_k (y)\mathrm{d} y\left\| \frac{f}{\mu_k^{\varepsilon_1}\eta_k} \right\|_{\infty}.
     \end{align*}
Using the fact that $y\mapsto\mu_k (y)$ and $y\mapsto\eta_k(y)$ are increasing together with 
       $     -\frac12\mp\Re(m)+\varepsilon   + \varepsilon_1\geq0$,
       we estimate both expressions by
\begin{align*}
       &\lesssim \mu_k (x)^{\varepsilon_1+\varepsilon}\eta_k (x)
       \int_0^x\mu_k (y)^{1-\varepsilon}|Q(y)|\mathrm{d} y\left\| \frac{f}{\mu_k^{\varepsilon_1}\eta_k} \right\|_{\infty}.
       \end{align*}
Since
       \begin{equation*}
       \int_0^x\mu_k (y)^{1-\varepsilon}|Q(y)|\mathrm{d} y
       =o(x^0),\end{equation*}
     we obtain
       \begin{equation*}
\left|       ( G_\rightarrow^0Qf)(x)\right|\leq o(x^0)\mu_k (x)^{\varepsilon_1+\varepsilon}
\eta_k (x)\left\| \frac{f}{\mu_k^{\varepsilon_1}\eta_k} \right\|_{\infty}.
         \end{equation*}
         This proves (i).

         For $m=0$, using \eqref{pio1} and \eqref{pio3}, we have
     \begin{align*} \notag     & \left| u^0(x)\int_0^x
        v^0(y)Q(y)f(y)\mathrm{d} y\right| \\
     &
\lesssim \mu_k (x)^{\frac12}\eta_k (x) 
       \int_0^x\mu_k (y)^{\frac12+\varepsilon_1}\lambda_k (y)^{1+\alpha}\eta_{-k}(y)|Q(y)|\eta_k (y)\mathrm{d}
                y\left\| \frac{f}{\mu_k^{\varepsilon_1}\lambda_k^\alpha\eta_k} \right\|_{\infty},
                \end{align*}
and
                \begin{align*}
    \notag &
       \left| v^0(x)\int_0^x u^0(y)Q(y)f(y)\mathrm{d} y\right|\\ 
&\lesssim \mu_k (x)^{\frac12}\lambda_k (x)\eta_{-k}(x) 
       \int_0^x\mu_k (y)^{\frac12+\varepsilon_1}\lambda_k (y)^\alpha\eta_k (y)|Q(y)|\eta_k (y)\mathrm{d} y\left\| \frac{f}{\mu_k^{\varepsilon_1}\lambda_k^\alpha\eta_k} \right\|_{\infty}.
     \end{align*}
Besides the arguments used in (i), we need to 
notice that
 $y\mapsto\lambda_k (y)$ is decreasing and that $1+\alpha-\beta\leq0$. We then
 estimate  both above expressions by
\begin{align*}
       &\lesssim
         \mu_k (x)^{\varepsilon_1+\varepsilon}\lambda_k (x)^{1+\alpha-\beta}
         \eta_k (x)
       \int_0^x\mu_k (y)^{1-\varepsilon}\lambda_k (y)^\beta|Q(y)|\mathrm{d} y\left\| \frac{f}{\mu_k^{\varepsilon_1}\lambda_k^\alpha\eta_k} \right\|_{\infty}.
       \end{align*}
Since
       \begin{equation*}
       \int_0^x\mu_k (y)^{1-\varepsilon}\lambda_k (y)^\beta|Q(y)|\mathrm{d} y
       =o(x^0),\end{equation*}
     we obtain
       \begin{equation*}
\left|       ( G_\rightarrow^0Qf)(x)\right|\leq o(x^0)\mu_k (x)^{\varepsilon_1+\varepsilon}\lambda_k (x)^{1+\alpha-\beta}
\eta_k (x)\left\| \frac{f}{\mu_k^{\varepsilon_1}\lambda_k^\alpha\eta_k} \right\|_{\infty}.
         \end{equation*}
         This proves (ii).
        \end{proof}

        \begin{corollary}\label{lemma2_app}
          Let $\mathrm{Re}(k)\ge0$
          and $n \in \mathbb{N}$.
  \begin{enumerate}[label=(\roman*)]
    \item
      Let $\Re(m)\geq0$, $m\neq0$. Suppose that $Q \in \mathscr{L}^{(0)}_0$.
      Then, for all $a>0$, for all $f\in L^\infty \big(]0,a[,\mu_k^{\frac12+\Re(m)}\eta_k\big)$ and $0 < x < a$,
\begin{align*}\frac{| (G_\rightarrow^0Q)^n f(x)|}{\mu_k (x)^{\frac12+\Re(m)}\eta_k (x)}
\leq
\frac{C^{n+1} }{n!} \Big ( \int_0^x \mu_k (y)
|Q(y)| \mathrm{d}y \Big )^n\sup_{y<x}\frac{|f(y)|}{\mu_k (y)^{\frac12+\Re(m)}\eta_k (y)}.
\end{align*}
\item Suppose $k\neq0$. Let $m=0$ and $\beta\ge1$. Suppose that $Q \in \mathscr{L}^{(0)}_{0,\mathrm{ln}^\beta}$.
  Then, for all $a>0$, $f\in L^\infty \big(]0,a[,\mu_k^{\frac12}\lambda_k^{\beta-1}\eta_k\big)$ and $0 < x < a$,
\begin{align*}\frac{| (G_\rightarrow^0Q)^n f(x)|}{\mu_k (x)^{\frac12}\lambda_k (x)^{\beta-1}\eta_k (x)}
\leq
\frac{C^{n+1} }{n!} \Big ( \int_0^x \mu_k (y)\lambda_k (y)^\beta
|Q(y)| \mathrm{d}y \Big )^n\sup_{y<x}\frac{|f(y)|}{\mu_k (y)^{\frac12}\lambda_k (y)^{\beta-1}\eta_k (y)}.
\end{align*}
\item 
  Let $\Re(m)\geq0$, $m\neq0$. Suppose that $Q \in \mathscr{L}^{(0)}_{2\Re(m)}$. 
 Then, for all $a>0$, for all $f\in L^\infty \big(]0,a[,\mu_k^{\frac12-\Re(m)}\eta_k\big)$ and $0 < x < a$,
\begin{align*}\frac{| (G_\rightarrow^0Q)^n f(x)|}{\mu_k (x)^{\frac12-\Re(m)}\eta_k (x)}
\leq 
\frac{C^{n+1} }{n!} \Big ( \int_0^x \mu_k (y) ^{1-2\Re(m)}
|Q(y)| \mathrm{d}y \Big )^n\sup_{y<x}\frac{|f(y)|}{\mu_k (y)^{\frac12-\Re(m)}\eta_k (y)}. 
\end{align*}
\end{enumerate}
Above, $C$ is a constant independent of $n$ and $k$.
\end{corollary}

\begin{proof}

To prove (i)  we first follow the proof of Lemma \ref{lemma1}(i) with $\varepsilon_1 
=\frac12+\Re(m)$, $\varepsilon=0$, obtaining
\begin{equation}
\left| \frac{\mu_k (y)^{\frac12+\Re(m)}\eta_k (y)}{\mu_k (x)^{\frac12+\Re(m)}\eta_k (x)} G_\rightarrow^0(x,y)Q(y) \right| \le  \mu_k(y) Q(y) \theta(x-y).\label{qrw}
\end{equation}
The operator $G_\to^0Q$ is clearly a forward Volterra operator (see 
Appendix \ref{volterra}). 
Applying Proposition \ref{volterra1} with $K(x,y)$ given by the integral kernel appearing in the left hand side of \eqref{qrw} then yields (i).

To prove (ii) we proceed analogously, using Lemma \ref{lemma1}(ii) with $\varepsilon_1=\frac12$,
$\varepsilon=0$, $\alpha=\beta-1$, obtaining
\begin{equation*}
\left| \frac{\mu_k (y)^{\frac12+\Re(m)}\lambda_k(y)^{\beta-1}\eta_k (y)}{\mu_k (x)^{\frac12+\Re(m)}\eta_k (x)} G_\rightarrow^0(x,y)Q(y) \right| \le  \mu_k(y) \lambda_k(y)^\beta Q(y)\theta(x-y).
\end{equation*}
Assuming $\beta\ge1$ and using that $\lambda_k(x)\ge1$, this implies
\begin{equation*}
\left| \frac{\mu_k (y)^{\frac12+\Re(m)}\lambda_k(y)^{\beta-1}\eta_k (y)}{\mu_k (x)^{\frac12+\Re(m)}\lambda_k(x)^{\beta-1}\eta_k (x)} G_\rightarrow^0(x,y)Q(y) \right| \le  \mu_k(y) \lambda_k(y)^\beta Q(y)\theta(x-y),
\end{equation*}
and hence we can conclude as in the case (i).

To prove (iii)  we follow the proof of Lemma \ref{lemma1}(i) with $\varepsilon_1 
=\frac12-\Re(m)$, $\varepsilon=2\Re(m)$.
\end{proof}

Unfortunately, the case $m=0$, $k=0$ is not covered by Lemma
\ref{lemma1}
and Corollary \ref{lemma2_app}, because then $\lambda_k $ is ill
defined.
The
following lemma and its corollary work for this case. Note however that Lemma \ref{lemma1a} is not global in $x\in]0,\infty[$ and $\Re(k)\ge0$
-- we need to restrict the values of $x$ and $k$.

   \begin{lemma}\label{lemma1a}
     Let
     $\mathrm{Re}(k)\ge0$,
 $m=0$ and $Q\in L^1_\mathrm{loc}]0,\infty[$. Let
 $ \varepsilon_1+\varepsilon\geq\frac12$,
        $1\leq\beta-\alpha$. Suppose that
       \begin{equation}
         \int_0^1 y^{1-\varepsilon}(1+|\mathrm{ln}(y)|^\beta)|Q(y)|\d y<\infty.
\label{bound3}         \end{equation}
Let  $k_0>0$. Then  
              \begin{equation}\label{nolog1}
         G_\rightarrow^0Q:
       L^\infty\big(]0,1[,x^{\varepsilon_1}(1+|\mathrm{ln}(x)|^\alpha)\big)\to
       L_0^\infty\big(]0,1[,x^{\varepsilon_1+\varepsilon}
       (1+|\mathrm{ln}(x)|^{\alpha+1-\beta})\big)
       \end{equation}
       is bounded by $C\times$\eqref{bound3} uniformly in $|k|\leq k_0$. 
   \end{lemma}
     
     \begin{proof}
Suppose \eqref{bound3}.      Recall that the solution $p_0^0(\cdot,k)$ has been introduced in \eqref{pzero}. We write $p_0^0(x)=p_0^0(x,k)$ and $u^0_0(x)=u^0_0(x,k)$ to shorten notations. We then have that
     \begin{align*}
       ( G_\rightarrow^0Qf)(x)=
      - p_0^0(x)\int_0^x  u_0^0(y)Q(y)f(y)\mathrm{d} y+
  u_0^0(x)\int_0^x p_0^0(y)Q(y)f(y)\mathrm{d} y.
     \end{align*}
Now, for $0<x\leq1$, by \eqref{pio1} and \eqref{pio4},
     \begin{align*}      
     \left| u_0^0(x)\int_0^x
       p_0^0(y)Q(y)f(y)\mathrm{d} y\right|  \lesssim x^{\frac12}
       \int_0^xy^{\frac12+\varepsilon_1}(1+|\mathrm{ln}(y)|^{1+\alpha})|Q(y)|\mathrm{d}
                y \left\| \frac{f}{x^{\varepsilon_1}(1+|\mathrm{ln}(x)|^\alpha)} \right\|_{\infty},
                \end{align*}
                and
              \begin{align*}
       \left|p_0^0(x)\int_0^x u_0^0(y)Q(y)f(y)\mathrm{d} y\right| \lesssim x^{\frac12}(1+|\mathrm{ln}(x)|)
       \int_0^xy^{\frac12+\varepsilon_1}(1+|\mathrm{ln}(y)|^\alpha)|Q(y)|\mathrm{d} y\left\| \frac{f}{x^{\varepsilon_1}(1+|\mathrm{ln}(x)|^\alpha)} \right\|_{\infty}.
     \end{align*}
Using the fact that
 $y\mapsto|\mathrm{ln}(y)|$ is decreasing and $1+\alpha-\beta\leq0$, we
 estimate  both above expressions by
\begin{align*}
       &\lesssim
         x^{\varepsilon_1+\varepsilon}(1+|\mathrm{ln}(x)|^{1+\alpha-\beta})
           \int_0^xy^{1-\varepsilon}(1+|\mathrm{ln}(y)|^\beta)|Q(y)|\mathrm{d} y\left\| \frac{f}{x^{\varepsilon_1}(1+|\mathrm{ln}(x)|^\alpha)} \right\|_{\infty}.
       \end{align*}
Applying
       \begin{equation*}
       \int_0^xy^{1-\varepsilon}(1+|\mathrm{ln}(y)|^\beta)|Q(y)|\mathrm{d} y
       =o(x^0),\end{equation*}
     we obtain
       \begin{equation*}
\left|       ( G_\rightarrow^0Qf)(x)\right|\leq o(x^0)x^{\varepsilon_1+\varepsilon}(1+|\mathrm{ln}(x)|^{1+\alpha-\beta})
\left\| \frac{f}{x^{\varepsilon_1}|(1+|\mathrm{ln}(x)|^\alpha)} \right\|_{\infty}.
         \end{equation*}
This concludes the proof of \eqref{nolog1}.
 \end{proof}

        \begin{corollary}\label{lemma23_app}
          Let $k=0$, $m=0$ and $n \in \mathbb{N}$.
          
\begin{enumerate}[label=(\roman*)]
\item          Suppose that $Q \in \mathscr{L}^{(0)}_{0,\mathrm{ln}}$. Then, for all $0<a<1$, $f\in L^\infty \big(]0,a[,x^{\frac12} \big)$ and $0 < x < a$,
\begin{align*}
\frac{| (G_\rightarrow^0Q)^n f(x)|}{x^{\frac12}}
\leq
\frac{C^{n+1} }{n!} \Big ( \int_0^x y\big ( 1 + |\mathrm{ln}(y)| \big)
|Q(y)| \mathrm{d}y \Big )^n\sup_{y<x}\frac{|f(y)|}{y^{\frac12}}.
\end{align*}
\item Suppose that $Q \in \mathscr{L}^{(0)}_{0,\mathrm{ln}^2}$. Then, for all $0<a<1$, $f\in L^\infty \big(]0,a[,x^{\frac12}(|\mathrm{ln}(x)|+1) \big)$ and $0 < x < a$,
\begin{align*}
\frac{| (G_\rightarrow^0Q)^n f(x)|}{x^{\frac12}(1+|\mathrm{ln}(x)|)}
\leq
\frac{C^{n+1} }{n!} \Big ( \int_0^x y\big ( 1 + |\mathrm{ln}(y)|^2 \big)
|Q(y)| \mathrm{d}y \Big )^n\sup_{y<x}\frac{|f(y)|}{y^{\frac12}(1+|\mathrm{ln}(y)|)}.
\end{align*}
\end{enumerate}
\end{corollary}

\begin{proof}
The proof is the same as that of Corollary \ref{lemma2_app}, applying Lemma \ref{lemma1a}. To prove (i), we use \eqref{nolog1} with $\varepsilon_1=\frac12$, $\varepsilon=0$, $\beta=1$ and $\alpha=0$. To prove (ii), we use \eqref{nolog1} with $\varepsilon_1=\frac12$, $\varepsilon=0$, $\beta=2$ and $\alpha=1$.
\end{proof}

\subsection{Solutions constructed with the help  of the forward
  Green's operator}\label{sec:sol_0}

In this subsection we construct solutions  to \eqref{eq:a1}
that approximate near $0$ the solutions to the unperturbed  equation
using the forward Green's operator as the main tool.
 Here the behavior of the perturbation at infinity is irrelevant, therefore we need only 
   assumptions on $Q$  restricted to the interval $]0,a[$,
   where $a>0$ is arbitrary.

The following theorem implies Propositions \ref{pr:opo0} and
\ref{prop:reg1!} from the introduction. Note that (i) and (ii) concern principal
solutions,  (iii) and (iv) concern arbitrary solutions.

  \begin{theorem}    \label{prop:reg10}
  Let $\Re(k)\geq0$.
    \begin{enumerate}[label=(\roman*)]
    \item Suppose that  $\Re(m) \geq0$, $m\neq0$,  $\varepsilon\ge0$,
$      Q \in \mathscr{L}^{(0)}_\varepsilon $. Let
$g^0\in\cN(L_{m^2}^0+k^2)$  and $g^0=\mathcal{O}(x^{\frac12+\Re(m)})$. Then
\begin{equation*}
g:=(\one+G_\to^0Q)^{-1}g^0
\end{equation*}
is the unique solution in $AC^1]0,\infty[$ to     \eqref{eq:a1} such that, 
\begin{align*}
  g(x)-g^0(x)&=o(x^{\frac12+\Re(m)+\varepsilon}),\\
\partial_x  g(x)-\partial_x g^0(x)&=o(x^{-\frac12+\Re(m)+\varepsilon}),\quad x\to0. 
\end{align*}
    \item Suppose that  $m=0$,  $\varepsilon\ge0$,
$      Q \in \mathscr{L}^{(0)}_{\varepsilon,\mathrm{ln}} $. Let
$g^0\in\cN(L_{0}^0+k^2)$  and $g^0=\mathcal{O}(x^{\frac12})$. Then
\[g:=(\one+G_\to^0Q)^{-1}g^0\]
is the unique solution in $AC^1]0,\infty[$ to     \eqref{eq:a1} such that, 
\begin{align*}
  g(x)-g^0(x)&=o(x^{\frac12+\varepsilon}),\\
\partial_x  g(x)-\partial_x g^0(x)&=o(x^{-\frac12+\varepsilon}),\quad x\to0. 
\end{align*}
    \item Suppose that  $\Re(m) \geq0$, $m\neq0$,  $\varepsilon\ge2\Re(m)$,
$      Q \in \mathscr{L}^{(0)}_\varepsilon $. Let
$g^0\in\cN(L_{m^2}^0+k^2)$. Then
\[g:=(\one+G_\to^0Q)^{-1}g^0\]
is the unique solution in $AC^1]0,\infty[$ to     \eqref{eq:a1} such that, 
\begin{align*}
  g(x)-g^0(x)&=o(x^{\frac12-\Re(m)+\varepsilon}),\\
\partial_x  g(x)-\partial_x g^0(x)&=o(x^{-\frac12-\Re(m)+\varepsilon}),\quad x\to0. 
\end{align*}
    \item Suppose that  $m=0$,  $\varepsilon\ge0$,
$      Q \in \mathscr{L}^{(0)}_{\varepsilon,\mathrm{ln}^2} $. Let
$g^0\in\cN(L_{0}^0+k^2)$. Then
\[g:=(\one+G_\to^0Q)^{-1}g^0\]
is the unique solution in $AC^1]0,\infty[$ to     \eqref{eq:a1} such that, 
\begin{align*}
 g(x)-g^0(x)&=o(x^{\frac12+\varepsilon}),\\
\partial_x  g(x)-\partial_x g^0(x)&=o(x^{-\frac12+\varepsilon}),\quad x\to0. 
\end{align*}
\end{enumerate}
\end{theorem}

\begin{proof}
To prove (i), we use Corollary \ref{lemma2_app}(i) which shows that, for any $a>0$, $\one+G_\to^0Q$ is invertible on $L^\infty(]0,a[,\mu_k^{\frac12+\Re(m)})$ with inverse given by
\begin{equation}
\big( (\one+G_\to^0Q)^{-1}f \big ) ( x ) = \sum_{n=0}^{\infty} \big ( (- G_\to^0Q)^n f \big )(x). \label{eq:Neumann_series_exp}
\end{equation}
Hence, if $g^0\in\cN(L_{m^2}^0+k^2)$ satisfies $g^0=\mathcal{O}(x^{\frac12+\Re(m)})$, $g=(\one+G_\to^0Q)^{-1}g^0$ is well defined in $L^\infty_{\mathrm{loc}}]0,\infty[$. Since $G_\to^0$ is a Green's operator, it then easily follows that $g$ belongs to $AC^1]0,\infty[$ and is a solution to  \eqref{eq:a1}. The asymptotic behavior near $0$ of $g$ and $\partial_xg$ follow from the Neumann series expansion \eqref{eq:Neumann_series_exp} and Lemma \ref{lemma1}(i). Finally, uniqueness is a consequence of standard properties of the Wronskian of two solutions in $\cN(L_{m^2}+k^2)$.

To prove (ii) we proceed analogously, using Corollary \ref{lemma2_app}(ii) (with $\beta=1$) and Lemma \ref{lemma1}(ii) in the case where $k\neq0$. If $k=0$, we use Corollary \ref{lemma23_app}(i) and Lemma \ref{lemma1a}.

To prove (iii) we use Corollary \ref{lemma2_app}(iii) and Lemma \ref{lemma1}(i).

To prove (iv), we use Corollary \ref{lemma2_app}(ii) with $\beta=2$ and Lemma \ref{lemma1}(ii) in the case where $k\neq0$. If $k=0$, we use Corollary \ref{lemma23_app}(ii) and Lemma \ref{lemma1a}.
\end{proof}

We can apply Theorem \ref{prop:reg10}(i) and (ii) to 
$g^0(x)= u_m^0(x,k)$. 
We obtain the following result, which implies Corollary \ref{coro1}
from the introduction.

\begin{proposition}   \label{prop:reg1h}    Let $m \in \C$,  $\varepsilon\ge 
  2\max\big(-\Re(m),0\big)$. 
  Suppose that 
  \begin{align*}
Q \in \mathscr{L}^{(0)}_\varepsilon , & \text{ if } m \neq 0 , \qquad Q \in \mathscr{L}^{(0)}_{\varepsilon,\mathrm{ln}} ,  \text{ if } m = 0 . 
\end{align*}
Then 
\begin{equation*}
 u_m(\cdot,k):=(\one+G_\to^0Q)^{-1} u_m^0(\cdot,k)
\end{equation*}
is the unique solution in $AC^1]0,\infty[$ to     \eqref{eq:a1} such that, 
\begin{align}\label{uu1a3}
   u_{m}(x,k)- u_{m}^0(x,k)&=o(x^{\frac12+\Re(m)+\varepsilon}),\\
\partial_x   u_{m}(x,k)-\partial_x  u_{m}^0(x,k)&=o(x^{-\frac12+\Re(m)+\varepsilon}),\quad x\to0. \label{uu1a4}
\end{align}
\end{proposition}

Note that $\Re(m)+\varepsilon\geq|\Re(m)|$. Therefore, the
error in \eqref{uu1a3} and \eqref{uu1a4} is always of a smaller order
than the most regular solutions to     \eqref{eq:a1x0}.
Let us stress that Proposition \ref{prop:reg1h}
includes the case $k=0$, where $ u_m^0(x,0)=x^{\frac12+m}/\Gamma(m+1)$.

Recall that in \eqref{pzero}
we have introduced the family of solutions
to     \eqref{eq:a1x0} for $m=0$ with a logarithmic behavior near zero, denoted
$p_0^0(x,k)$. This family includes the logarithmic case $m=0$, $k=0$,
which does not belong to the family $ u_0^0(x,k)$:
\begin{align*}
p_0^0(x,0)&:=x^{\frac12}\mathrm{ln}(x) .
\end{align*}
We can  apply 
Theorem \ref{prop:reg10}(iv) to 
$g^0(x)=p_0^0(x,k)$, obtaining the following eigensolutions of the
perturbed eigenequation. Note that Proposition \ref{prop:reg1} implies
Corollary \ref{coro2} from the introduction.
  \begin{proposition}
      \label{prop:reg1}
      Let $\Re(k)\geq0$,
     $m=0$, $\varepsilon\ge0$ and $
Q \in \mathscr{L}^{(0)}_{\varepsilon,\mathrm{ln}^2}$.
  Then \begin{equation}\label{wrew1}
p_0(\cdot,k):=(\one+G_\to^0Q)^{-1}p_0^0(\cdot,k) 
\end{equation}
is the unique solution in $AC^1]0,\infty[$ to     \eqref{eq:a1} such that, 
\begin{align*}
  p_0(x,k)-p_0^0(x,k)&=o(x^{\frac12+\varepsilon}),\\
\partial_x    p_0(x,k)-\partial_x p_0^0(x,k)&=o(x^{-\frac12+\varepsilon}),\quad x\to0. 
\end{align*}
\end{proposition}

In the following proposition we fix the perturbation $Q$ and study the regularity of the solutions $u_m(\cdot,k)$ and $p_0(\cdot,k)$ with respect to $m$ and $k$.

 \begin{proposition}\label{prop:anal_um2}
$ $
   \begin{enumerate}[label=(\roman*)]
   \item
Let $\varepsilon>0$ and suppose that
$
Q \in \mathscr{L}^{(0)}_\varepsilon .$
 Then for any $x>0$ the maps
 \begin{align}
\Big\{\Re(m)\geq-\frac\varepsilon2\Big\}\times\C\ni (m,k)&\mapsto 
                                                u_{m}(x,k),
                                                           \partial_x
                                                           u_{m}(x,k)
    \end{align}
are regular.
\item
Let 
$Q \in \mathscr{L}^{(0)}_0$.
  Then for any $x>0$ the maps
 \begin{align}
    \Big\{\Re(m)\geq0,\ m\neq0\Big\}\times\C\ni (m,k)&\mapsto 
                                               u_{m}(x,k),\,
                                           \partial_x
                                           u_{m}(x,k)\label{porto1}
 \end{align}
are regular.
If we strengthen the assumption to
$
 Q \in \mathscr{L}^{(0)}_{0,\mathrm{ln}} ,
$
then in \eqref{porto1} we can include
 $m=0$.
\item
Let 
$Q \in \mathscr{L}^{(0)}_{0,\mathrm{ln}^2}$.
  Then for any $x>0$ the maps
 \begin{align}
    \Big\{\Re(k)\geq0\Big\}\ni k&\mapsto 
                                              p_0(x,k),\,
                                           \partial_x
                                           p_0(x,k)\label{porto1c}
 \end{align}
are regular.

\end{enumerate}
   \end{proposition}

\begin{proof}
   We use the continuity and   analyticity of the function $ u^0$  and  of
     the map $G_{\to}^0Q$ with respect to parameters. More precisely,
     for all fixed $(x,y)$, the map $(m,k)\mapsto G_\to^0(x,y)$ is
analytic. Lemma \ref{lemma2_app}
      and an induction argument then shows that, for all $x>0$, $(
      G_\to^0Q )^n  u^0(x)$ is analytic on $\{ \Re(m)>-\varepsilon/ 2 \}$. Since, by Lemma \ref{lemma2_app}, the series
     \begin{equation*}
     u_m(x,k) = \sum_{n=0}^\infty ( - G_\to^0Q )^n  u_m^0(x,k),
     \end{equation*}
     converges uniformly on every compact subset of $\{
     \Re(m)>-\varepsilon/ 2  \}$. This proves the statement
     concerning the analyticity of $ u_m$. Continuity is proven similarly.
     
     The regularity of $\partial_x u_m$, $p_0$ and $\partial_x p_0$ follows in the same way.
\end{proof}

We conclude this subsection with the following more precise estimate (compared to Proposition \ref{prop:reg1h}) on the difference between $u_0$ and $u_0^0$ (assuming that the stronger condition $Q \in \mathscr{L}^{(0)}_{\varepsilon,\mathrm{ln}^2}$ holds). We will need this estimate to study closed realization of $L_{0}$ in Section \ref{closed_operators}.
  \begin{proposition}
      \label{prop:reg1_0_bis}
Let $\Re(k)\ge0$, $m=0$ and suppose that 
  $Q \in \mathscr{L}^{(0)}_{0,\mathrm{ln}^2} .$
 Then
\begin{align*}
  u_0(x,k)- u_0^0(x,k)&=o(x^{\frac12} | \mathrm{ln}(x) |^{-1} ),\\
\partial_x  u_0(x,k)-\partial_x u_0^0(x,k)&=o(x^{-\frac12} | \mathrm{ln}(x)|^{-1}),\quad x\to0. 
\end{align*}
\end{proposition}
\begin{proof}
Recall from Proposition \ref{prop:reg1h} that $u_0-u_0^0=-G_\to^0Qu_0$ and that $u_0(x,k)=\mathcal{O}(x^\frac12)$. If $k=0$, it then suffices to use Lemma \ref{lemma1}(ii) with $\varepsilon_1=\frac12$, $\varepsilon=0$, $\beta=2$, $\alpha=0$. If $k=0$, we use Lemma \ref{lemma1a}, also with $\varepsilon_1=\frac12$, $\varepsilon=0$, $\beta=2$, $\alpha=0$.
\end{proof}

\subsection{Asymptotics of non-principal solutions
  near 0} 
  In this subsection, under the minimal assumptions $Q \in \mathscr{L}^{(0)}_0$ if $m\neq0$, $Q \in \mathscr{L}^{(0)}_{0,\mathrm{ln}}$ if $m=0$, we show that any solution to \eqref{eq:a1} behaves like non-principal unperturbed solutions near $0$. 
  
The following proposition provides a rather rough estimate on all eigensolutions. Note that, for $m\neq0$, if $Q \in \mathscr{L}^{(0)}_\varepsilon$ and $\varepsilon\ge2\Re(m)$, then Proposition \ref{cor:u_not_L2a}(i) is a consequence of Theorem \ref{prop:reg10}(iii). Likewise, if $m=0$ and $Q \in \mathscr{L}^{(0)}_{0,\mathrm{ln}^2}$, then Proposition \ref{cor:u_not_L2a}(ii) is a consequence of Theorem \ref{prop:reg10}(iv).

\begin{proposition}\label{cor:u_not_L2a}
  Let $\Re(k)\ge0$.
  \begin{enumerate}[label=(\roman*)]
  \item  Let $\Re(m)\geq0$, $m\neq0$. Suppose that
$
Q \in \mathscr{L}^{(0)}_0.$
  Then, for all $g\in\Ker(L_{m^2}+k^2)$,
\begin{align}\label{uu1g}
  g(x)=\mathcal{O}(x^{\frac12-\Re(m)}),\qquad   \partial_x g(x)=\mathcal{O}(x^{-\frac12-\Re(m)}) , \qquad 
x\to0. 
\end{align}
Moreover, if $\Re(m)>0$ and $g$ is linearly independent of $u_m(\cdot,k)$, then
\begin{align}\label{guu1}
\lim_{x\to0}\frac{g(x)}{x^{\frac12-m}}\quad\text{exists and does not vanish}.
\end{align}
\item Let $m=0$ and $
  Q \in \mathscr{L}^{(0)}_{0,\mathrm{ln}}.$ Then, for all $g\in\Ker(L_{0}+k^2)$,
  \begin{align}\label{uu1g2}
  g(x)=\mathcal{O}(x^{\frac12}\mathrm{ln}(x)),\qquad   \partial_x g(x)=\mathcal{O}(x^{-\frac12}\mathrm{ln}(x)) , \qquad 
x\to0. 
  \end{align}
 Moreover, if $g$ is linearly independent of $u_0(\cdot,k)$, then
\begin{align}
\lim_{x\to0}\frac{g(x)}{x^{\frac12}\mathrm{ln}(x)}\quad\text{exists and
 does not vanish}.
  \label{guu2}
\end{align}
\end{enumerate}
\end{proposition}
\begin{proof} 
We prove (i), (ii) follows in the same way. It is well known that the Wronskian of two
  eigensolutions of a 1-dimensional Schr\"odinger equation  is constant. Proposition \ref{prop:reg1h} gives the solution $u=u_m\in\Ker(L_{m^2}+k^2)$.
Assuming that $u_m$ and $W$ are known, we solve the ordinary differential equation
\begin{equation}
g(x)  u'(x) - g'(x)  u(x) = W , \label{eq:ODE_wronskian.u}
\end{equation}
for the unknown function $g$. Obviously the solutions to
\begin{equation*}
g(x)  u'(x) - g'(x) u(x) = 0 ,
\end{equation*}
are given by $g(x) = \lambda   u( x )$, $\lambda  \in \mathbb{C}$,
and we seek a particular solution to \eqref{eq:ODE_wronskian.u} of the
form $g( x ) = \lambda (x)  u( x )$, with $\lambda  \in C^1]0,\infty[$. This gives
\begin{equation*}
\lambda '(x)  u( x )^2  = W .
\end{equation*}
By \eqref{uu1} we know that for some $C_0\neq0$
\begin{align*}
   u(x)-C_0x^{\frac12+m}&=o(x^{\frac12+\Re(m)})
  .\end{align*}
This implies that there exists $\alpha>0$ such that $ u( x ) \neq 0$ for $0<x \le \alpha$, and hence 
\begin{align*}
  \lambda (x)-\lambda (\alpha) = \int_\alpha^x \frac{ W }{ u( y )^2 } \mathrm{d} y
  &= \int_\alpha^x  W\Big(C_0y^{-1-2m}+o(y^{-1-2\Re(m)})\Big)
  \mathrm{d} y \notag \\
 & =Cx^{-2m}+o(x^{-2\Re(m)}).
\end{align*}
Now
\begin{equation*}
g( x ) = \Big ( \lambda (\alpha) + \int_\alpha^x \frac{ W }{ u( y
  )^2 } \mathrm{d} y \Big )  u( x ) , 
\end{equation*}
implies \eqref{uu1g} and \eqref{guu1}.
\end{proof}

Note that  Proposition \ref{cor:u_not_L2a}   implies, under rather
weak assumptions, that $u_{m}(\cdot,k)$ is the only solution square
integrable near zero if $\Re(m)\geq1$.

\subsection{The two-sided Green's operator}

Mapping properties of  the two-sided Green's
operator  $G_{\bowtie}^0$ will be needed to construct solutions with a
prescribed behavior near zero in situations where we cannot apply the forward Green's operator  $G_\to^0$.
Note that the two-sided Green's operator is not invariant with respect to the
change $m\to-m$. The following lemma is meaningful only for $\Re(m)\geq0$.

The operator  $G_{\bowtie}^0Q$ is not Volterra. In order to make
$\one+G_{\bowtie}^0Q$ invertible, we will compress it to a sufficiently small
interval $]0,a[$.
Recall from \eqref{bowtie} that
$G_{\bowtie}^{0}$ compressed to the interval $]0,a[$
is denoted $G_{\bowtie}^{0(a)}$.

\begin{lemma}\label{lemmaG0}
 Let $\Re(k)\geq0$, $0< a\leq\infty$ and $Q\in L^1_\mathrm{loc}]0,\infty[$.
\begin{enumerate}[label=(\roman*)]
\item
  Let  $m\neq0$,
  $\frac12-\Re(m)\leq\varepsilon_1+\varepsilon\leq\frac12+\Re(m)$
  and
        \begin{equation}
         \int_0^a\mu_k (y)^{1-\varepsilon}|Q(y)|\d y<\infty.
\label{bound4}         \end{equation}
Then  
              \begin{equation}\label{nologa}
         G_{\bowtie}^{0(a)}Q:
 L^\infty\big(]0,a[, \mu_k ^{\varepsilon_1}\eta_{-k}\big) \to L^\infty\big(]0,a[,\mu_k ^{\varepsilon_1+\varepsilon}\eta_{-k}\big)
\end{equation}
is bounded by $C\times$\eqref{bound4} uniformly in $k$ and $a$.
Moreover, if   $\varepsilon_1+\varepsilon<\frac12+\Re(m)$,
then the image of
\eqref{nologa} is in  $L_0^\infty\big(]0,a[,\mu_k ^{\varepsilon_1+\varepsilon}\eta_{-k}\big)$.
     \item
       Let  $m=0$, $k\neq0$,
       $\frac12=\varepsilon_1+\varepsilon$
and  $0\leq\beta-\alpha\leq1$. Let
        \begin{equation}
         \int_0^a\mu_k (y)^{1-\varepsilon}\lambda_k (y)^\beta|Q(y)|\d y<\infty.
\label{bound5}         \end{equation}
Then  
              \begin{equation}\label{nologa1}
         G_{\bowtie}^{0(a)}Q:
  L^\infty\big(]0,a[,
 \mu_k ^{\varepsilon_1}\lambda_k ^\alpha\eta_{-k}\big) \to
 L^\infty\big(]0,a[,\mu_k ^{\varepsilon_1+\varepsilon}\lambda_k ^{\alpha-\beta+1}\eta_{-k}\big)
\end{equation}
is bounded by $C\times$\eqref{bound5} uniformly in $k$ and $a$.
Moreover, if $\beta-\alpha<1$, then the image of \eqref{nologa1} is in
$ L_0^\infty\big(]0,a[,\mu_k ^{\varepsilon_1+\varepsilon}\lambda_k ^{\alpha-\beta+1}\eta_{-k}\big)$.
\end{enumerate}
\end{lemma}

\begin{proof} For simplicity, let $a=\infty$.
We prove (i). We have
    \begin{align*}
    G_{\bowtie}^0Qf(x) &= u^0(x) 
    \int_x^\infty
    v^0(y)Q(y)f(y)\mathrm{d} y
+   v^0(x) \int_0^x
 u^0(y)Q(y)f(y)\mathrm{d} y .
\end{align*}
The second term is treated as in the proof
of Lemma \ref{lemma1}, 
using $\frac12-\Re(m)\leq \varepsilon_1+\varepsilon$, namely
\begin{align*}                                                                          
     \notag &
       \left| v^0(x)\int_0^x u^0(y)Q(y)f(y)\mathrm{d} y\right|\\ &\lesssim \mu_k (x)^{\varepsilon_1+\varepsilon} \eta_{-k}(x) 
       \int_0^x\mu_k (y)^{1-\varepsilon} |Q(y)| \mathrm{d} y\left\| \frac{f}{\mu_k^{\varepsilon_1}\eta_{-k}} \right\|_{\infty}.
     \end{align*}

Consider now the first term. We estimate
\begin{align*}
&\Big |  u^0(x)     \int_x^\infty    v^0(y) Q(y) f(y)\mathrm{d} y \Big |\\
&\lesssim \mu_k (x)^{\frac12+\Re(m)}\eta_k (x) \int_x^\infty \mu_k (y)^{\frac12-\Re(m)+\varepsilon_1} \eta_{-k}(y)^2 |Q(y)| \mathrm{d}y \left\| \frac{f}{\mu_k^{\varepsilon_1}\eta_{-k}} \right\|_{\infty} \\
&\lesssim \mu_k (x)^{\varepsilon_1+\varepsilon} \eta_{-k}(x)\int_x^\infty
                                                                                                                              \mu_k (y)^{1-\varepsilon} |Q(y)| \mathrm{d}y \left\| \frac{f}{\mu_k^{\varepsilon_1}\eta_{-k}} \right\|_{\infty},
\end{align*}
where we used $\frac12+\Re(m)\geq \varepsilon_1+\varepsilon$. This proves \eqref{nologa}.

Suppose now that $\varepsilon_1+\varepsilon<\frac12+\Re(m)$. Since $y\mapsto\mu_k (y)^{1-\varepsilon} |Q(y)|$ is integrable on
$]0,\infty[$,
we can apply Lemma \ref{lm:integr} with $h(y)=\mu_k (y)^{\frac12+\Re(m)-\varepsilon_1-\varepsilon}\eta_k (x)$, which gives
\begin{equation*}
\int_x^\infty \mu_k (y)^{\frac12-\Re(m)+\varepsilon_1} \eta_{-k}(y)^2 |Q(y)| \mathrm{d}y = o \big( \mu_k (x)^{-\frac12-\Re(m)+\varepsilon_1+\varepsilon}\eta_{-k}(x)^2 \big), \quad x\to0.
\end{equation*}
This yields
 \begin{equation*}
   G_{\bowtie}^0Qf(x)=o\big(\mu_k ^{\varepsilon_1+\varepsilon}(x)\eta_{-k}(x)\big) \left\| \frac{f}{\mu_k^{\varepsilon_1}\eta_{-k}} \right\|_{\infty} ,
 \end{equation*}
 and hence concludes the proof of (i).
 
 To prove (ii), we proceed similarly, replacing the estimate \eqref{pio2} on $v_m^0$ by the estimate \eqref{pio3} for $m=0$.
\end{proof}

\begin{remark}\label{rk:313}
Applying Lemma \ref{lemmaG0} with $\varepsilon=0$, it follows that, for $m\neq0$ and $\frac12-\Re(m)\leq\varepsilon_1\leq\frac12+\Re(m)$,
\begin{equation*}
\big\|G^{0(a)}_{\bowtie}Q\big\|\le C\int_0^a\mu_k(y)^{1-\varepsilon}|Q(y)|\mathrm{d} y \text{ on } L^\infty\big(]0,a[, \mu_k^{\varepsilon_1}\eta_{-k}\big) ,
\end{equation*}
where the constant $C$ is independent of $k$ and $a$, but dependent on
$m$. However, if the values of $m$ are restricted to
$|m|> m_0 $, $0\le\Re(m)\le M$ for some $ m_0 >0$, $M>0$, then one infers from the proof that the constant $C$ can be chosen uniformly.
\end{remark}

In the next corollary we show the invertibility of $\one+G_{\bowtie}^{0(a)}Q$ for small enough $a>0$.

  \begin{corollary}\label{lemmaG0a}
Let $\Re(k)\geq0$. 
\begin{enumerate}[label=(\roman*)]
\item
Let $m\neq0$ and
  $\frac12-\Re(m)\leq\varepsilon_1\leq\frac12+\Re(m)$. Suppose that
  $Q\in \mathscr{L}_0^{(0)}$.
Then, for  small
  enough $a>0$, we have 
  \begin{equation*}
  \big \| G_{\bowtie}^{0(a)}Q \big \|<1 \text{ on } L^\infty(]0,a[,\mu_k ^{\varepsilon_1}\eta_{-k}),
  \end{equation*}
    so that
  $(\one+G_{\bowtie}^{0(a)}Q)^{-1}$ exists.
\item
  Let $m=0$, $k\neq0$. Suppose that
 $Q\in  \mathscr{L}_{0,\mathrm{ln}}^{(0)}$ and
 $0\leq\alpha\leq1$. Then, for
 small
enough $a>0$, we have 
\begin{equation*}
\big\|G_{\bowtie}^{0(a)}Q\big\|<1 \text{ on } L^\infty(]0,a[, \mu_k ^{\varepsilon_1}\lambda_k ^\alpha\eta_{-k}),
\end{equation*}
so that
  $(\one+G_{\bowtie}^{0,(a)}Q)^{-1}$ exists.
\end{enumerate}
\end{corollary}

\begin{proof}
To prove (i), it suffices to apply Lemma \ref{lemmaG0}(i) with $\varepsilon=0$.

To prove (ii), we apply Lemma \ref{lemmaG0}(ii) with $\varepsilon=0$ and $\beta=1$.
\end{proof}

 \subsection{Solutions constructed with the help of the two-sided
   Green's operator}\label{sec:sol_01}

 The goal of this subsection is similar to that of Subsection \ref{sec:sol_0}: to construct solutions  to \eqref{eq:a1}
that approximate near $0$ the solutions to the unperturbed  equation
\eqref{eq:a1x0}. 
In this subsection we cover a different parameter range than in Subsection \ref{sec:sol_0}. This is accomplished by using a different tool. Instead
of the forward Green's operator, we use the
two-sided Green's 
operator compressed to  a sufficiently small interval
$]0,a[$, 
$G_{\bowtie}^{0(a)}$, which was studied in the previous subsection. 
 The construction here will be
less canonical than in Subsection \ref{sec:sol_0} -- it will depend on the
parameter $a$.

\begin{theorem}
Let $\Re(k)\ge0$.
  \begin{enumerate}[label=(\roman*)]
                                 \item
 Let  $ m_0 >0$, $M>0$ and $Q\in
   \mathscr{L}^{(0)}_0$. Then for small enough $a>0$, for all $m\in\mathbb{C}$ such that $|m|> m_0 $, $0\le\Re(m)\le M$, for all
    $g^0\in\cN(L_{m^2}^0+k^2)$,
\begin{equation*}
  g^{\bowtie}:=(\one+G_{\bowtie}^{0(a)}Q)^{-1}g^0
  \end{equation*}
is a solution in $ AC^1]0,\infty[$ to \eqref{eq:a1}. If in addition $0\leq\varepsilon<2\Re(m)$, then 
\begin{align}
  g^{\bowtie}(x)-g^0(x)&=o(x^{\frac12-\Re(m)+\varepsilon}), \label{eq:fkap1}
  \\
\partial_x
  g^{\bowtie}(x)-\partial_xg^0(x)&=o(x^{-\frac12-\Re(m)+\varepsilon}). \label{eq:fkap2}
\end{align}
\item
Let $m=0$ and assume that $k\neq 0$. Suppose that $Q\in
\mathscr{L}^{(0)}_{0,\mathrm{ln}}$. Then for small enough $a>0$, for all $g^0\in\cN(L_{0}^0+k^2)$,
\begin{equation*}
  g^{\bowtie}:=(\one+G_{\bowtie}^{0(a)}Q)^{-1}g^0
  \end{equation*}
is a solution  in  $ AC^1]0,\infty[$
to     \eqref{eq:a1} such that
\begin{align*}
  g^{\bowtie}(x)-g^0(x)&=o(x^{\frac12}\mathrm{ln}(x)),\\
\partial_x    g^{\bowtie}(x)-\partial_xg^0(x)&=o(x^{-\frac12}\mathrm{ln}(x)). 
\end{align*}
  \end{enumerate}
\label{conhua}
  \end{theorem}
  
\begin{proof}
To prove (i), we apply Corollary \ref{lemmaG0a}(i) (and Remark \ref{rk:313}) with $\varepsilon_1=\frac12-\Re(m)$: By making $a>0$ small
enough, we can thus make sure that for
 $|m|> m_0 $, $0\leq\Re(m)\leq M$, $m\neq0$, the operator  $G_{\bowtie}^{0(a)}Q$ has
the norm $<1$  on the space $L^\infty(]0,a[,\mu_k ^{\frac12-\Re(m)}\eta_{-k})$. Hence $g^{\bowtie}$ is well-defined and is a solution to \eqref{eq:a1}.

We have $g^{\bowtie} - g^0 = (-G_{\bowtie}^{0(a)}Q)g^{\bowtie}$.  Since $g^{\bowtie} = \mathcal{O}(x^{\frac12-\Re(m)})$ by Proposition \ref{cor:u_not_L2a} and since $\varepsilon<2\Re(m)$, we can apply Lemma \ref{lemmaG0}(i) with $\varepsilon_1=\frac12-\Re(m)$. This yields \eqref{eq:fkap1}--\eqref{eq:fkap2}.

To prove (ii) we proceed in the same way, using Corollary
\ref{lemmaG0a}(ii) and Lemma \ref{lemmaG0}(ii)
 with $\varepsilon=0$ and $\alpha=\beta=1$.
\end{proof}

 We can apply Theorem \ref{conhua}(i)
to  the unperturbed solutions
$g^0$ equal to $u_{-m}^0(\cdot,k)$ and
$u_{m}^0(\cdot,k)$, obtaining solutions $u_{-m}^{\bowtie(a)}(\cdot,k)$ and
$u_{m}^{\bowtie(a)}(\cdot,k)$.
The solutions $u_m^{\bowtie(a)}(\cdot,k)$ for $\Re(m)\geq0$ are not very useful,
 since we have then $ u_m (\cdot,k)$ at our disposal.
Therefore, in the following proposition we restrict ourselves to
$u_{-m}^{\bowtie(a)}(\cdot,k)$. They can serve as a non-principal 
solution defined when $u_{-m}$ is not available. 

\begin{proposition}\label{prop:ubowtie}
Let $\Re(k)\ge0$. Suppose that  $Q\in  \mathscr{L}^{(0)}_0$. Let
 $ m_0 >0$, $M>0$. Let
   $a>0$ be small
enough as in the previous theorem. Then for
$|m|> m_0 $, $0\le \Re(m)\leq M$, setting
\begin{equation*}
   u_{-m}^{\bowtie(a)}(\cdot,k):=(\one+G_{\bowtie}^{0(a)}Q)^{-1} u_{-m}^0(\cdot,k)
\end{equation*}
we obtain  a solution  in  $ AC^1]0,\infty[$
   to     \eqref{eq:a1}. 
If we impose the assumption
$Q\in  \mathscr{L}^{(0)}_\varepsilon$ for
 $0\leq\varepsilon<2\Re(m)$, then 
\begin{align}
   u_{-m}^{\bowtie(a)}(x,k)-
  u_{-m}^0(x,k)&=o(x^{\frac12-\Re(m)+\varepsilon})
                 ,\label{eq:ubowtieasympt1}\\
\partial_x  u_{-m}^{\bowtie(a)}(x,k)-\partial_x 
  u_{-m}^0(x,k)&=o(x^{-\frac12-\Re(m)+\varepsilon})
                 .\label{eq:ubowtieasympt2}
\end{align}
If  $Q\in  \mathscr{L}^{(0)}_\varepsilon$ with
$0\le2\Re(m)\le\varepsilon$ and $m\not\in\mathbb{N}$,
then there exists  $c_m^{\bowtie(a)}(k)\in\mathbb{C}$ such that
\begin{equation} 
u_{-m}^{\bowtie(a)}(x,k)= u_{-m}(x,k)+c_m^{\bowtie(a)}(k) u_m(x,k).\label{cem}
\end{equation}
\end{proposition}

\begin{proof}
The first part of the proposition is a direct consequence of Theorem
\ref{conhua}(i).
To see \eqref{cem} note  that by  Proposition \ref{prop:reg1h}, both $u_m(\cdot,k)$ and $u_{-m}(\cdot,k)$ are well-defined. Moreover, if $m\notin\mathbb{N}$, they are linearly independent, as follows from their asymptotics near $0$. 
\end{proof}

The  eigensolutions $ u_{{-}m}^{\bowtie(a)} $ constructed in
Proposition \ref{prop:ubowtie} are given by  convergent
expansions
\begin{align}
   u_{{-}m}^{\bowtie(a)}(x,k)&=\sum_{j=0}^\infty
                             (-G_{\bowtie}^{0(a)}Q)^j u_{{-}m}^0(x,k).
\label{series}\end{align}
Note that the individual terms on the right hand side of 
\eqref{series} are well defined under rather weak
assumptions and their behavior near zero weakly depends on $a$.

\begin{lemma}
\label{prop:indep_a} Let $\Re(k)\geq0$. Assume that $\Re(m)\geq0$, $m\neq0$ and
$Q\in  \mathscr{L}^{(0)}_0$, or $m=0$, $Q\in
\mathscr{L}^{(0)}_{0,\ln}$ and $k\neq0$. Let $0<a,b$. Then for any $j\in\mathbb{N}$
there
exists $c_{m}^j(k)$ such that
\begin{align}\label{diffe}
                              (G_{\bowtie}^{0(a)}Q)^j
  u_{{-}m}^0(x,k)-
  (G_{\bowtie}^{0(b)}Q)^j u_{{-}m}^0(x,k)=  c_{m}^j(k) u^0_m(x,k) + o(
  x^{\frac12+\Re(m)} ) . 
\end{align}
\end{lemma}

\begin{proof}
Suppose that $m\neq0$. We will prove \eqref{diffe} by induction with respect to $j$. Let us denote
the left hand side of \eqref{diffe} by $z^j(x,k)$. Clearly, $z^0(x,k)=0$.
Assume that \eqref{diffe} is true for a given $j$. Let
$0\leq x\leq a<b$. Now
\begin{align*}
z^{j+1}(x)&=(G_{\bowtie}^{0(a)}Q-G_{\bowtie}^{0(b)}Q) (G_{\bowtie}^{0(a)}Q)^ju_{{-}m}^0(x)+G_{\bowtie}^{0(b)}Qz^j(x)\\
  &=
 u^0_m(x) 
    \int_a^b
  v^0_m(y)Q(y)(G_{\bowtie}^{0(a)}Q)^{j} u_{{-}m}^0(y) \mathrm{d} y
  + u^0_m(x) 
    \int_0^b
  v^0_m(y)Q(y)z^j(y)\mathrm{d} y\\
  &\quad-
   u^0_m(x) 
    \int_0^x
    v^0_m(y)Q(y)z^j(y)\mathrm{d} y
    +v^0_m(y)\int_0^xu_m^0(y)Q(y) z^j(y) \mathrm{d} y.
\end{align*}
The first term is clearly proportional to $u_m^0$.
By the induction assumption $z_j=\mathcal{O}(x^{\frac12+\Re(m)})$. Therefore
the integral in the second term is finite, and hence the second term
is also proportional to $u_m^0$. By the same argument the
third term  is $o(x^{\frac12+\Re(m)})$. Finally,
since $z_j=\mathcal{O}(x^{\frac12+\Re(m)})$,
the integral in the fourth term is $o(x^{2\Re(m)})$. Hence the 
fourth term is also $o(x^{\frac12+\Re(m)})$.

The case $m=0$ with $Q\in
\mathscr{L}^{(0)}_{0,\ln}$ and $k\neq0$ can be treated in the same way.
\end{proof}

Under the assumptions of Lemma \ref{prop:indep_a}, we introduce
the following notation for a partial sum of the series \eqref{series}:
\begin{align}\label{defrn}
   u_{{-}m}^{0(a)[n]}(x,k)&:=\sum_{j=0}^n
                                (-G_{\bowtie}^{0(a)}Q)^j
                                 u_{{-}m}^0(x,k)
                                 ,\\ u_{{-}m}^{0[n]}(x,k)&:= u_{{-}m}^{0(1)[n]}(x,k).
 \end{align}
Thus we choose (quite arbitrarily) $a=1$ as the ``standard value'' in \eqref{defrn}.

 The next proposition shows that the functions
  $u_{{-}m}^{0[n]}(\cdot,k)$
  well approximate non-principal solutions under the assumption
  $0\leq\Re(m)\leq\frac{\varepsilon }{2}(n+1)$, $m\neq0$.

\begin{proposition}\label{comui}
Let $\Re(k)\ge0$ and $\Re(m)\ge0$, $m\neq0$. Assume that $\varepsilon\ge0$ and  $Q\in  \mathscr{L}^{(0)}_\varepsilon$. 
Let $n$ be a nonnegative integer  such that 
$\varepsilon\ge\frac{2}{n+1}\Re(m)$. Suppose that 
   $a>0$ is small 
enough, so that $u_{-m}^{\bowtie(a)}(\cdot,k)$ is well defined, as described in Proposition \ref{prop:ubowtie}. 
 Then there exists 
$c_{m}^{(a)[n]}(k)\in\mathbb{C}$ such that 
  \begin{align*}
 u_{{-}m}^{\bowtie(a)}(x,k)- u_{{-}m}^{0[n]}(x,k)-c_{m}^{(a)[n]}(k) u_m(x,k)&=o(x^{\frac12+\Re(m)}),\\
    \partial_x                                                                
    u_{{-}m}^{\bowtie(a)}(x,k)-\partial_x u_{{-}m}^{0[n]}(x,k)- c_{m}^{(a)[n]}(k) \partial_x u_m(x,k)
    &=o(x^{-\frac12+\Re(m)}).
  \end{align*}
\end{proposition}
  
\begin{proof}
Applying   \begin{equation}\label{labb}
  (G_{\bowtie}^{0(a)}Qf)(x) = u^0(x) \int_0^a v^0 (y) Q(y) f(y) \d y + G_\to^0Qf(x),
  \end{equation}
we can write
\begin{align}\notag
  u_{{-}m}^{\bowtie(a)} (x,k) =&    u_{{-}m}^{0(a)[n]}(x,k) +
    (-G_{\bowtie}^{0(a)}Q)^{n+1} u_{{-}m}^{\bowtie(a)} (x,k)\\\notag
  =&
    u_{{-}m}^{0(a)[n]}(x,k)-u_m^0(x,k)\int_0^a v_m^0(x,k)Q(y)  (-G_{\bowtie}^{0(a)}Q)^{n}
u_{{-}m}^{\bowtie(a)} (y,k)\d y\\&-G_\to^0Q  (-G_{\bowtie}^{0(a)}Q)^{n}u_{{-}m}^{\bowtie(a)}(x,k).
\end{align}
Suppose $n>0$. Let $\tilde n\le n$ be a positive integer such that $\frac{2}{\tilde n}\Re(m)>\varepsilon\ge\frac{2}{\tilde n+1}\Re(m)$. We apply repeatedly Lemma \ref{lemmaG0}(i) with 
$\varepsilon_1=\frac12-\Re(m)+j\varepsilon$ for $j=0,\dots,\tilde n-1$, noting that $\frac12-\Re(m)\le\varepsilon_1+\varepsilon<\frac12+\Re(m)$, to 
show that 
\begin{equation}(-G_{\bowtie}^{0(a)}Q)^{\tilde n} u_{{-}m}^{\bowtie(a)} (x,k) = o(
x^{\frac12-\Re(m)+\tilde n\varepsilon}).\label{abb1}\end{equation}
Applying then again repeatedly Lemma \ref{lemmaG0}(i) with $\varepsilon=0$ we deduce that
\begin{equation}(-G_{\bowtie}^{0(a)}Q)^{n} u_{{-}m}^{\bowtie(a)} (x,k) = o(
x^{\frac12-\Re(m)+\tilde n\varepsilon}).\label{abb1-2}\end{equation}
 Because of this, and since $\varepsilon\ge\frac{2}{\tilde n+1}\Re(m)$,
\begin{equation}
  \int_0^a v_m^0(y,k)Q(x)  (-G_{\bowtie}^{0(a)}Q)^{n}
  u_{{-}m}^{\bowtie(a)} (y,k)\d y \end{equation}
is finite. Now we apply Lemma \ref{lemma1}(i) with
$\varepsilon_1=\frac12-\Re(m)+\tilde n\varepsilon$, noting that
$\varepsilon_1+\varepsilon\ge\frac12+\Re(m)$,  to show that
\begin{equation}\label{eq:rht-1}
  G_\to^0Q  (-G_{\bowtie}^{0(a)}Q)^{n}u_{{-}m}^{\bowtie(a)} (x,k)=o(x^{\frac12+\Re(m)}).
\end{equation}
If $n=0$, applying Lemma \ref{lemma1}(i) with
$\varepsilon_1=\frac12-\Re(m)$, we see that \eqref{eq:rht-1} still holds.
Finally, by Lemma \ref{prop:indep_a} we can replace
$u_{{-}m}^{0(a)[n]}(x,k)$ with
$u_{{-}m}^{0[n]}(x,k)$.
\end{proof}

We can use the functions $u_{{-}m}^{0[n]}(\cdot,k)$ to describe
boundary conditions near zero of non-principal solutions.

  \begin{proposition}   \label{boundcon}
  Let $\Re(k)\geq0$ and $\Re(m)\ge0$, $m\neq0$.
  Suppose that  $Q \in \mathscr{L}^{(0)}_\varepsilon$,  $\varepsilon\ge0$.
Let  $n$ be a nonnegative integer such that
  $\frac{\varepsilon}{2}(n+1)\geq \Re(m)$.
Then 
\begin{align}\label{un-m}
 u_{-m}^{[n]}(\cdot,k)&:=u_{-m}^{0[n]}(\cdot,k)
 +(-1)^{n+1}\big(\one+G_\to^0Q\big)^{-1}
                          G_\to^0Q\big(G_{\bowtie}^{0(1)}Q)^nu_{-m}^{0}(\cdot,k)
\end{align}
is a solution 
in $AC^1]0,\infty[$ to     \eqref{eq:a1} such that
\begin{align}\label{soli1}
    u_{-m}^{[n]}(x,k)-u_{-m}^{0[n]}(x,k)
  &=o(x^{\frac12+\Re(m)}),\\\label{soli2}
   \partial_x u_{-m}^{[n]}(x,k)-\partial_xu_{-m}^{0[n]}(x,k)
&=o(x^{-\frac12+\Re(m)}).
  \end{align}
 \end{proposition}
\begin{proof}
Note that $u_{-m}^{0[n]}(x,k) = \mathcal{O}(x^{\frac12-\Re(m)})$. As in the proof of the previous proposition, applying repeatedly Lemma \ref{lemmaG0}(i) and next Lemma \ref{lemma1}(i), we obtain that
\begin{align*} 
G_\to^0Q\big(G_{\bowtie}^0Q)^nu_{-m}^{0[n]}(\cdot,k) = o(x^{\frac12+\Re(m)}).
\end{align*}
Then we can use Corollary \ref{lemma2_app}(i) which shows that, for any $a>0$, $\one+G_\to^0Q$ is invertible on $L^\infty(]0,a[,\mu_k^{\frac12+\Re(m)})$. Applying $L_{m^2}+k^2=(L^0_{m^2}+k^2)(\one+G_\to^0Q)$ to \eqref{un-m} and using the definition \eqref{defrn}, we then obtain
\begin{align*}
&(L_{m^2}+k^2) (u_{-m}^{[n]}(\cdot,k))\\
&=(L_{m^2}+k^2)\sum_{j=0}^n (-G_{\bowtie}^{0(1)}Q)^j u_{{-}m}^0(\cdot,k) -Q\big(-G_{\bowtie}^{0(1)}Q)^nu_{-m}^{0}(\cdot,k).
\end{align*}
Next, using that $L_{m^2}+k^2=L^0_{m^2}+k^2+Q$ together with the fact that $G_{\bowtie}^{0(1)}$ is a right inverse of $L^0_{m^2}+k^2$ gives
\begin{align*}
&(L_{m^2}+k^2) (u_{-m}^{[n]}(\cdot,k))\\
&=-Q\sum_{j=1}^n (-G_{\bowtie}^{0(1)}Q)^{j-1} u_{{-}m}^0(\cdot,k)+Q\sum_{j=0}^n (-G_{\bowtie}^{0(1)}Q)^j u_{{-}m}^0(\cdot,k) -Q\big(-G_{\bowtie}^{0(1)}Q)^nu_{-m}^{0}(\cdot,k)\\
&=0.
\end{align*}
Hence $u_{-m}^{[n]}(\cdot,k)$ belongs to $AC^1]0,\infty[$ and is a solution to  \eqref{eq:a1} satisfying \eqref{soli1}--\eqref{soli2}.
\end{proof}

 \begin{proposition}\label{prop:anal_um^n}
Let $\varepsilon\ge0$, $n$ a nonnegative integer and suppose that
$
Q \in \mathscr{L}^{(0)}_\varepsilon .$
 Then for any $x>0$ the maps
 \begin{align*}
    \Big\{\frac{\varepsilon}{2}(n+1)\geq\Re(m)\geq0,\ m\neq0\Big\}\times    \Big\{\Re(k)\geq0\Big\}\ni (m,k)&\mapsto 
                                                u_{-m}^{[n]}(x,k),
                                                           \partial_x
                                                           u_{-m}^{[n]}(x,k)
    \end{align*}
are regular.
   \end{proposition}

\begin{proof}
The proof is similar to that of Proposition \ref{prop:anal_um2}.
\end{proof}

Here is a drawback of Proposition \ref{boundcon}: the
boundary conditions are described by a  function 
$u_{-m}^{0[n]}(\cdot,k)$ which depends on $k$.
We already know that for principal solutions 
the boundary condition does not depend on $k$. One can ask whether one can use the same
boundary conditions for all $k$ in the non-principal case, e.g.
\begin{align}
 u_{{-}m}^{0[n]} (x)&:= u_{{-}m}^{0[n]}(x,0).\label{defrn1}
 \end{align}
 Thus we would like to use $k=0$ as the ``standard value'' in \eqref{defrn1},
which typically gives the simplest 
  expressions.

 Let us check what is the situation in the
unperturbed case. Let $\Re(m)\geq0$. We have
\begin{equation}\label{eq:asymptumu-m}
u_{-m}^0(x,k)=\frac{x^{\frac12-m}}{\Gamma(1-m)}+\mathcal{O}(x^{\frac52-\Re(m)}),\qquad
  u_{m}^0(x,k)=\mathcal{O}(x^{\frac12+\Re(m)}).
  \end{equation}
Hence we need the condition $\Re(m)<1$ to make sure that
\begin{equation}\label{eq:asymptumu-m-2}
  u_{-m}^0(x,k)=\frac{x^{\frac12-m}}{\Gamma(1-m)}+o(x^{\frac12+\Re(m)}),
\end{equation}
which guarantees that 
$  u_{-m}^0(x,k)$ with distinct $k$ give the same boundary condition.

\begin{proposition} \label{thm:k=0}
In addition to the assumptions of
Proposition \ref{boundcon} suppose that $\Re(m)<1$. Then in
  \eqref{soli1} and \eqref{soli2}
  we can replace $u_{-m}^{0[n]}(x,k)$ with
  $u_{-m}^{0[n]}(x)$ defined  in \eqref{defrn1}, (or with $u_{-m}^{0[n]}(x,k')$ for any $k'$).
\end{proposition}

\begin{proof}
Proposition \ref{thm:k=0} easily follows from the definition \eqref{defrn1} of $u_{-m}^{0[n]}(x,k)$ together with \eqref{eq:asymptumu-m-2}.
\end{proof}

In concrete cases, it is not difficult to compute $u_{{-}m}^{0[n]}$ explicitly. The following remark provides an example in the case where $Q$ has a Coulomb singularity at $0$.

\begin{remark}\label{rk:coulomb}
Suppose that $Q(x) = - \frac{ \beta }{ x }\one_{]0,1]}(x)$ with $\beta\in\C$. Then $Q
\in \mathscr{L}^{(0)}_\varepsilon$ for $\varepsilon<1$.
Hence for $0\le\Re(m)<1$, $m\neq0$, we can take $n=1$ in 
Proposition \ref{comui} and we have that
\begin{align*}
u_{-m}^{0[1]}(x) = u^0_{-m}(x,0) - G_{\bowtie}^{0(1)}Qu_{-m}^0(x,0).
\end{align*}
Consider for simplicity the generic case $m\neq\frac12$. Since $u^0_{\pm m}(x,0)=\frac{x^{\frac12\pm m}}{\Gamma(1 \pm m)}$ and $v^0_m(x,0)=\frac12\Gamma(m)x^{\frac12-m}$, we can compute
\begin{align*}
G_{\bowtie}^{0(1)}Qu_{-m}^0(x)&=\frac{1}{2m\Gamma(1-m)} \Big( x^{\frac12+m} \int_x^1y^{1-2m} \frac{-\beta}{y}\mathrm{d}y + x^{\frac12-m} \int_0^xy \frac{-\beta}{y}\mathrm{d}y \Big ) \\
 &= \frac{\beta}{2m\Gamma(1-m)} \Big ( x^{\frac12+m} \frac{x^{1-2m}}{1-2m} - \frac{x^{\frac12+m}}{1-2m} -  x^{\frac32-m}  \Big ) \\
 &=\frac{ x^{\frac12-m} }{ \Gamma(1-m) }  \frac{\beta x}{1-2m} - \frac{\beta x^{\frac12+m} }{2m(1-2m)\Gamma(1-m)} .
\end{align*}
Hence
\begin{align*}
u_{-m}^{0[1]}(x) = \frac{ x^{\frac12-m} }{ \Gamma(1-m) } \Big ( 1 -  \frac{\beta x}{1-2m} \Big ) - \frac{\beta x^{\frac12+m} }{2m(1-2m)\Gamma(1-m)} .
\end{align*}
We recover the function  $j_{\beta,-m}$ from (2.3) of
\cite{DeFaNgRi20_01}, which was used to describe the boundary
conditions of the Whittaker operator.
\end{remark}

\subsection{The logarithmic Green's operator}

For $m=0$ we could use
 Lemma \ref{lemmaG0}(ii), Corollary \ref{lemmaG0a}(ii) and
 Theorem \ref{conhua}(ii) to construct eigensolutions with the help of
 the two-sided Green's operator $G_{\bowtie}^0$. The  drawback of this
 approach is the
 lack of the limit at $k=0$. Therefore for $m=0$ we 
prefer to use the logarithmic Green's operator $G_\diamond^0$, which is well defined for $k=0$.
More precisely, we will use the logarithmic Green's operator compressed to
a finite interval, $         G_\diamond^{0(a)}$.

Below we describe mapping properties of $
G_\diamond^{0(a)}$. The result  is analogous to
Lemma \ref{lemmaG0}(ii), however includes  $k=0$.

\begin{lemma}\label{lemmaG0b} 
Let $k_0>0$, $\Re(k)\geq0$ such that $|k|\le k_0$, $0< a<1$,
  $\frac12=\varepsilon_1+\varepsilon$ and
  $0\leq\beta-\alpha\leq1$. Suppose that $Q\in L^1_\mathrm{loc}]0,\infty[$ and
       \begin{equation}
         \int_0^a y^{1-\varepsilon}\big(1-\mathrm{ln}(y)\big)^\beta|Q(y)|\d y<\infty.
\label{bound6}         \end{equation}
Then  
              \begin{equation}\label{nolog2}
         G_\diamond^{0(a)}Q:
       L^\infty\big(]0,a[,x^{\varepsilon_1}(1-\mathrm{ln}(x))^\alpha\big)\to
       L^\infty\big(]0,a[,x^{\varepsilon_1+\varepsilon}
       (1-\mathrm{ln}(x))^{\alpha+1-\beta}\big)
       \end{equation}
       is bounded by $C\times$\eqref{bound6} uniformly in $0< a<1
       $ and $|k|\leq k_0$. If in addition
        $\beta-\alpha<1$, then the image of \eqref{nolog2} is
        contained in $L_0^\infty\big(]0,a[,x^{\varepsilon_1+\varepsilon}
       (1-\mathrm{ln}(x))^{\alpha+1-\beta}\big)$.
\end{lemma}

\begin{proof}
The proof is identical to that of Lemma \ref{lemmaG0}(ii), using the solution $p_0^0$ instead of $v_0^0$ and \eqref{pio4} instead of \eqref{pio3}.
\end{proof}

  \begin{corollary}\label{lemmaG0a-log}
Let  $k_0>0$, $\Re(k)\geq0$ such that $|k|\le k_0$. Suppose that
    $Q\in\mathscr{L}_{0,\mathrm{ln}}^{(0)}$. Then for
    $0\leq\alpha\leq1$ and
 small
enough $a>0$ we have 
\begin{equation*}
\|G_{\diamond}^{0(a)}Q\|<1 \text{ on } L^\infty(]0,a[,x ^{\frac12}(1-\mathrm{ln}(x))^\alpha),
\end{equation*}
so that
  $(\one+G_{\diamond}^{0(a)}Q)^{-1}$ exists.
\end{corollary}

\begin{proof}
It suffices to apply Lemma \ref{lemmaG0b} with $\varepsilon=0$ and $\beta=1$.
\end{proof}

\subsection{Solutions constructed with help of the logarithmic Green's 
  operator}

 We continue with the case $m=0$.
The
following theorem is the analog
of  Theorem \ref{conhua}(ii) in the context of the logarithmic Green's operator
 $G_\diamond^0$.
 
 \begin{theorem}\label{conhua-diamond}
Let  $k_0>0$ and $\Re(k)\geq0$ such that $|k|\le k_0$. Suppose that
    $Q\in\mathscr{L}_{0,\mathrm{ln}}^{(0)}$. Then for all
    $g^0\in\cN(L_{0}^0+k^2)$, for small enough $a$
    \[\label{defr}  g^\diamond:=(\one+G_\diamond^{0(a)}Q)^{-1}g^0\]
    exists and 
is a  solution  in  $ AC^1]0,\infty[$
to     \eqref{eq:a1} such that
\begin{align*}
  g^\diamond(x)-g^0(x)&=o(x^{\frac12}\mathrm{ln}(x)),\\
\partial_x    g^\diamond(x)-\partial_xg^0(x)&=o(x^{-\frac12}\mathrm{ln}(x)). 
\end{align*}
\end{theorem}
\begin{proof}
It suffices to proceed as in the proof of  Proposition \ref{conhua},
using Corollary \ref{lemmaG0a-log} and Lemma \ref{lemmaG0b}
 with $\alpha=\beta=1$.
\end{proof}

Applying Theorem \ref{conhua-diamond}  to
$p_0^0$, we obtain the following result.

\begin{proposition}\label{prop:pdiamonda}
Let  $k_0>0$, $\Re(k)\geq0$ such that $|k|\le k_0$. Suppose that
    $Q\in\mathscr{L}_{0,\mathrm{ln}}^{(0)}$. Then for 
   $a>0$ small
enough,
     \begin{equation}
       \label{defrdia}
       p_0^{\diamond(a)}:=(\one+G_\diamond^{0(a)}Q)^{-1}p_0^0
     \end{equation}
is a  non-principal solution  in  $ AC^1]0,\infty[$ such that
   to     \eqref{eq:a1}
\begin{align*}
  p_0^{\diamond(a)}(x,k)-
  p_0^0(x,k)&=o(x^{\frac12}\mathrm{ln}(x))
                 ,\\
  \partial_x    p_0^{\diamond(a)}(x,k)-\partial_x
  p_0^0(x,k)&=o(x^{-\frac12}\mathrm{ln}(x))
                 .
\end{align*}
\end{proposition}

\begin{proof}
This is a direct consequence of Theorem \ref{conhua-diamond}.
\end{proof}

\subsection{Summary of distinguished solutions} \label{subsec:summary0}
The next table summarizes the distinguish solutions of the perturbed eigenequation with a prescribed behavior near the origin constructed in this section.

\medskip
\renewcommand{\arraystretch}{1.7}
\begin{center}
\begin{tabular}{|Sc|Sc|Sc|Sc|}
  \hline
  Solution & Parameters & Conditions on $Q$ & Green's operator \\
  \hline
\multirow{2}{*}{$u_m(\cdot,k)$} &
$\Re(m)\ge-\frac{\varepsilon}{2}$,
                                  $m\neq0$&
                                            $Q\in\mathscr{L}^{(0)}_\varepsilon$, $\varepsilon\ge0$ & \multirow{2}{*}{Forward  $G^0_\rightarrow$}\\
  & $m=0$  & $Q\in\mathscr{L}^{(0)}_{0,\mathrm{ln}}$  &  \\
  \hline
  $p_0(\cdot,k)$ & $m=0$ & $Q\in\mathscr{L}^{(0)}_{0,\mathrm{ln}^2}$ & Forward  $G^0_\rightarrow$\\
  \hline
  \multirow{2}{*}{$ u_{m}^{\bowtie(a)}(\cdot,k)$} &
                                                    \multirow{2}{*}{$-M
                                                    \le \Re(m)<0$, $|m|> m_0 >0$}  & \multirow{2}{*}{$Q\in\mathscr{L}^{(0)}_0$} & Two-sided  $G^{0(a)}_{\bowtie}$ \\ 
   &  & & compressed to $]0,a[$  \\
 \hline
$  u_{-m}^{[n]}(\cdot,k)$&$\frac{\varepsilon}{2}(n+1)\ge\Re(m)\geq0$,
                        &
                          $Q\in\mathscr{L}^{(0)}_\varepsilon$,
                          $\varepsilon>0$&\begin{tabular}{c}Forward
                                            $G^0_\rightarrow$\\
                                            and two-sided
                                            $G_{\bowtie}^{0(1)}$
                                            \end{tabular}\\                                         
   \hline
 \multirow{2}{*}{$p_0^{\diamond(a)}(\cdot,k)$} &
                                                 \multirow{2}{*}{$m=0$,
                                                 $|k|\leq k_0$} & \multirow{2}{*}{$Q\in\mathscr{L}^{(0)}_{0,\mathrm{ln}}$} & Logarithmic  $G^{0(a)}_\diamond$ \\ 
     & & &  compressed to $]0,a[$ \\
  \hline
\end{tabular}
\smallskip
\captionof{table}{\textit{Distinguished solutions of the perturbed
    eigenequation with a prescribed behavior near $0$}. Our convention
  is that a solution $g_m(\cdot,k)$ of \eqref{eigen} (with
  $g=u,p,\dots$) has the same behavior near $0$ as the unperturbed
  solution $g_m^0(\cdot,k)$. We everywhere assume that  $\Re(k)\ge0$. The second column recalls the range of parameters for which the solution $g_m(\cdot,k)$ is defined, the third column gives the conditions on $Q$ that are required in order to define $g_m(\cdot,k)$ and the fourth column recalls the Green's operator used to construct $g_m(\cdot,k)$.}
\end{center}

  \section{Solutions of the perturbed Bessel equation 
   regular near infinity}\label{section:solutions-infty}

 Recall that $w_m^0(\cdot,k)$ is a solution of the unperturbed
 eigenequation which is proportional to $v_m^0(\cdot,k)$ and behaves
 as $\e^{-kx}$ at infinity.
 In this section we construct and study the
 solution to \eqref{eq:a1} with the same asymptotic behavior.
 In the literature, when $m=\pm1/2$ and $Q$ is
 real-valued, this solution is usually called the {\em Jost solution}. We will use the same name in our more general context.

We will assume that $m\in\mathbb{C}$ is arbirary and $\Re(k)\geq0$, or equivalently, $|\mathrm{arg}(k)|\leq\frac\pi2$. The proofs of the results stated in this section are often similar to that of Section \ref{section:solutions}. We will focus on the differences.

Recall that $\mu_k ,\lambda_k ,\eta_{\pm k}$ are defined in \eqref{shorthand}--\eqref{shorthand2}  and that the spaces $L^\infty(]a,\infty[,\phi)$ and $L^\infty_\infty(]a,\infty[,\phi)$ are defined in \eqref{eq:defLinfty}--\eqref{eq:defLinfty12}. 
We use a similar convention as in the previous section: if the operator $\one + G_{\bullet}^{0}Q$ is invertible on $L^\infty(]a,\infty[,\phi)$ for some $a>0$ and some positive measurable function $\phi$ on $]a,\infty[$, where $G^0_\bullet$ is a Green's operator, and if $f :]0,\infty[\to\mathbb{C}$ is such that its restriction to $]a,\infty[$ belongs to $L^\infty(]a,\infty[,\phi)$, then $(\one + G_{\bullet}^{0}Q)^{-1}f$ should be understood as $(\one + G_{\bullet}^{0}Q)^{-1}$ applied to the restriction of $f$ on $]a,\infty[$. Clearly, if in addition $f\in\cN(L_{m^2}^0+k^2)$, then $(\one + G_{\bullet}^{0}Q)^{-1}f$ is a solution to \eqref{eq:a1} on $]a,\infty[$. The unique solution on $]0,\infty[$ which coincides with $(\one + G_{\bullet}^{0}Q)^{-1}f$ on $]a,\infty[$ will be denoted by the same symbol.

To simplify notations, we often write $w^0=w^0_m(\cdot,k)$.

\subsection{The backward Green's operator}
We consider the operator $G_\leftarrow^0Q$. The
results proven here will be used to construct Jost solutions.
Note that $G_\leftarrow^0$ is invariant with respect to the change of sign of
$m$. Therefore, it is enough to assume that $\Re(m)\geq0$.

\begin{lemma}\label{lemma2}
  Let $\Re(k)\geq0$ and $Q\in L^1_\mathrm{loc}]0,\infty[$.
  \begin{enumerate}[label=(\roman*)]
  \item Let $\Re(m)\geq0$, $m\neq0$ and
$\varepsilon+\varepsilon_1\leq\frac12-\Re(m)
   $. Suppose that
\begin{equation}
       \int_0^\infty \mu_k (y)^{1-\varepsilon}|Q(y)|\mathrm{d} y<\infty.
\label{bound7}\end{equation}
    Then 
\begin{equation*}
      G_\leftarrow^0Q :  L^\infty\big( ]0,\infty[, \mu_k ^{\varepsilon_1}\eta_{-k} \big) \to L_\infty^\infty\big(]0,\infty[, \mu_k ^{\varepsilon_1+\varepsilon}\eta_{-k} \big)
    \end{equation*}
    is bounded by $C\times$\eqref{bound7}
    uniformly in $k$.
  \item Let $m=0$, $k\neq0$, $\varepsilon_1+\varepsilon\leq\frac12$,
    $\alpha\geq\beta$  and 
 \begin{equation}
       \int_0^\infty \mu_k (y)^{1-\varepsilon}\lambda_k (y)^\beta|Q(y)|\mathrm{d} y<\infty.
\label{bound8}\end{equation}
    Then 
\begin{equation*}
      G_\leftarrow^0Q : 
 L^\infty\big( ]0,\infty[, \mu_k ^{\varepsilon_1}\lambda_k^\alpha\eta_{-k} \big) \to L_\infty^\infty\big(]0,\infty[, \mu_k ^{\varepsilon_1+\varepsilon}\lambda_k ^{\alpha+1-\beta}  \eta_{-k} \big)
\end{equation*}
    is bounded by $C\times$\eqref{bound8}
    uniformly in $k$.
\end{enumerate}
\end{lemma}

\begin{proof}
The proof is essentially the same as that of Lemma \ref{lemma1}.
\end{proof}

Here is a corollary of the above lemma.

   \begin{corollary}\label{lemma4_app}
     Let $\Re(k)\geq0$ and  $n \in \mathbb{N}$.
\begin{enumerate}[label=(\roman*)]
\item Let $\Re(m)\geq0$, $m\neq0$.
Suppose that $Q \in \mathscr{L}^{(\infty)}_0$.
      Then, for all $a>0$, for all $f\in L^\infty \big(]a,\infty[,\mu_k^{\frac12-\Re(m)}\eta_{-k}\big)$ and $x > a$,
\begin{equation*}
\frac{ \big | (G_\leftarrow^0Q)^n f ( x ) \big | }{ \mu_k (x)^{\frac12-\Re(m)}\eta_{-k}(x) }  \le 
 \frac{C^{n+1} }{n!} \Big ( \int_x^\infty \mu_k (y) |Q(y)| \mathrm{d}y \Big )^n \sup_{y>x}\frac{|f(y)|}{\mu_k (y)^{\frac12-\Re(m)}\eta_{-k}(y)}.
\end{equation*}
\item Suppose $k\neq0$. Let $m=0$.
Suppose that $Q \in \mathscr{L}^{(\infty)}_0$.
      Then, for all $a>0$, for all $f\in L^\infty \big(]a,\infty[,\mu_k^{\frac12} \lambda_k \eta_{-k}\big)$ and $x > a$,
\begin{equation*}
\frac{ \big | (G_\leftarrow^0Q)^n f ( x ) \big | }{ \mu_k (x)^{\frac12}\lambda_k (x)\eta_{-k}(x) }  \le 
 \frac{C^{n+1} }{n!} \Big ( \int_x^\infty \mu_k (y)\lambda_k (y) |Q(y)| \mathrm{d}y
 \Big )^n
 \sup_{y>x}\frac{|f(y)|}{\mu_k (y)^{\frac12}\lambda_k (y)\eta_{-k}(y)}.
\end{equation*}
\end{enumerate}
  Above, $C$ is a constant independent of $n$ and 
  $k$.
\end{corollary}

\begin{proof}
We proceed as in the proof of Corollary \ref{lemma2_app}.

To prove (i)  we use Lemma \ref{lemma2}(i) with $\varepsilon=1$ and $\varepsilon_1 
=-\frac12+\Re(m)$.

To prove (ii) we use Lemma \ref{lemma2}(ii) with $\varepsilon=1$, $\varepsilon_1 
=-\frac12$ and $\alpha=\beta=1$.
\end{proof}

The case $m=0$, $k=0$ is not covered by Lemma
\ref{lemma2} and Corollary \ref{lemma4_app}, because then $\lambda_k $ is ill
defined.
The
following lemma and its corollary work for this case. 

\begin{lemma}\label{lemma2bis}
Let $m=0$, $k=0$ and $Q\in L^1_\mathrm{loc}]0,\infty[$. Let $\varepsilon_1+\varepsilon\leq\frac12$,
    $\beta-\alpha\geq1$  and suppose
   \begin{equation*}
         \int_1^\infty y^{1-\varepsilon} ( 1 + \mathrm{ln}(y) )^\beta|Q(y)|\d y<\infty .
\end{equation*} 
    Then 
\begin{equation*}
      G_\leftarrow^0Q : 
 L^\infty(]1,\infty[, x^{\varepsilon_1}(1+\mathrm{ln}(x))^\alpha \big) \to L_\infty^\infty\big(]1,\infty[, x^{\varepsilon+\varepsilon_1}(1+\mathrm{ln}(x))^{\alpha+1-\beta} \big)
\end{equation*}
    is bounded.
\end{lemma}
     \begin{proof}
It suffices to proceed as in the proof of Lemma \ref{lemma1a}, using that $y\mapsto y^{-\frac12+\varepsilon_1+\varepsilon}$ and $y\mapsto(\mathrm{ln}(y))^{1-\beta+\alpha}$ are decreasing on $]1,\infty[$.
\end{proof}

        \begin{corollary}\label{lemma23_appbis}
          Let $k=0$, $m=0$ and $n \in \mathbb{N}$.
          
\begin{enumerate}[label=(\roman*)]
\item          Suppose that $Q \in \mathscr{L}^{(\infty)}_{1,\mathrm{ln}}$. Then, for all $a>1$, $f\in L^\infty \big(]a,\infty[,x^{\frac12} \big)$ and $a < x$,
\begin{align*}
\frac{| (G_\leftarrow^0Q)^n f(x)|}{x^{\frac12}}
\leq
\frac{C^{n+1} }{n!} \Big ( \int_x^\infty y \big ( 1 + |\mathrm{ln}(y)| \big)
|Q(y)| \mathrm{d}y \Big )^n\sup_{y>x}\frac{|f(y)|}{y^{\frac12}}.
\end{align*}
\item Suppose that $Q \in \mathscr{L}^{(\infty)}_{1,\mathrm{ln}^2}$. Then, for all $a>1$, $f\in L^\infty \big(]a,\infty[,x^{\frac12}|\mathrm{ln}(x)| \big)$ and $a < x$,
\begin{align*}
\frac{| (G_\leftarrow^0Q)^n f(x)|}{x^{\frac12}(1+|\mathrm{ln}(x)|)}
\leq
\frac{C^{n+1} }{n!} \Big ( \int_x^\infty y\big ( 1 + |\mathrm{ln}(y)| \big)^2
|Q(y)| \mathrm{d}y \Big )^n\sup_{y>x}\frac{|f(y)|}{y^{\frac12}(1+|\mathrm{ln}(y)|)}.
\end{align*}
\end{enumerate}
\end{corollary}

\begin{proof}
The proof is the same as that of Corollary \ref{lemma2_app}, applying Lemma \ref{lemma2bis}. To prove (i), we use Lemma \ref{lemma2bis} with $\varepsilon_1=\frac12$, $\varepsilon=0$, $\beta=1$ and $\alpha=0$. To prove (ii), we use Lemma \ref{lemma2bis} with $\varepsilon_1=\frac12$, $\varepsilon=0$, $\beta=2$ and $\alpha=1$.
\end{proof}

\subsection{Jost solutions constructed with the help of the backward Green's operator}\label{sec:sol_infty}

 In this subsection, using the backward Green's operator, we construct the solution to \eqref{eq:a1} which behaves  as $\e^{-kx}$ at infinity.

The next proposition implies Proposition \ref{prop:reginfty0} from the introduction.

\begin{proposition}\label{prop:reginfty}
 Suppose that $Q \in \mathscr{L}^{(\infty)}_0$. Let $m\in\C$ and $\Re(k)\geq0$, $k\neq0$.
Then
\begin{equation*}
w_m(\cdot,k):=(\one+G_\leftarrow^0Q)^{-1}w^0(\cdot,k)
\end{equation*}
is the unique solution in
$AC^1]0,\infty[$ to
 \eqref{eq:a1} such that
\begin{align}\label{eq:f1'2.}
  w_{m} ( x ,k) -
  w_{m}^0 ( x ,k) &=o(\e^{-x\Re(k)}),\\\label{eq:f1'2..}
\partial_x      w_{m} ( x ,k) -
\partial_x    w_{m}^0 ( x ,k) &=o(\e^{-x\Re(k)})
  , \qquad x \to \infty .
\end{align}
Moreover, for all $m\in\C$, we have 
\begin{equation}
  w_{m} ( x ,k) =  w_{-m} ( x ,k) . \label{eq:parityv}
  \end{equation}
\end{proposition}

\begin{proof}
  Let $a>0$. Clearly, $w^0 \in L^\infty( ]a,\infty[,  \mu_k \eta_{-k})$,
Hence, by Corollary \ref{lemma4_app}, 
    \begin{equation}\label{eq:vfixpoint}
      w=(\one+G_\leftarrow^0Q)^{-1}w^0 \in L^\infty( ]a,\infty[,
    \mu_k \eta_{-k})
    \end{equation}   is well defined. As in the proof of Proposition \ref{prop:reg1}, this implies that $w\in AC^1]0,\infty[$ and, using in addition that $G_\leftarrow^0$ is a right inverse of $L_{m^2}$, that $w$ is a solution to \eqref{eq:a1}.

 Next, since
\begin{equation}\label{podo2}
      w-w^0=-G_\leftarrow^0Qw,
    \end{equation}
applying Lemma \ref{lemma2}, we obtain that
       $w-w^0 \in      L_\infty^\infty\big(
    ]a,\infty[,\e^{-x\Re(k)}\big)$.
This proves \eqref{eq:f1'2.}.
Equation \eqref{eq:f1'2..} is proven similarly.

Uniqueness follows exactly as in the proof of Proposition \ref{prop:reg1}. This in turn implies \eqref{eq:parityv} since $w^0_m=w^0_{-m}$.
\end{proof}

Now, for a fixed $Q$, we can study the regularity of Jost solutions with respect to $(m,k)$.

\begin{proposition}\label{prop:reginfty-an}
   Suppose that $Q \in \mathscr{L}^{(\infty)}_0$.
Then for all $x>0$, the maps
\begin{equation*}
\big \{\Re(k)\geq0, \, k\neq0 \big \}\ni (m , k ) \mapsto w_m( x,k ),\,\partial_xw_m(x,k)
\end{equation*}
 are regular.
\end{proposition}

\begin{proof}
     The analyticity and continuity follows as in the proof of Proposition \ref{prop:anal_um2}, using the analyticity of $w^0$ and the map $G_\leftarrow^0Q$ together with Lemma \ref{lemma4_app}. 
\end{proof}

\begin{remark}
  In general, $( m , k ) \mapsto w_m( x ,k)$ does not extend analytically to ${\{ \frac{\pi}{2}\leq |\mathrm{arg}(k)|\}}$ in the same way as for $( m , k ) \mapsto u_m( x,k )$. However, if the condition $Q \in \mathscr{L}^{(\infty)}_0$ is strengthened, assuming
\begin{equation*}
\int_1^\infty \e^{\Lambda y} |Q(y)| \mathrm{d}t < \infty ,
\end{equation*}
for some $\Lambda > 0$, then one can verify that $(m , k ) \mapsto w_m( x,k )$ extends analytically to 
\begin{equation*}
{\{\mathrm{Re}(k) > -\Lambda / 2 , |\mathrm{arg}(k)|<\pi\}}.
\end{equation*}
\end{remark}

Proposition 
\ref{prop:reginfty}
is restricted to $k\neq0$, because the usual short-range condition
is insufficient to cover the zero energy case. Therefore,
 in our analysis most of 
the time we avoid considering Jost solutions for $k=0$. In the remainder of this subsection, we describe a modification of
Proposition 
\ref{prop:reginfty} about the case $k=0$.

It will be convenient to introduce notation for differently normalized
Jost solutions, parallel to the unperturbed case:
\begin{equation}
v_m(x,k):=  \sqrt{\frac\pi{2k}}\Big(\frac{k}2\Big)^{m}w_m(x,k).
\end{equation}

Recall that the unperturbed eigenequation has the following solutions at
$k=0$:
\begin{equation*}
  u_m^0(x,0):=\frac{x^{\frac12+m}}{\Gamma(1+m)},\quad m\in\C;\qquad
  p^0(x,0):=x^{\frac12}\mathrm{ln}(x), \quad m=0.
  \end{equation*}
}
The
following proposition implies Proposition \ref{prop:reginfty+intro} of
the introduction.

\begin{proposition}\label{prop:reginfty+}
  Let $m\in\C$, $k=0$. Suppose that
 $\delta\geq1+2\max(-\Re(m),0)$
  and
  \begin{align*}
Q \in \mathscr{L}^{(\infty)}_{\delta} , & \text{ if } m \neq 0 , \qquad Q \in \mathscr{L}^{(\infty)}_{\delta,\mathrm{ln}} ,  \text{ if } m = 0 . 
  \end{align*}

Then 
\begin{equation*}
 q_{-m}:=(\one+G_\leftarrow^0Q)^{-1}u_{-m}^0(\cdot,0)
\end{equation*}
is the unique solution 
in   $AC^1]0,\infty[$ to     \eqref{eq:a1} for $k=0$ such that, 
\begin{align}\label{uu1a3bis}
  q_{-m}(x)- x^{\frac12-m}&=o(x^{\frac32-\Re(m)-\delta}),\\
\partial_x q_{-m}(x)-\partial_x x^{\frac12-m}&=o(x^{\frac12-\Re(m)-\delta}),\quad x\to\infty. \label{uu1a4bis}
\end{align}
Besides, if $m\neq0$, then
\begin{align}
\lim_{k\to0} v_m(x,k)&=\frac12\Gamma(m)q_{-m}(x)
              .\label{uu2bis}
\end{align}
\end{proposition}

\begin{proof}
We proceed as in the proof of Theorem \ref{prop:reg10}. In particular, the fact that $q_{-m}$ is well-defined follows from Corollary \ref{lemma4_app}(i) in the case $m\neq0$ and Corollary \ref{lemma23_appbis}(i) if $m=0$.

The limit \eqref{uu2bis} follows from \eqref{eq:f1'2.}, \eqref{uu1a3bis} and \eqref{eq:w_m^0}.
\end{proof}

Note that $\frac32-\Re(m)-\delta\leq \frac12-\Re(m)$. Therefore, the
error in \eqref{uu1a3bis} is always of a smaller order
than $x^{\frac12-m}$.

 \begin{proposition}\label{prop:anal_qm}$ $
 \begin{enumerate}[label=(\roman*)]
 \item
Let 
$Q \in \mathscr{L}^{(\infty)}_1.$
  Then for any $x>0$ the maps
 \begin{align}
    \Big\{\Re(m)\geq0,\ m\neq0\Big\}\ni m&\mapsto 
                                               q_{-m}(x),\,
                                           \partial_x
                                           q_{-m}(x)\label{porto1_qm}
 \end{align}
are regular.
If we strengthen the assumption to
$
 Q \in \mathscr{L}^{(\infty)}_{1,\mathrm{ln}} ,
$
then in \eqref{porto1_qm} we can include
$m=0$.
\item
Let $\delta>1$ and
$Q \in \mathscr{L}^{(\infty)}_\delta.$
  Then for any $x>0$ the maps
 \begin{align}
    \Big\{\Re(m)\geq\frac12(1-\delta)\Big\}\ni m&\mapsto 
                                               q_{-m}(x),\,
                                           \partial_x
                                           q_{-m}(x)\label{porto1_qm2}
 \end{align}
are regular.
\end{enumerate}   \end{proposition}
   
\begin{proof}
The proof is similar to that of Proposition \ref{prop:anal_um2}.
\end{proof}

For $k=0$, $m=0$, we can also construct a solution with the same
behavior as
$x^{\frac12}\ln(x)$ at $\infty$, but we need to strengthen the condition on $Q$. Note that Proposition \ref{prop:reginftylog} implies Proposition \ref{prop:reginftylogintro} from the introduction.

\begin{proposition}\label{prop:reginftylog}
  Let $m=0$, $k=0$, $\delta\ge1$, $\beta\ge2$. Suppose that
$
 Q \in \mathscr{L}^{(\infty)}_{\delta,\mathrm{ln}^\beta} $.
Then 
\begin{equation}
 q_{0,\mathrm{ln}}:=(\one+G_\leftarrow^0Q)^{-1}p^0(\cdot,0)
\end{equation}
is the unique solution in
$AC^1]0,\infty[$ to \eqref{eq:a1} for $k=0$  such that
\begin{align}\label{uu1a3bislog}
  q_{0,\mathrm{ln}}(x)- x^{\frac12}\ln(x)&=o(x^{\frac32-\delta}\mathrm{ln}(x)^{2-\beta})),\\
  \partial_x q_{0,\mathrm{ln}}(x)-\partial_x x^{\frac12}\ln(x)&=o(x^{\frac12-\delta}\mathrm{ln}(x)^{2-\beta})),\quad x\to\infty. \label{uu1a4bislog}
\end{align}
\end{proposition}

\begin{proof}
The proof is again similar to that of Theorem \ref{prop:reg10}, using now Corollary \ref{lemma23_appbis}(ii) and Lemma \ref{lemma2bis}.
\end{proof}

\subsection{Asymptotics of  non-principal solutions near \texorpdfstring{$\infty$}{Lg}.}

In this subsection, under the minimal assumptions
$Q\in\mathscr{L}^{(\infty)}_0$, we show that all elements of
$\Ker(L_{m^2}+k^2)$  not proportional to Jost solutions  behave like
non-principal unperturbed solutions near $\infty$. In particular, the
following proposition shows that they 
are not square integrable.

\begin{proposition}\label{cor:u_not_L2}
Let $m\in\C$, $\Re(k)\ge0$ and $k\neq0$. Suppose that $Q \in \mathscr{L}^{(\infty)}_0$.  Let $g$ be a solution of \eqref{eq:a1} linearly independent with $w_m(\cdot,k)$ constructed in Proposition \ref{prop:reginfty}. Then there exists a constant $C \neq0$ such that
\begin{equation}
g( x ) = C  \e^{k x}+o(\e^{\Re(k) x}), \quad x \to \infty.
\label{sert}\end{equation}
\end{proposition}
\begin{proof} Similarly as in the proof of Proposition \ref{cor:u_not_L2a},
assuming that $w\equiv w_{m}( \cdot , k )$ and $W$ are known, we solve the ordinary differential equation
\begin{equation*}
g(x) w'(x) - g'(x) w(x) = W . 
\end{equation*}
The ansatz $g( x ) = \lambda (x) w( x )$
yields
\begin{equation*}
\lambda '(x) w( x )^2  = W .
\end{equation*}
By Proposition \ref{prop:reginfty} and \eqref{eq:equivKbm_infty}, we know that
\begin{equation*}
  w( x ) =  \e^{-  k x }+
o( \e^{-  k x })  , \quad x \to \infty.
\end{equation*}
This implies that there exists $\alpha>0$ such that $w( x ) \neq 0$ for $x \ge \alpha$, and hence 
\begin{align*}
  \lambda (x)-\lambda (\alpha) = \int_\alpha^x \frac{ W }{ w( y )^2 } \mathrm{d} y
  &=  \int_\alpha^x  W\Big(\e^{2ky}+o(\e^{2ky})\Big)
  \mathrm{d} y = \frac{W}{2k}\e^{2kx}+o(\e^{2kx})
  .
\end{align*}
Now
\begin{equation*}
g( x ) = \Big ( \lambda (\alpha) + \int_\alpha^x \frac{ W }{ w( y)^2 } \mathrm{d} y \Big ) w( x ) , 
\end{equation*}
implies the estimate \eqref{sert}.
\end{proof}

\subsection{Global estimates on Jost solutions}

In the sequel we will need an estimate of the Jost solutions constructed in Proposition \ref{prop:reginfty} global in $x$ and $k$, which we prove in this subsection.  

 Since
$w_{m}(\cdot,k)=w_{-m}(\cdot,k)$, we can assume in the next proposition that
$\mathrm{Re}(m)\ge0$.

\begin{proposition}\label{prop:global_estim_Jost}
Let $m\in\C$ be such that $\Re(m)\ge0$. Suppose that $Q \in \mathscr{L}^{(\infty)}_0$.  
We have then the
following estimates on  the solution $w_m(\cdot,k)$ to \eqref{eq:a1},
where  the constant $C$ is  uniform in  $\Re(k)\geq0$, $k\neq0$: for all $x>0$,
\begin{equation*}
| w_m(x,k) | \le 
\left\{ 
\begin{aligned}
&C|k|^{\frac12-\Re(m)} \mu_k (x)^{\frac12-\Re(m)} \eta_{-k}(x) \mathrm{exp}\Big ( C \int_x^\infty \mu_k (y) |Q(y)| \mathrm{d}y \Big ) , \quad m\neq 0; \\
&C|k|^{\frac12} \mu_k (x)^{\frac12} \lambda_k (x) \eta_{-k}(x)  \mathrm{exp}\Big ( C \int_x^\infty \mu_k (y) \lambda_k (y) |Q(y)| \mathrm{d}y \Big ) , \quad m= 0.
\end{aligned}
\right.
\end{equation*}
The same estimates hold for $|\partial_x w_m(x,k) |$, replacing $\mu_k (x)^{\frac12-\Re(m)}$ by $\mu_k (x)^{-\frac12-\Re(m)}$.
\end{proposition}

\begin{proof}
Note that by \eqref{eq:w_m^0} and the estimates recalled in Subsection \ref{subsec:whittaker}, we have
\begin{align*}
&\big | w^0_m(x,k) \big | \lesssim |k|^{\frac12-\Re(m)} \mu_k(x)^{\frac12-\Re(m)} \eta_{-k}(x),\\
  &\big |\partial_x w^0_m(x,k) \big | \lesssim
    |k|^{\frac12-\Re(m)} \mu_k(x)^{-\frac12-\Re(m)} \eta_{-k}(x),  \quad \text{if } m\neq 0, \\
  &\big | w^0_0(x,k) \big | \lesssim |k|^{\frac12} \mu_k(x)^{\frac12}
                                                                                                                               \lambda_k(x) \eta_{-k}(x),\\
&\big |\partial_x w^0_0(x,k) \big | \lesssim |k|^{\frac12}
                                                                                                                                                              \mu_k(x)^{-\frac12} \lambda_k(x) \eta_{-k}(x),  \quad \text{if } m= 0.
\end{align*}
Hence $w^0_m(\cdot,k) \in L^\infty(]0,\infty[,\mu_k ^{\frac12-\Re(m)}\eta_{-k}\big)$ for
$m\neq0$, $w^0_0(\cdot,k) \in
L^\infty(]0,\infty[,\mu_k ^{\frac12}\lambda_k \eta_{-k}\big)$ for $m=0$,
and therefore since
$
w_m(\cdot,k) = (\one+G_\leftarrow^0Q)^{-1} w^0_m(\cdot,k) ,$
we can apply Corollary \ref{lemma4_app}.
\end{proof}

\subsection{Asymptotics of Jost solutions near \texorpdfstring{$0$}{Lg}}

In the next section we will need the asymptotics near $0$ of the Jost
solution $w=w_{m}(\cdot,k)$. It is given by the following proposition.

\begin{proposition}\label{prop:easy_case}
  Let $\Re(k)\geq0$, $k\neq0$. Suppose that 
$
Q \in \mathscr{L}^{(\infty)}_0 .
$
    \begin{enumerate}[label=(\roman*)]
  \item
 If $m\neq 0$, $0\leq\varepsilon<2\Re(m)$ and $Q \in \mathscr{L}^{(0)}_\varepsilon$, then as $x\to0$,
\begin{align}
w_m(x) &=w_m^0(x) + \langle u_m^0 |Qw_m \rangle v_m^0(x) + o( x^{\frac12-\mathrm{Re}(m)+\varepsilon} ) , \label{eq:aaab-1_3} \\
\partial_x w_m(x)&  =\partial_x w_m^0(x)
+ \langle u_m^0 |Qw_m \rangle\partial_x v_m^0(x) + o( x^{-\frac12-\mathrm{Re}(m)+\varepsilon} ) . \label{eq:aaab0_3}
\end{align}
\item If $m\neq0$, $0\leq2\mathrm{Re}(m) \le \varepsilon$,
$Q \in \mathscr{L}^{(0)}_\varepsilon $,
 then as $x\to0$,
\begin{align}
w_m(x) &=w_m^0(x) + \langle u_m^0 |Qw_m \rangle v_m^0(x) - \langle v_m^0 |Qw_m \rangle u_m^0(x) + o( x^{\frac12-\mathrm{Re}(m)+\varepsilon} ) , \label{eq:aaab} \\
\partial_x w_m(x)& = \partial_x w_m^0(x) + \langle u_m^0 |Qw_m \rangle \partial_x v_m^0(x) -  \langle v_m^0 |Qw_m \rangle \partial_x u_m^0(x)
 + o( x^{-\frac12-\mathrm{Re}(m)+\varepsilon} ). \label{eq:aaab2}
\end{align}
 \item If $0\leq\varepsilon 
$, $Q \in \mathscr{L}^{(0)}_{\varepsilon,\mathrm{ln}^2} $,
 and $m=0$,
then as $x\to0$,
\begin{align}
w_0(x) &=w_0^0(x) + \langle u_0^0 |Qw_0 \rangle v_0^0(x) - \langle v_0^0 |Qw_0 \rangle u_0^0(x) + o( x^{\frac12+\varepsilon}  ) , \label{eq:aaabc} \\
\partial_x w_0(x)& = \partial_x w_0^0(x) + \langle u_0^0 |Qw_0
                 \rangle \partial_x v_0^0(x) -  \langle v_0^0
                 |Qw_0 \rangle \partial_x u_0^0(x) + o( x^{-\frac12+\varepsilon}
                 ). \label{eq:aaabc2}
\end{align}
\end{enumerate}
\end{proposition}
\begin{proof}
We only prove the statements for $w_m$. The statements for $\partial_xw_m$ are proven in the same way.
We omit the index $m$ and the argument $k$ in this proof. Recall  \eqref{podo2}:
  \begin{align*}
    w=w^0-G_\leftarrow^0Qw.
  \end{align*}
Suppose first that $m\neq 0$. By Proposition \ref{prop:global_estim_Jost}, we have\begin{align*}
 w(x)&=\mathcal{O}(x^{\frac12-\Re(m)}), \quad \partial_x
       w(x)=\mathcal{O}(x^{-\frac12-\Re(m)}), \quad  x\to0;\\
 w(x)&=\mathcal{O}(\e^{-\Re(k)x})  \quad  \partial_xw(x)=\mathcal{O}(\e^{-\Re(k)x}) , \quad  x\to\infty.
 \end{align*} 
  
 First assume that $2\mathrm{Re}(m) \le  \varepsilon $. Then, since  $Q \in \mathscr{L}^{(\infty)}_0 \cap
 \mathscr{L}^{(0)}_\varepsilon$, it follows from Proposition \ref{prop:global_estim_Jost} that $G_\leftrightarrow^0Qw$ are well-defined. Thus we can write
  \begin{align*}
   G_\leftarrow^0Qw=G_\leftrightarrow^0Qw
    +G_\rightarrow^0Qw.
  \end{align*}
  We can use Lemma \ref{lemma1}(i) with $\varepsilon_1=\frac12-\Re(m)$,
  which gives
\begin{align*} G_\rightarrow^0 Qw 
&=o (x^{\frac12-\Re(m)+\varepsilon}). 
  \end{align*} 
  This proves \eqref{eq:aaab}.

Suppose now that $0\leq\varepsilon<2\Re(m)$. Then
$\langle u^0 |Qw \rangle$ is still well-defined but $\langle v^0 |Qw
\rangle$ is not any more. For $a>0$ and $0<x<a$, we write
    \begin{align*}
    G_\leftarrow^0Qw(x)
&=- \langle u^0 | Qw \rangle  v^0(x) -G_{\bowtie}^{0(a)}Qw(x) + \mathcal{O}( x^{\frac12+\Re(m)} ).
\end{align*}
Applying Lemma \ref{lemmaG0}(i) with $\varepsilon_1=\frac12-\Re(m)$, we obtain
\begin{align*} G_{\bowtie}^{0(a)} Qw 
&=o (x^{\frac12-\Re(m)+\varepsilon}),
\end{align*}
which establishes 
\eqref{eq:aaab-1_3}.

The proof in the case $m=0$ is similar, using Lemma \ref{lemmaG0}(i) with $\varepsilon_1=\frac12$, $\alpha=1$, $\beta=2$.
 \end{proof}
\subsection{Summary of distinguished solutions}\label{subsec:summaryinfty}
In this subsection we recall the distinguished
solutions of the perturbed eigenequation with a prescribed behavior
near infinity constructed in this section. They are described in the
following table, which has the same structure as that of Subsection \ref{subsec:summary0}.

\medskip

\renewcommand{\arraystretch}{1.7}
\begin{center}
\begin{tabular}{|Sc|Sc|Sc|Sc|}
  \hline
  Solution & Parameters & Conditions on $Q$ & Green's operator \\
  \hline
  \begin{tabular}{c}$w_m(\cdot,k)$\\
    $v_m(\cdot,k)$\end{tabular}
    & $m\in\C$, $\Re(k)\ge0$, $k\neq0$ & $Q\in\mathscr{L}^{(\infty)}_0$ & Backward  $G^0_\leftarrow$\\
  \hline
  \multirow{2}{*}{$q_{-m}(\cdot)$} & \multirow{2}{*}{$\Re(m)\ge\frac{1-\delta}{2}$, $k=0$} & $Q\in\mathscr{L}^{(\infty)}_{\delta}$, $\delta\ge1$, if $m\neq0$ & \multirow{2}{*}{Backward  $G^0_\leftarrow$}\\
  & & $Q\in\mathscr{L}^{(\infty)}_{1,\mathrm{ln}}$, if $m=0$ &   \\
  \hline
  $q_{0,\mathrm{ln}}(\cdot)$ & $m=0$, $k=0$ & $Q\in\mathscr{L}^{(\infty)}_{1,\mathrm{ln}^2}$ & Backward  $G^0_\leftarrow$\\
  \hline
\end{tabular}
\smallskip
\captionof{table}{\textit{Distinguished solutions of the perturbed eigenequation with a prescribed behavior near $\infty$}.}
\end{center}

\section{Wronskians}
In this section we study the wronskians of distinguished solutions constructed in the previous two sections.
\subsection{The Jost function}

Suppose that $Q \in \mathscr{L}^{(\infty)}_0 $ and $Q \in \mathscr{L}^{(0)}_\varepsilon$ with $\varepsilon\ge\max\big(0,-2\Re(m)\big)$ if $m \neq 0$, $Q \in \mathscr{L}^{(0)}_{0,\mathrm{ln}}$ if $m=0$. Let $\Re(k)\geq0$, $k\neq0$.
Then, by Propositions \ref{prop:reg1h} and \ref{prop:reginfty}, 
both $ u_{m}( \cdot , k )$ and $ v_{m}( \cdot , k )$ are well defined.
Their  Wronskian
\begin{equation}
  \Wr_{m}(k) :=
  \Wr \big ( v_{m}( \cdot , k ) , u_{m}( \cdot , k ) \big ) 
  ,\label{jost}
\end{equation}
will be called the {\em Jost function}. Assuming in addition that $Q \in \mathscr{L}^{(\infty)}_\delta$ for some $\delta\ge1$, we also set for $k=0$ and $\Re(m)\ge\frac12(1-\delta)$, $m\neq 0$, 
\begin{equation}
  \Wr_{m}(0) :=
\frac12 \Gamma(m)  \Wr \big ( q_{-m}( \cdot ) , u_{m}( \cdot , 0 ) \big ) 
  .\label{jost_k=0}
\end{equation}

Using the regularity properties of $u_m(x,k)$ and $v_m(x,k)$, we obtain the regularity of the map $(m,k)\mapsto \Wr_m(k)$ on suitable domains.

 \begin{proposition}\label{prop:analyticiy_Jost}$ $
   \begin{enumerate}[label=(\roman*)]
   \item
If
 $
Q \in \mathscr{L}^{(0)}_\varepsilon\cap 
 \mathscr{L}^{(\infty)}_0
$
 for some $\varepsilon>0$,
 then  the map $(m,k)\mapsto
   \Wr_{m}(k)$ is regular on
 \begin{align}
  \Big\{\Re(m)\geq-\frac\varepsilon2 \Big \} \times \Big \{
   \Re(k)\geq0,\, k\neq0\Big\}.
 \end{align}
\item
If
$Q \in \mathscr{L}^{(0)}_0\cap  \mathscr{L}^{(\infty)}_0,
$
  then the map $(m,k)\mapsto
   \Wr_{m}(k)$ is regular on
 \begin{align}
  \Big\{\Re(m)\geq0,\ m\neq0  \Big \} \times \Big \{ \Re(k)\geq0,\, k \neq 0\Big\}.
  \label{porto1ka} \end{align}
 \item  If 
$
Q \in \mathscr{L}^{(0)}_{0,\mathrm{ln}} \cap \mathscr{L}^{(\infty)}_{1}$,
 then the map $(m,k)\mapsto
   \Wr_{m}(k)$ is regular on
 \begin{align}  \label{porto1ka_bis} 
  \Big\{\Re(m)\geq0 \Big\} \times \Big \{\Re(k)\geq0\Big\} \setminus \big\{ m=0 \big \} \times \big \{ k = 0 \big \} .
  \end{align}
   \end{enumerate}
   \end{proposition}
\begin{proof}
Analyticity and continuity are consequences of the analyticity and
continuity of the maps $u_{m}( x , k ) $, $w_{m}( x , k )$, $q_m(x)$ and
their derivatives (see Propositions \ref{prop:anal_um2}, \ref{prop:reginfty-an} and \ref{prop:anal_qm}).

Continuity at $k=0$ in \eqref{porto1ka_bis} uses in addition the limit \eqref{uu2bis}.
\end{proof}

The next proposition gives a convenient representation of the Jost function.
\begin{proposition}\label{prop:jost_funct1}
   Suppose that
$
Q \in \mathscr{L}^{(\infty)}_0
$. Let $m\in\mathbb{C}$. Suppose that $Q \in \mathscr{L}^{(0)}_\varepsilon$ with $\varepsilon=\max(0,-2\Re(m))$ if  $m\neq0$, or
 $Q \in \mathscr{L}^{(0)}_{0,\mathrm{ln}^2}$ if $m=0$. 
Then for $\Re(k)\geq0$, $k\neq0$
\begin{align}
  \Wr_{m}(k) =
  1 
  +
  \langle u_{m}^0(\cdot,k) | Q v_{m}(\cdot,k) \rangle. 
  \label{kzl}
\end{align}
\end{proposition}
\begin{proof}
Consider first $m\neq0$.
  By Proposition \ref{prop:reg1h} with $\varepsilon=\max(0,-2\Re(m))$,  for
  $ x\to0$ we have
  \begin{align}\label{uu1a+a}
    u_m(x,k)&=u^0_m(x,k)+o(x^{\frac12+|\Re(m)|}) 
    \\\label{uu1a+b}
      \partial_x  u_m(x,k)&=\partial_x u^0_m(x,k)+o(x^{-\frac12+|\Re(m)|}).
\end{align}
By Proposition \ref{prop:easy_case}(i) with $\varepsilon=0$ we have
\begin{align}
v_m(x,k) &=v^0_m(x,k) + \langle u^0_m |Qv_m \rangle v^0_m(x,k) + o( x^{\frac12-|\mathrm{Re}(m)|} ) , \label{eq:aaab-1_3_bis} \\
\partial_x v_m(x,k)&  =\partial_x v^0_m(x,k)
+ \langle u^0_m |Qv_m \rangle \partial_x v^0_m(x,k) + o( x^{-\frac12-|\mathrm{Re}(m)|} ) . \label{eq:aaab0_3_bis}
\end{align}

Now \eqref{eq:w_m^0}, \eqref{uu1a+a}--\eqref{uu1a+b} and \eqref{eq:aaab-1_3_bis}--\eqref{eq:aaab0_3_bis} yield
  \begin{align*}
\Wr(v_m,u_m;x)
    &=\Wr(v^0_m,u_m;x)\Big(
      1+\langle u^0_m|Qv_m\rangle\Big)
+o(x^{0}).
  \end{align*}
 Moreover it follows from \eqref{uu1a+a} and  \eqref{uu1a+b} that
\begin{align*}
  \lim_{x\to0}\Wr(v_m^0,u_m;x)&=\Wr(v^0_m,u^0_m;x)+o(x^0) = 1 + o(x^0).
  \end{align*}
  Then note that $\Wr(v_m,u_m;x)$
  does not depend on $x$
  and use the definition \eqref{jost} of $\mathscr{W}_m(k)$.

In the case $m=0$,
we use Proposition \ref{prop:easy_case} (iii).
Then we repeat the same arguments, using in addition
\begin{equation*}
\Wr(u^0_0,u_0;x)=\Wr(u^0_0,u^0_0;x)+o(x^0),\quad\Wr(u^0_0,u^0_0;x)=0.
\end{equation*}
This ends the proof of \eqref{kzl}.
\end{proof}
The asymptotic behavior of the Jost function can be deduced from Proposition \ref{prop:jost_funct1} together with the following lemma.

\begin{lemma}\label{prop:jost_funct2}
 Let $m\in\mathbb{C}$. Suppose that $Q \in \mathscr{L}^{(0)}_\varepsilon\cap 
 \mathscr{L}^{(\infty)}_0$ with $\varepsilon=\max(0,-2\Re(m))$ if  $m\neq0$, or
 $Q \in \mathscr{L}^{(0)}_{0,\mathrm{ln}^2}\cap 
 \mathscr{L}^{(\infty)}_0$ if $m=0$. 
Then
\begin{align}
  \langle u_m^0 | Qv_m \rangle = o( |k|^0 )+
 \mathcal{O(}{|k|^{-1+\varepsilon}}), \qquad |k| \to \infty,\quad\Re(k)\geq0. \label{eq:l1}
\end{align}
\end{lemma}
\begin{proof}
Assume that $m \neq 0$.
We have the estimate
\begin{equation}
  |u^0_m(x,k)|\lesssim \mu_k(x)^{\frac12+\Re(m)}\eta_k(x),
  \end{equation}
uniformly in $k$. Next, in the estimate of Proposition \ref{prop:global_estim_Jost}, we first note that the big exponential on the right hand side
is uniformly bounded for large $k$. Therefore, uniformly for large enough $|k|$,
\begin{equation}
  |v_m(x,k)|\lesssim  \mu_k(x)^{\frac12-|\Re(m)|}\eta_{-k}(x).
  \end{equation}
Hence
\begin{align*}
 & \big | \langle u^0_m | Qv_m \rangle \big | \\
 &=  \Big | \int_0^\infty u^0_m(y) Q(y) v_m( y ) \mathrm{d} y \Big | \\
  &\lesssim   \int_0^\infty \mu_k(x)^{1-\varepsilon}
 | Q(x)|\mathrm{d} x\\
  &\lesssim \Big ( \int_0^{\frac1{|k|}} x^{1-\varepsilon} |Q(x)|\mathrm{d} x+
  |k|^{-1+\varepsilon} \int_{\frac1{|k|}}^1|Q(x)|\mathrm{d} x
  +  |k|^{-1+\varepsilon}\int_1^\infty|Q(x)|\mathrm{d} x\Big).
\end{align*}
  First note that
\begin{align*}
 \int_0^{\frac1{|k|}} x^{1-\varepsilon} |Q(x)|\mathrm{d} x = o( 1 ) \quad \text{ and } \quad  \int_1^\infty|Q(x)|\mathrm{d} x = \mathcal{O}( 1 ).
\end{align*}
Therefore, the first term on the right  is $ \smallO{
  |k|^0}$ and the third is $ \mathcal{O}(|k|^{-1+\varepsilon})$.

If $1>\varepsilon\geq0$, then applying Lemma \ref{lm:integr} with $h(y)=y^{1-\varepsilon}$ and
$x=\frac{1}{|k|}$, we obtain
\begin{align*}
  \int_{\frac1{|k|}}^1|Q(x)|\mathrm{d} x = o(|k|^{1-\varepsilon}).
\end{align*}
So the middle term is $ o(
  |k|^0 )$.

If $\varepsilon\geq1$, then
\begin{align*}
  \int_{\frac1{|k|}}^1|Q(x)|\mathrm{d} x = o(|k|^{1-\varepsilon}).
\end{align*}
So the middle term is $ \mathcal{O}(
  |k|^{-1+\varepsilon} )$.

Finally, if $m=0$, the proof is identical, the only difference being that
\begin{align*}
&  \big | \langle u^0_0 | Qv_0 \rangle \big |  \\
&\lesssim   \Big ( \int_0^{\frac1{|k|}} x ( 1 -  | \mathrm{ln}(|k|x) | ) |Q(x)|\mathrm{d} x+
  |k|^{-1}\int_{\frac1{|k|}}^1|Q(x)|\mathrm{d} x
  +|k|^{-1}\int_1^\infty|Q(x)|\mathrm{d} x\Big).
\end{align*}
\end{proof}

We deduce from the previous lemma that, for $\Re(m)>-1$,
  $\Wr_m$ cannot constantly vanish except on a discrete set. Corollary \ref{cor-asim} implies Proposition \ref{prop:wronsk_intro} from the introduction.

\begin{corollary}
In addition to the assumptions of Proposition \ref{prop:jost_funct1},
   suppose that $\Re(m)>-1$.
Then
\begin{align}\label{asim}
 \lim_{|k|\to\infty} \Wr_{m}(k) = 1,
  \quad \Re(k)\geq0.
\end{align}
Therefore,
\begin{equation*}
 \big \{ k \in \mathbb{C} , \, \mathrm{Re}(k)>0 , \, \Wr_{m}(k) = 0
 \big \}
 \quad\text{ is discrete.}
\end{equation*}\label{cor-asim}
\end{corollary}

\begin{proof}
Under the condition $\Re(m)>-1$ the right hand side of \eqref{eq:l1} becomes
$ o( |k|^0)$. Hence
\eqref{asim} follows by \eqref{kzl}.

The fact
that $ \{ k \in \mathbb{C} , \, \mathrm{Re}(k)>0 , \, \Wr_{m}(k) = 0
\} $ is discrete is then a consequence of the analyticity of $\Wr_m$
stated in Proposition \ref{prop:analyticiy_Jost}. 
\end{proof}

Note that Corollary \ref{cor-asim} is the second place in our paper
where the condition $\Re(m)>-1$ appears (see also Theorem \ref{thm:k=0}). This condition will play an
important role in Section \ref{closed_operators} about closed realizations.

\subsection{Wronskians -- refined results}

If $u_{-m}$ is ill-defined, we can often use $u_{-m}^{[n]}$
instead. 

\begin{proposition} \label{wronio}
  Let $\Re(k)\geq0$, $k\neq0$.
  Suppose that  $      Q \in \mathscr{L}^{(0)}_\varepsilon \cap 
 \mathscr{L}^{(\infty)}_0$,  $\varepsilon\ge0$.
Let  $n$ be a nonnegative integer,
  $\frac{\varepsilon}{2}(n+1)\geq \Re(m) \geq0$, $m\not\in\mathbb{N}$.
Then
\begin{align}\label{wrono}
  \Wr\big( u_{-m}^{[n]}(\cdot,k),u_m(\cdot,k)\big)&=\frac{2\sin(m\pi)}{\pi}.
 \end{align}
 Hence there exists a constant $C_m^{[n]}(k)$ such that
 \begin{align}\label{wroni-0}
v_m(\cdot,k)=\frac{\Wr_m(k)\pi}{2\sin(m\pi)}u_{-m}^{[n]}(\cdot,k)+C_m^{[n]}(k)u_m(\cdot,k). 
 \end{align}
\end{proposition}

\begin{proof}
First we check \eqref{wrono}, which follows from \eqref{eq:wronsk_m-m}, using also Propositions \ref{prop:reg1h} and \ref{boundcon}. Then we write
 \begin{align}
v_m(\cdot,k)=B_m^{[n]} (k)u_{-m}^{[n]}(\cdot,k)+C_m^{[n]} (k)u_m(\cdot,k),
 \end{align}
 and take the Wronskian of both sides with $u_m(\cdot,k)$. This allows us to compute
 $B_m^{[n]} (k)$ and yields \eqref{wroni-0}. 
 \end{proof}

\begin{proposition}\label{rk:wronskian}
  Suppose that the assumptions of Proposition \ref{wronio} are
  satisfied.  Then the map
\begin{equation*}
\Big \{ 0 \le \Re(m) \leq\frac{\varepsilon}{2}(n+1) ,\quad m\neq 0\Big \} \times \Big \{ \Re(k) \ge 0 , \, k
\neq 0 \Big \}
\ni(m,k)\mapsto \Wr ( v_m(\cdot,k) , u_{{-}m}^{[n]}( \cdot ,k ) )
\end{equation*}
is regular.
\end{proposition}

\begin{proof}
This follows from Propositions \ref{prop:anal_um^n} and \ref{prop:reginfty-an}.
 \end{proof}

\subsection{Green's functions for  perturbed Bessel operators}

As for every 1-dimensional Schr\"odinger operator, we can define the 
canonical bisolution and various Green's functions for perturbed 
Bessel operators. The solutions that we constructed allow us to
give explicit expressions for these Green's functions.

As usual, when defining Green's operators we will always
assume that $\Re(k)\geq0$ (although sometimes an extension to a larger
domain is possible).  Let
$Q \in \mathscr{L}^{(\infty)}_0 \cap \mathscr{L}^{(0)}_0$ for $m\neq0$,  and 
$
Q \in \mathscr{L}^{(\infty)}_0 \cap \mathscr{L}^{(0)}_{0,\ln}$, for $m=0$.
The canonical bisolution associated with the operator $L_{m^2}+k^2$ is
\begin{equation}
  G_{m^2,\leftrightarrow}(-k^2; x , y ) = 
  \frac{1 }{\Wr_{m}(k)}\big ( v_{m}( x , k ) u_{m}( y , k ) - u_{m}( x , k ) v_{m}( y , k ) \big ),
\label{qrel}\end{equation}
where  $v_m(\cdot,k)$, $u_m(\cdot,k)$ are the solutions to
\eqref{eigen} constructed in Propositions \ref{prop:reg1h} and
\ref{prop:reginfty}, and $\Wr_m(k)$ is the Jost function defined in
\eqref{jost}. The expression \eqref{qrel} is well-defined when $\Wr_m(k)\neq0$ and has a limit at $k=0$. For $m=0$ we can use
\begin{align*}
  G_{0,\leftrightarrow}(-k^2;x,y) 
  &=\frac{1}{\Wr_0(k)}\big(- p_0(x,k)u_0( y,k ) +
                          u_0( x ,k) p_0( y,k )\big).
\end{align*}
 From the canonical bisolution, we can construct in the usual way the \emph{forward Green's
  operator} $G_{m^2,\rightarrow}(-k^2)$ and the \emph{backward Green's operator
$G_{m^2,\leftarrow}(-k^2)$ of $L_{m^2}+k^2$}:
\begin{align}\label{right}
G_{m^2,\rightarrow}(-k^2;x,y)& : = 
\theta(x-y) G^0_{m^2,\leftrightarrow}(-k^2; x , y ),\\\label{left}
G_{m^2,\leftarrow}(-k^2;x,y) & := 
  -\theta(y-x) G^0_{m^2,\leftrightarrow}( -k^2;x , y ).
\end{align}

 Green's operators defined by
specifying boundary conditions at zero and at infinity will be called
{\em two-sided}.
They will often be bounded on  $L^2]0,\infty[$ and
coincide with the resolvents of various closed realizations of 
 $L_{m^2}$. However, they  are of
interest even when they do not define bounded operators and  do
not correspond to closed realizations of 
 $L_{m^2}$.

The most natural two-sided Green's operator corresponds to {\em pure boundary
  conditions}. In the unperturturbed case they are usually called  {\em homogeneous boundary
  conditions}, but in the perturbed case this name seems no
longer appropriate. It can be defined for
$0\leq\varepsilon$,
$Q \in \mathscr{L}^{(\infty)}_0 \cap \mathscr{L}^{(0)}_\varepsilon$,  $m\neq0$,
 $-\frac \varepsilon2\leq\Re(m)$, and 
$
Q \in \mathscr{L}^{(\infty)}_0 \cap
\mathscr{L}^{(0)}_{0,\ln}$, $m=0$. Moreover, if $\Re(m)\leq0$ we neeed to assume
$k\neq0$. Then if  $\Wr_m(k)\neq0$ we set
\begin{equation}    G_{m,\bowtie}( -k^2 ; x , y ) :=
  \frac{1}{ \Wr_{m}(k)    }
\left \{
\begin{array}{lr}
  u_{m}(x,k)v_{m}(y,k)
  , & x < y , \\
 v_{m}(x,k)u_{m}(y,k) , & y<x.
\end{array}
\right.
\label{green-bowtie}
\end{equation}

We also have Green's operators with mixed boundary conditions.
The cleanest situation we have under the assumption
$0\leq\varepsilon$, 
$Q \in \mathscr{L}^{(\infty)}_\varepsilon \cap \mathscr{L}^{(0)}_0$,  $m\neq0$,
$|\Re(m)|\leq\frac \varepsilon2$,
$k\neq0$.
Then if $k\neq0$, $\kappa\in \C\cup\{\infty\}$ and 
$ \Wr_m(k)+\kappa\frac{\Gamma(1-m)}{\Gamma(1+m)}\frac{k^{2m}}{2^{2m}}\Wr_{-m}(k)\neq0$, we define
\begin{align}\label{bowtie-kappa}
  &  G_{m,\bowtie,\kappa}( -k^2 ; x , y )
\\:=&
    \frac{1}{ \big( \Wr_m(k)+\kappa\frac{\Gamma(1-m)}{\Gamma(1+m)}\frac{k^{2m}}{2^{2m}}\Wr_{-m}(k)\big) }
\left \{
\begin{array}{lr}
  ( u_m + \kappa \frac{\Gamma(1-m)}{\Gamma(1+m)}u_{-m} )(x,k)v_{m}(y,k)
  , & x < y , \\
 v_{m}(x,k) (u_m + \kappa \frac{\Gamma(1-m)}{\Gamma(1+m)} u_{-m} )(y,k) , & y<x.
\end{array}
\right.\notag\end{align}
Note that
\begin{equation}
  G_{m,\bowtie,\kappa}( -k^2 ; x , y )=
   G_{-m,\bowtie,\kappa^{-1}}( -k^2 ; x , y ).\end{equation}

Similarly, for
$
Q \in \mathscr{L}^{(\infty)}_0 \cap 
\mathscr{L}^{(0)}_{0,\ln}$,
$m=0$, if
$\nu \Wr_0(k)+\Wr( v_0 (\cdot,k), p_{0}(\cdot,k) )
\neq 0$,
then  we define
\begin{align}\label{bowtie-zero}
&    G^\nu_{0,\bowtie}( -k^2 ; x , y ) \\:=&
    \frac{1}{\big(\nu\Wr_0(k)+\Wr( v_0(\cdot,k) ,
      p_{0} (\cdot,k))\big)
      }
\left \{
\begin{array}{lr}
  ( \nu  u_0 +  p_{0} )(x,k)v_{0}(y,k)
  , & x < y , \\
 v_{0}(x,k) ( \nu  u_0 +  p_{0} )(y,k) , & y<x.
\end{array}
\right.\notag\end{align}

Without the assumption
$\Re(m)\geq-\frac \varepsilon2$ we are not guaranteed the existence of
$u_m$, and hence we are not sure whether Green's function with pure
boundary conditions can be extended. However, we can use the boundary
conditions given by $u_{-m}^{[n]}$. 
Choose $\varepsilon>0$, $Q \in \mathscr{L}^{(\infty)}_\varepsilon \cap \mathscr{L}^{(0)}_0$,  $m\neq0$, and a nonnegative integer $n$.
Then for $-\frac{\varepsilon}{2}(n+1)\leq\Re(m)<0$ if
 $k\neq0$, $\kappa\in \C\cup\{\infty\}$ and 
$
\Wr\big(v_{m}(\cdot,k), u_{m}^{[n]}(\cdot,k)\big)+\kappa\frac{\Gamma(1-m)}{\Gamma(1+m)}\frac{k^{2m}}{2^{2m}}\Wr_{-m}(k)\neq0$,
we can define
\begin{align} \label{bowtie-kappa-n}&
    G_{m,\bowtie,\kappa}^{[n]}( -k^2 ; x , y )\\ := & \frac{1}{ \Wr\big(v_{m}(\cdot,k), u_{m}^{[n]}(\cdot,k)\big)+\kappa\frac{\Gamma(1-m)}{\Gamma(1+m)}\frac{k^{2m}}{2^{2m}}\Wr_{-m}(k) }
\left \{
\begin{array}{lr}
  (  u_{m}^{[n]}+\kappa\frac{\Gamma(1-m)}{\Gamma(1+m)} u_{-m})(x,k)v_{m}(y,k)
  , & x < y , \\
 v_{m}(x,k) ( u_{m}^{[n]} +\kappa\frac{\Gamma(1-m)}{\Gamma(1+m)} u_{-m})(y,k) , & y<x.
\end{array}
                                                                 \right.\notag
\end{align}

\section{Closed realizations of \texorpdfstring{$L_{m^2}$}{Lg}}\label{closed_operators}

In this section we describe realizations of  $L_{m^2}$ as closed operators
on $L^2]0,\infty[$. We will see that under certain assumptions on the perturbation $Q$ one can impose the boundary condition at $0$ in a similar way as in the unperturbed case. If we fix $Q$, it is also natural to organize these  operators in holomorphic families, analogous to the holomorphic families studied in  \cite{DeFaNgRi20_01}.

    In the first  two subsections we recall
 the basics of the theory of 1d Schr\"odinger operators and their boundary conditions.

\subsection{1-dimensional Schr\"odinger operators on the halfline}
\label{1-dimensional Schr\"odinger operators on the halfline}

We will  follow \cite{DeGe19_01}, other references include \cite{DS3,EE}.

Suppose that $]0,\infty[\ni x\mapsto V(x)$ is a function in
$L_\loc^1]0,\infty[$, possibly complex valued. Consider the expression
    \[L:=-\partial_x^2+V(x),\]
viewed as a linear map from $AC^1]0,\infty[$ to $L_\loc^1]0,\infty[$. Let us describe the four most obvious closed realizations
    of $L$ on $L^2]0,\infty[$.

First define
    \begin{align*}
      \cD(L^{\max})
& := \big \{ f \in L^2]0,\infty[ \, \cap \, AC^1]0,\infty[\quad \mid\ L f \in L^2]0,\infty[ \big \} ,\\
  \cD(L^{\min})&:=\text{the closure of }\{f\in\cD(L^{\max})\ \mid\
  f=0\text{ near }0\text{ and }\infty\},\\
    \cD(L^0)&:=\text{the closure of }\{f\in\cD(L^{\max})\ \mid\
    f=0\text{ near }0\},\\
      \cD(L^\infty)&:=\text{the closure of }\{f\in\cD(L^{\max})\ \mid\
      f=0\text{ near }\infty\}.
\end{align*}
Above, 
$\cD(L^{\max})$ is treated as a Hilbert space with the norm given by
$\|f\|_L^2:=\|Lf\|^2+\|f\|^2$.
We define
\begin{equation*}
L^{\max}:=L\big|_{\cD(L^{\max})},\quad L^{\min}:=L\big|_{\cD(L^{\min})},\quad
L^0:=L\big|_{\cD(L^0)},\quad
L^\infty:=L\big|_{\cD(L^\infty)}.
\end{equation*}
Then $L^{\max}$, $L^{\min}$, $L^0$ and $L^\infty$ are closed operators satisfying
\begin{equation*}
 L^{\max}\supset  L^\infty\supset L^{\min},\quad
 L^{\max}\supset  L^0\supset L^{\min}.
 \end{equation*}

Let us quote some general results.
The following proposition is well-known, it is e.g. proven as Proposition 5.15 of \cite{DeGe19_01}:
\begin{proposition}
  If \begin{align}
  \limsup_{c\to\infty}\int_c^{c+1}|V(x)|\mathrm{d} x<\infty,
  \label{bounda1}
\end{align}
then
\begin{equation}\label{bounda}
  L^{\max}=L^\infty,\quad L^0=L^{\min}.
\end{equation}
Thus, there is no need to set boundary conditions at infinity.
\end{proposition}

By  \cite[Theorem 6.12]{DeGe19_01}, we have

\begin{proposition}\label{prio}
Suppose that (\ref{bounda}) holds.
Then we have the following alternative:
\begin{align*}
&    \text{(i) either }
    \cD(L^{\max})/\cD(L^{\min})=0\\
&\quad    \text{ and }
  \dim\big\{f\in \Ker(L-\lambda )\ \mid\ f\text{ is square integrable near }0\big\}\leq1\quad \text{ for all $\lambda \in\C$ };\notag\\
&  \text{(ii) or }
  \cD(L^{\max})/\cD(L^{\min})=2\\
& \quad  \text{ and }
  \dim\{ f\in\Ker(L^{\max}-\lambda )\ \mid\ f\text{ is square integrable near }0\}=2\quad\text{ for all $\lambda \in\C$}. \quad \notag
\end{align*}
\end{proposition}

  Until the end of this subsection
  we suppose that alternative (ii) of Proposition \ref{prio} holds.
  Fix $\lambda\in\C$ and $\xi\in C_\mathrm{c}[0,\infty[$ equal $1$ near
$0$. Then by \cite {DeGe19_01}  we have a direct sum decomposition
\begin{equation}\label{wroo1}
  \cD(L^{\max})=\cD(L^{\min})\oplus\{\xi f\ |\
  f\in\mathcal{N}(L)\},\end{equation}
where $\mathcal{N}(L)$ denotes all functions in $AC^1[0,\infty[$
annihilated by $L$.

We are interested in operators $L^\bullet$ lying ``in the middle'' between
$L^{\min}$ and $L^{\max}$, that is satisfying
\begin{equation*}
  L^{\min}\subset L^\bullet\subset L^{\max},
\end{equation*}
where both inclusions are of codimension $1$.
All such operators correspond to one-dimensional
  subspaces of $ \cD(L^{\max})/\cD(L^{\min})$. To specify
  such a subspace it is enough to choose
  \begin{equation}
    r\in
    \cD(L^{\max}),\quad r\notin\cD(L^{\min}),\label{rdom}
  \end{equation}
  and to define
  \begin{align}\label{roro1}
    \cD(L^r):&=\cD(L^{\min})+\mathbb{C} r,\\ L^r:&=L^{\max}\big|_{\cD(L^r)}. 
  \end{align}

\subsection{Boundary functionals}

    We continue to analyze general 1d Schr\"odinger operators.
  Until the end of this subsection we assume \eqref{bounda}. 
We will give a method to describe boundary conditions which is often
more practical than \eqref{roro1}.

First recall the concept of Wronskian  \eqref{wronskian}. If $f,g\in \cD(L^{\max})$,  then $f,g\in AC^1]0,\infty[\subset C^1]0,\infty[$,
hence their Wronskian at
$x\in]0,\infty[$, denoted
$\Wr(f,g;x)$, is well defined. Interestingly, the
Wronskian extends to the boundary $x=0$, as follows  e.g. from \cite[Theorem 4.4]{DeGe19_01}:
\begin{proposition}\label{roro4}
For $f,g\in\cD(L^{\max})$
\begin{equation*}
\lim_{x\searrow0}\Wr(f,g;x)=:\Wr(f,g;0)
\end{equation*}
exists and defines a continuous
bilinear form.
If in addition (\ref{bounda}) holds, then
\begin{align*}
  \mathcal{D} ( L^{\mathrm{min}} ) &= \Big \{ f \in \mathcal{D} ( L^{\mathrm{max}} ) \ \mid 
  \ \Wr( f , g ; 0 ) = 0 \text{ for all } g \in \mathcal{D} ( L^{\mathrm{max}} ) \Big \}. 
\end{align*}
\end{proposition}

Let us define the {\em boundary space}
\begin{equation*}
  \cB:= \big(\cD(L^{\max})/\cD(L^{\min})\big)',
  \end{equation*}
  where the prime denotes the dual.

  Let $r$ be as in \eqref{rdom}.
Let $\phi\neq0$ be    a boundary functional (that 
is, an element of $\cB$)  such that
$\phi(r)=0$. Obviously, 
  \begin{align*}
    \cD(L^r):&=\{f\in\cD(L^{\max})\mid \phi(f)=0\}
  \end{align*}
is equivalent to \eqref{roro1}.

The boundary functional $\phi$ 
 can be simply written as
\begin{equation}\label{wroo}
  \phi=c\Wr(r,\cdot;0),\end{equation}
where $c\neq0$.
Using \eqref{wroo1} and \eqref{wroo} we obtain
\begin{corollary}
  Suppose that the alternative (ii) of  Proposition \ref{prio}
  holds. Fix $\lambda\in\C$.
  Then we have a natural isomorphism of $\cB$ and $\mathcal{N}(L-\lambda)$:
  \begin{equation}\label{isom}
    \cB=\{\Wr(f,\cdot;0)\ |\ f\in\mathcal{N}(L-\lambda)\}.\end{equation}\end{corollary}

We will say that a function $f\in C^1]0,\infty[$ {\em possesses the
  Wronskian at zero on $\cD(L^{\max})$}
if 
\begin{equation*}
\Wr(f,g;0):=\lim_{x\searrow0}\Wr(f,g;x),\quad g\in\cD(L^{\max} ),
  \end{equation*}
exists. Proposition \ref{roro4} says that each  function in
$\cD(L^{\max})$ possesses the Wronskian at zero  on $\cD(L^{\max})$.

In practice, it may be difficult to make explicit an element $r$ in
$\cD(L^{\max})$ describing the functional $\phi$ by \eqref{wroo}.
Instead, we can often find a
simpler function $r_1$, not necessarily in $\cD(L^{\max})$,
which also possesses the Wronskian at zero  on $\cD(L^{\max})$ and
such that
\begin{equation}
\Wr(r,\cdot;0)=\Wr(r_1,\cdot;0).\end{equation} Then
instead of \eqref{roro1}  the domain of $L^r$ can be equivalently
characterized as :
  \begin{align}\label{roro2}
    \cD(L^r):&=\{f\in\cD(L^{\max})\mid \Wr(r_1,f;0)=0\}. 
  \end{align}

\subsection{The maximal and minimal realizations of \texorpdfstring{$L_{m^2}$}{Lg}}

As everywhere in this paper, we assume
that $]0,\infty[\ni x\mapsto Q(x)$ belongs to $L_\loc^1]0,\infty[$.
For  $m \in \mathbb{C}$, set
\begin{equation*}
  V_{m^2}(x):= \Big(m^2-\frac14\Big)\frac{1}{x^2}  + Q(x).
\end{equation*}
We consider the differential expression
    \begin{equation*}
L_{m^2}:=-\partial_x^2+  V_{m^2}(x).
      \end{equation*}
      as a linear map on $AC^1]0,\infty[$. 
    By applying the definitions of Subsection
    \ref{1-dimensional Schr\"odinger operators on the halfline}, we can introduce
    the closed operators
    $L_{m^2}^{\max}$, $L_{m^2}^{\min}$ such that
\begin{equation*}
\big ( L^{\mathrm{min}}_{m^2}  \big )^\# = L^{\mathrm{max}}_{m^2} , \qquad \big ( L^{\mathrm{max}}_{m^2} \big )^\# = L^{\mathrm{min}}_{m^2} .
\end{equation*}

Until the end of this section we assume that $Q \in \mathscr{L}^{(\infty)}_0$.

\begin{proposition}\label{prop:Lmin} Let $m\in\C$.
  Suppose that
  \begin{align*}
Q \in \mathscr{L}^{(0)}_0 , & \text{ if } m \neq 0 , \qquad Q \in \mathscr{L}^{(0)}_{0,\mathrm{ln}} ,  \text{ if } m = 0 .  
\end{align*}
    Then the  following holds:
\begin{enumerate}[label=(\roman*)]
\item If $1\leq | \mathrm{Re}( m )|$, then $L^{\mathrm{min}}_{m^2} = L^{\mathrm{max}}_{m^2}$. 
\item If $| \mathrm{Re}( m ) |  < 1$, then $\mathcal{D} ( L^{\mathrm{min}}_{m^2} )$ is a closed subspace of $\mathcal{D} ( L^{\mathrm{max}}_{m^2} )$ of codimension $2$.
\end{enumerate}
\end{proposition}
\begin{proof}    
Obviously,
    the condition \eqref{bounda1} holds. Therefore,
 only the boundary conditions at zero matter.

We can assume that $\Re(m)\ge0$ and $\Re(k)\geq0$.
For  $m\neq0$, in
the space $\Ker(L_{m^2}+k^2)$ all
 elements proportional to $u_{m}(\cdot,k)$    behave as
 $x^{\frac12+m}$, 
 and all other elements of $\Ker(L_{m^2}+k^2)$, 
 by Proposition \ref{cor:u_not_L2a}, behave as $x^{\frac12-m}$.
For $m=0$ they behave respectively as  $x^{\frac12}$ and
 $x^{\frac12}\mathrm{ln}(x)$. Both are   square integrable iff
    $|\Re(m)|<1$. Hence
\begin{equation}
\begin{split}\label{pau1}
  &\dim\big\{f\in\Ker(L_{m^2}^{\max}+k^2)\ \mid\ f\text{ is square integrable near }0\big\}\leq1 \text{ for all $\Re(k)\geq0$ } \\
  &\text{ iff }\ |\Re(m)|\geq1;  
  \end{split}
  \end{equation}
and
\begin{equation}
\begin{split}  
  \label{pau2}
  &\dim\{f\in\Ker(L_{m^2}^{\max}+k^2)\ \mid \ f\text{ is square integrable near }0\}=2 \text{ for all $\Re(k)\geq0$}\\
  &\text{ iff }\ |\Re(m)|<1 .
  \end{split}
\end{equation}

Now  we apply Proposition \ref{prio}: $(i)$
  corresponds to (\ref{pau1})   and $(ii)$
corresponds to (\ref{pau2}).
\end{proof}

\subsection{Closed realizations of the unperturbed Bessel operator}

Let us recall the basic theory of closed realizations of $L_{m^2}^0$.
We will essentially follow \cite{DeRi17_01},
except that we will put  the superscript $0$ on the symbols of various operators.

Let $\cB_{m^2}^0$ denote  the boundary space of $L_{m^2}^0$. Below we
give natural bases of $\cB_{m^2}^0$:
\begin{align}
\Wr(x^{\frac12-m},\cdot;0),&\quad  \Wr(x^{\frac12+m},\cdot;0),\quad
                             m\neq0;\label{real4}\\
  \Wr(x^{\frac12},\cdot;0),&\quad
                             \Wr(x^{\frac12}\mathrm{ln}(x),\cdot;0),\quad m=0.\label{real5}
                             \end{align}
Note that for $|\Re(m)|<1$,
\begin{align}
  \Wr(x^{\frac12-m},x^{\frac12+m})&=2m, \label{eq:wronskian_mneq0} \\
  \Wr(x^{\frac12},x^{\frac12}\mathrm{ln}(x))&=1, \label{eq:wronskian_m=0}
\end{align}
which implies the linear independence of \eqref{real4} and \eqref{real5}.

Let us describe the basic families of closed realizations of 
Bessel operators. We will use two kinds of definitions of their
domains. 
In what follows,  $\xi$ is a smooth compactly supported function equal to $1$ near $x=0$.

We have the family of realizations with pure boundary conditions
defined for  $\Re(m)>-1$:
\begin{align}\label{dom1}
  \Dom(H_{m}^0)  :=& \, \cD(L_{m^2}^0)+\mathbb{C} x^{\frac12+m}\xi(x)\\
=& \, \Big\{f\in \Dom(L_{m^2}^{0,\max})\mid 
                 \, \Wr(x^{\frac12+m},f;0)=0\Big\},\label{wron1}
  \\\label{real1}
H_{ m}^0:=& \, L_{ m^2}^0\big|_{\Dom(H_{m}^0)}.
\end{align}

We have also two families with mixed boundary conditions: The first is the generic family defined for $-1<\Re(m)<1$,
 $m\neq0$,
$\kappa\in\C\cup\{\infty\}$:
\begin{align}\label{dom2}
  \Dom(H_{ m,\kappa}^0) &:
                          =\Dom(L_{m^2}^{0,\min})+\mathbb{C}\big(x^{\frac12+m}+
                          \kappa x^{\frac12-m}\big)\xi(x)\\
&= \Big\{f\in \Dom(L_{m^2}^{0,\max})\mid
\Wr\big(x^{\frac12+m} + \kappa x^{\frac12-m},f;0
\big)=0\Big\},\label{wron2}
  \\\label{real2}
\Dom(H_{m,\infty}^0) & := \Dom(H_{ -m}^0),\quad
H_{m,\kappa}^0:=L_{ m^2}^0\big|_{\Dom(H_{m,\kappa}^0)}.
\end{align}
The second family corresponds to $m=0$ and depends on
$\nu\in\C\cup\{\infty\}$:
\begin{align}\label{dom3}
  \Dom(H_{0}^{0,\nu}) & :
                         =\Dom(L_{m^2}^{0,\min})+\mathbb{C}\big(x^{\frac12}\mathrm{ln}(x)
                        + \nu x^{\frac12}\big)\xi(x)\\
& = \Big\{f\in \Dom(L_{0}^{0,\max})\mid
\Wr\big(\nu x^{\frac12} + x^{\frac12}\mathrm{ln}(x),f;0
\big)=0\Big\},\label{wron3}
  \\\label{real3}
\Dom(H_{ 0}^{0,\infty}) &: = \Dom(H_{0}^{0}),\quad
H_{ 0}^{0,\nu}:=L_{0}^0\big|_{\Dom(H_{0}^{0,\nu})}.
\end{align}
The families of closed operators defined in \eqref{real1},
\eqref{real2} and \eqref{real3} are clearly independent of the cutoff $\xi$. They
are holomorphic with respect to the parameters $m,\kappa,\nu$. They are
situated between $L_{m^2}^{0,\min}$ and  $L_{m^2}^{0,\max}$.

\subsection{Boundary functionals for perturbed Bessel operators}
\label{bounda-sec}

 In this subsection, as well as in Subsections \ref{ss0}, \ref{subsec:mixedI} and \ref{subsec:mixedII}, we 
analyze boundary conditions near zero and the corresponding closed
realizations of perturbed Bessel operators.  For definiteness, throughout these four 
subsections we assume that $Q \in \mathscr{L}^{(\infty)}_0$.

It does not seem practical to
define   boundary conditions for perturbed Bessel operators
analogously as in \eqref{real1}, \eqref{real2} and \eqref{real3}. In fact, even after imposing
stronger conditions on $Q$, such as in Proposition \ref{prop:reg1}, it
is not easy to describe explicitly sufficiently well  the behavior of
elements in $\cD(L^{\mathrm{max}}_{m^2})$  near zero.
In particular,
conditions of Theorem \ref{prop:reg10}
 in general do not allow us to conclude that $x^{\frac12\pm m}\xi(x)$ belongs
 to  $\cD(L_{m^2}^{\max})$ and $x^{\frac12}\mathrm{ln}(x)\xi(x)$ to
    $\cD(L_{0}^{\max})$.

    Fortunately, we can use the method of \eqref{wron1}, \eqref{wron2}
    and \eqref{wron3} involving  the Wronskian at $0$. 
    The precise choice of a boundary functional is in general
    more complicated than in the unperturbed case, as we explain
    below.
    
    Let us first describe some properties of the minimal domain.
In the next proof, we denote by $G_\rightarrow=G_{m^2,\rightarrow}(-k^2)$ the forward Green's
operator associated to $L_{m^2}$ defined in \eqref{right}.

\begin{lemma}\label{lemko}$ $ 
    \begin{enumerate}[label=(\roman*)]
    \item
    Let $0\leq\Re(m)<1$, $m\neq0$,
  $Q \in \mathscr{L}^{(0)}_0$ and   $h\in\cD(L_{m^2}^{\min})$. Then
  \begin{align}\label{iuy1}
    h(x)=o(x^{\frac32}),&\quad 
                          \partial_xh(x)=o(x^{\frac12}).
  \end{align}
  \item Let $m=0$, 
    $Q \in \mathscr{L}^{(0)}_{0,\mathrm{ln}}$ and
    $h\in\cD(L_{m^2}^{\min})$. Then
    \begin{align}
\label{iuy2}
    h(x)=o(x^{\frac32}\mathrm{ln}(x)),&\quad 
    \partial_xh(x)=o(x^{\frac12}\mathrm{ln}(x)).\end{align}\end{enumerate}
\end{lemma}

\begin{proof}
Let $h\in\cD(L_{m^2}^{\min})$ and $c>0$. Let $\Re(k)\ge0$ be such that $\Wr_m(k)\neq0$ ($k$ exists by Corollary \ref{cor-asim}). Using e.g. \cite[Proposition 7.3]{DeGe19_01} we know that there exists
$f\in L^2]0,c[$ such
that
\begin{equation*}
h\big|_{]0,c[}=G_\rightarrow(-k^2) f\big|_{]0,c[}.
  \end{equation*}
  By e.g. \cite[Proposition 7.5]{DeGe19_01}
\begin{align*}
  |G_\to(-k^2) f(x)|&\leq \frac12\big(|w_m(x,k)|\|u_m(\cdot,k)\|_x+
                      | u_m(x,k)|\| 
                      w_m(\cdot,k)\|_x\big)\|f\|_x,\\
    |\partial_xG_\to(-k^2) f(x)|&\leq \frac12\big(|\partial_x w_m(x,k)|\|u_m(\cdot,k)\|_x+
                      |\partial_x u_m(x,k)|\| 
                      w_m(\cdot,k)\|_x\big)\|f\|_x,
  \end{align*}
where $\|f\|_x:=(\int_0^x|f(y)|^2\d y)^{\frac12}$. By Proposition \ref{prop:global_estim_Jost}, for small $x$,
\begin{align*}
&| w_m(x,k) | \lesssim x^{\frac12-\Re(m)} , \quad | \partial_x w_m(x,k) | \lesssim x^{-\frac12-\Re(m)} , \quad \text{if }m \neq 0 , \\
&| w_0(x,k) | \lesssim x^{\frac12}|\mathrm{ln}(x)| , \quad | \partial_x w_m(x,k) | \lesssim x^{-\frac12}|\mathrm{ln}(x)| , \quad \text{if }m = 0 ,
\end{align*}
while, by Proposition \ref{prop:reg1h},
\begin{align*}
&| u_m(x,k) | \lesssim x^{\frac12+\Re(m)} , \quad | \partial_x u_m(x,k) | \lesssim x^{-\frac12+\Re(m)} .
\end{align*}
This yields the estimates \eqref{iuy1}--\eqref{iuy2}.
\end{proof}

\pagebreak[2]

\begin{lemma}\label{lemma:wronsk}$ $
\begin{enumerate}[label=(\roman*)]
\item
Let $0\leq\Re(m)<1$, $m\neq0$ and $Q\in\mathscr{L}^{(0)}_0$. Let $g \in AC^1]0,\infty[$ be such that $g(x)=o(x^{\frac12+\Re(m)})$ and $\partial_xg(x)=o(x^{-\frac12+\Re(m)})$. Then
\begin{equation*}
\Wr(g,f;0)=0 \quad \text{for all } f \in \Dom ( L^{\mathrm{max}}_{m^2} ).
\end{equation*}
\item The same holds if $m=0$,
   $Q\in\mathscr{L}^{(0)}_{0,\mathrm{ln}}$ and $g \in
  AC^1]0,\infty[$ satisfies $g(x)=o(x^{\frac12} | \mathrm{ln}(x)
  |^{-1} )$ and $\partial_xg(x)=o(x^{-\frac12} | \mathrm{ln}(x) |^{-1}
  )$.
\end{enumerate}
\end{lemma}
\begin{proof} Fix any $k\in\C$. Every $f \in \Dom (
  L^{\mathrm{max}}_{m^2} )$ can be written as $f=\xi f_0+f_1$ where
  $f_0\in\mathcal{N}(L_{m^2}^{\max}+k^2)$ and $f_1
      \in\cD(L_{m^2}^{\min})$. Now (i)
follows by  Lemma \ref{lemko}(i) 
and Proposition \ref{cor:u_not_L2a}(i), and (ii) follows by
 Lemma \ref{lemko}(ii) 
and Proposition \ref{cor:u_not_L2a}(ii).
\end{proof}

In what follows, we will denote by $\cB_{m^2}$ the space of boundary 
functionals of $L_{m^2}$ with a given perturbation $Q$. 
We will describe
convenient bases of
$\cB_{m^2}$.
In other words, we will find
pairs of linearly independent functionals on $\cD(L_{m^2}^{\max})$
that vanish on  $\cD(L_{m^2}^{\min})$.

Note that cases (i)(a) and (iii)(a) of the next theorem have quite weak assumptions on the
perturbation, however  their non-principal boundary functionals depend
on an arbitrary parameter $a$.

\begin{theorem}
  \label{thm-bc}$ $
\begin{enumerate}[label=(\roman*)]
\item Let $0<\Re(m)<1$.\begin{enumerate}\item Assume $Q \in \mathscr{L}^{(0)}_0$.
  Let $a>0$ be small enough as in Proposition
  \ref{prop:ubowtie}.
  Then
\begin{equation}\label{ia}
  \Wr( u_{-m}^{\bowtie(a)}(\cdot,0),\cdot;0),\quad  \Wr(x^{\frac12+m},\cdot;0) ,
  \end{equation}
  is a basis of $\cB_{m^2}$.
\item
  Suppose that the assumption 
  is strengthened to $Q \in
  \mathscr{L}^{(0)}_\varepsilon$ for some $\varepsilon>0$ (but we drop
  the assumption on $a$). Let $n$ be a non-negative integer such that $\Re(m)\le\frac{(n+1)\varepsilon}{2}$. Then
\begin{equation}\label{ib}
\Wr(  u_{{-}m}^{0[n]},\cdot;0),\quad \Wr(x^{\frac12+m},\cdot;0)
\end{equation}
is a basis of $\cB_{m^2}$.
\item 
 If we assume
  $0<\varepsilon
  $, 
 $0<\Re(m)\leq\frac{\varepsilon }{2}$ and $Q \in 
 \mathscr{L}^{(0)}_\varepsilon$, then 
\begin{equation}\label{ic}
\Wr(x^{\frac12-m},\cdot;0),\quad\Wr(x^{\frac12+m},\cdot;0)
\end{equation}
 is a basis of $\cB_{m^2}$.
\end{enumerate}
\item Let $\Re(m)=0$, $m\neq0$. Assume $Q \in \mathscr{L}^{(0)}_0$. Then
\begin{equation}
  \Wr(x^{\frac12-m},\cdot;0),\quad  \Wr(x^{\frac12+m},\cdot;0)
\label{iia}  \end{equation}
  is a basis of $\cB_{m^2}$.
\item Let $m=0$.
  \begin{enumerate}\item
    Assume $Q \in \mathscr{L}^{(0)}_{0,\mathrm{ln}}$.   Let $a>0$ be small enough as in Proposition
  \ref{prop:ubowtie}.
 Then
\begin{equation}
\Wr(p_0^{\diamond(a)},\cdot;0),\quad \Wr( u_0(\cdot,0) ,\cdot;0)
\label{iiia}\end{equation}
is a basis of $\cB_{0}$.
\item
Suppose that the assumption on $Q$ is strengthened to $Q \in
\mathscr{L}^{(0)}_{0,\mathrm{ln}^2}$ (but we drop the condition on $a$). Then
\begin{equation}
\Wr(p_0,\cdot;0),\quad \Wr(x^{\frac12},\cdot;0)
\label{iiib}\end{equation}
is a basis of $\cB_{0}$.
\item
If the assumption is further strengthened to $Q \in \mathscr{L}^{(0)}_\varepsilon$ for some $\varepsilon>0$, then
\begin{equation}
\Wr(x^{\frac12}\mathrm{ln}(x),\cdot;0),\quad \Wr(x^{\frac12},\cdot;0)
\label{iiic}\end{equation}
is a basis of $\cB_{0}$.
\end{enumerate}\end{enumerate}
\end{theorem}

\begin{proof}
(i)  Recall that  in Propositions \ref{prop:reg1h} and
\ref{prop:ubowtie} we introduced the
functions 
$
u_m(\cdot,0), u_{-m}^{\bowtie(a)}(\cdot,0)$
spanning $\mathcal{N}(L_{m^2})$.
Therefore,  by \eqref{isom}, 
\begin{equation}\label{wrow0}
  \Wr( u_{-m}^{\bowtie(a)}(\cdot,0),\cdot;0),\quad
  \Wr(u_m(\cdot,0),\cdot;0)\end{equation}
 is a basis of $\cB_{m^2}$.
 Using \eqref{uu1a3}--\eqref{uu1a4} and  Lemma \ref{lemma:wronsk}(i)
we see that
\begin{equation}\label{wroni}
  \Wr( u_m(\cdot,0),\cdot;0)=  \frac{1}{\Gamma(m+1)}\Wr(x^{\frac12+m},\cdot;0).\end{equation}
Therefore, we can replace $u_{ m}(\cdot,0)$ by $x^{\frac12+ m}$,
obtaining the basis \eqref{ia}.

Assume now that $Q \in \mathscr{L}^{(0)}_\varepsilon$ for some
$\varepsilon>0$
 and suppose  that $n$ is a positive integer,  $0<\varepsilon<\frac{2}{n+1}$, 
 $0<\Re(m)\leq\frac{\varepsilon (n+1)}{2}$ and $Q \in
 \mathscr{L}^{(0)}_\varepsilon$.  By Proposition \ref{comui}, we have
 \begin{align*}
   \Wr( u_{{-}m}^{\bowtie(a)}(\cdot,0),\cdot;0)
   =\Wr( u_{{-}m}^{0[n]},\cdot;0)+c_{m}^{(a)[n]}\Wr(u_m(\cdot,0),\cdot;0),
   \end{align*}
for some constant $c_{m}^{(a)[n]}$ depending on the parameters. Therefore,
\eqref{ib} is also a basis of $\cB_{m^2}$.

If $2\Re(m)\leq\varepsilon$, then we can take $n=0$:
\begin{equation}
  u_{{-}m}^{0[0]}(\cdot,0) =
  u_{-m}^0(\cdot,0)=\frac{x^{\frac12-m}}{\Gamma(1-m)}
\end{equation}
Then
 we can replace $\Wr( u_{-m}^{0[0]}(\cdot,0),\cdot;0)$ with
 $\Wr(x^{\frac12-m},\cdot;0)$
obtaing the basis \eqref{ic}.

\vspace{0,2cm}

(ii) By
Proposition \ref{prop:reg1h},  both $ u_m(\cdot,0)$ and
$u_{-m}(\cdot,0)$ are well defined and span $\mathcal{N}(L_{m^2})$.  By   \eqref{isom}, 
we obtain a basis of $\cB_{m^2}$
\begin{equation*}
  \Wr(u_m(\cdot,0),\cdot;0),\quad  \Wr(u_{-m}(\cdot,0),\cdot;0). 
  \end{equation*}
  Now
  \eqref{wroni} is still valid so that
  we can replace $u_{\pm m}(\cdot,0)$ by $x^{\frac12\pm m}$, obtaining
  the basis \eqref{iia}
\vspace{0,2cm}

(iii) In Propositions \ref{prop:reg1h} and \ref{prop:pdiamonda} we constructed
functions $ u_0(\cdot,0)$ and $p_0^{\diamond(a)}(\cdot,0)$ spanning
$\mathcal{N}(L_0)$.
It follows from   \eqref{isom} that \eqref{iiia}
is a basis of $\cB_{0}$.

 If we strengthen the assumption to 
 $Q\in\mathscr{L}^{(0)}_{0,\mathrm{ln}^2}$, then in Proposition \ref{prop:reg1} we defined
 $p_0(\cdot,0)\in\mathcal{N}(L_0)$.  The functions $p_0(\cdot,0)$ and $u_0(\cdot,0)$
 span $\mathcal{N}(L_0)$.
 Therefore, by   \eqref{isom}, 
 \begin{equation*}
   \Wr(u_0(\cdot,0),\cdot;0),\quad  \Wr(p_{0}(\cdot,0),\cdot;0). 
 \end{equation*}
 is a basis of $\cB_{0}$.
 Besides, by Proposition \ref{prop:reg1_0_bis} and  Lemma \ref{lemma:wronsk}(ii)
we have
 \begin{align*}
&   \Wr( u_0(\cdot,0),\cdot;0)=  \Wr(x^{\frac12},\cdot;0).
   \end{align*}
Therefore, \eqref{iiib} 
is a basis of $\cB_{0}$.

If the assumption is further strengthened to 
$Q\in\mathscr{L}^{(0)}_{\varepsilon}$ with $\varepsilon>0$, then
 by  Proposition \ref{prop:reg1} and   Lemma \ref{lemma:wronsk}(ii)
we have
 \begin{align*}
& \Wr(p_0,\cdot;0) =  \Wr(x^{\frac12} \mathrm{ln}(x),\cdot;0).
   \end{align*}
Therefore, \eqref{iiic}
is a basis of $\cB_{0}$.
\end{proof}

\begin{remark}
Let $Q(x) = - \frac{ \beta }{ x }\one_{]0,1]}(x)$ as in Remark \ref{rk:coulomb}. Taking $n=1$ in the previous theorem, it follows from that remark that for $0<\Re(m)<1$,
\begin{equation*}
\Wr\Big(x^{\frac12-m}\Big ( 1 -  \frac{\beta x}{1-2m} \Big ),\cdot;0\Big),\quad \Wr(x^{\frac12+m},\cdot;0)
\end{equation*}
 forms a basis of $\cB_{m^2}$.
\end{remark}

The next table summarizes the bases of $\cB_{m^2}$ that we constructed, depending on the values of $-1<\Re(m)<1$ and on the conditions on $Q$. 

\medskip

\renewcommand{\arraystretch}{1.7}
\begin{center}
\begin{tabular}{|Sc|Sc|Sc|Sc|Sc|}
  \hline
 & \multirow{2}{*}{$-1<\Re(m)<0$} & \multicolumn{2}{c|}{$\Re(m)=0$} & \multirow{2}{*}{$0<\Re(m)<1$} \\
 \cline{3-4}
 & & $m\neq0$ & $m=0$ & \\
  \hline
  $Q \in \mathscr{L}^{(0)}_0$ & $x^{\frac12-m}$, $u_{m}^{\bowtie(a)}$ & $x^{\frac12-m}$, $x^{\frac12+m}$ & ? & $u_{-m}^{\bowtie(a)}$, $x^{\frac12+m}$ \\
  \hline
    $Q \in \mathscr{L}^{(0)}_{0,\mathrm{ln}}$ & $x^{\frac12-m}$, $u_{m}^{\bowtie(a)}$ & $x^{\frac12-m}$, $x^{\frac12+m}$ & $p_0^{\diamond(a)}$, $x^{\frac12}$ & $u_{-m}^{\bowtie(a)}$, $x^{\frac12+m}$ \\
    \hline
        $Q \in \mathscr{L}^{(0)}_{0,\mathrm{ln}^2}$ & $x^{\frac12-m}$, $u_{m}^{\bowtie(a)}$ & $x^{\frac12-m}$, $x^{\frac12+m}$ & $p_0$, $x^{\frac12}$ & $u_{-m}^{\bowtie(a)}$, $x^{\frac12+m}$ \\
    \hline
      \multirow{2}{*}{$Q \in \mathscr{L}^{(0)}_{\varepsilon}$} &
                                                        $-\frac{(j+1)\varepsilon}{2} \le \Re(m)
               \leq                                                  0$ & \multirow{2}{*}{$x^{\frac12-m}$, $x^{\frac12+m}$} & \multirow{2}{*}{$x^{\frac12}\mathrm{ln}(x)$, $x^{\frac12}$}  & $0 \leq\Re(m) \le  \frac{(j+1)\varepsilon}{2}$ \\
       \cline{2-2} \cline{5-5}
 & $x^{\frac12-m}$, $ u_{m}^{0[j]}$ &  &  & $ u_{-m}^{0[j]}$, $x^{\frac12+m}$ \\
    \hline
\end{tabular}
\medskip
\captionof{table}{\textit{Bases of the boundary space $\cB_{m^2}$ of $L_{m^2}$}. In each case, we write $g_1$, $g_2$ if $\Wr(g_1,\cdot;0)$, $\Wr(g_2,\cdot;0)$ is a basis of $\cB_{m^2}$. To shorten notations, we write $u_{m}^{\bowtie(a)}$ for $u_{m}^{\bowtie(a)}(\cdot,0)$ and likewise for other functions. Note that $u_{m}^{\bowtie(a)}$ and $p_0^{\diamond(a)}$ depend on an arbitrary parameter $a$. Each line corresponds to a condition on $Q$, from the minimal one $Q \in \mathscr{L}^{(0)}_0$ to the strongest one $Q \in \mathscr{L}^{(0)}_{\varepsilon}$, with $\varepsilon>0$ (beside the condition $Q \in \mathscr{L}^{(\infty)}_0$ which we everywhere assume). In the last line, $j\in \{ 0 , \dots , n \}$, where $n$ is the smallest nonnegative integer such that $\frac{(n+1)\varepsilon}{2} \ge 1$.}
\end{center}

As mentioned above, in the last line, for $j=0$, $ u_{m}^{0[0]}$ can be replaced by $x^{\frac12+m}$ and $ u_{-m}^{0[0]}$ can be replaced by $x^{\frac12-m}$. For growing values or $\Re(m)>0$, the picture is then that, to pass from the region $R_j:=\{\frac{j\varepsilon}{2} < \Re(m) \le  \frac{(j+1)\varepsilon}{2}\}$ to $R_{j+1}$, one needs to add a further term to $ u_{-m}^{0[j]}$ in order to still have an element of $\cB_{m^2}$. Of course, we could also use $ u_{-m}^{0[n]}$ in the whole region $0 < \Re(m) < 1$ for any $n$ such that $\frac{(n+1)\varepsilon}{2} \ge 1$, but then all the terms of order $o(x^{\frac12+m})$ in $ u_{-m}^{0[n]}$ are irrelevant.

\subsection{The perturbed Bessel operator with pure boundary conditions}
\label{ss0}

In this subsection we introduce the most natural family
of perturbed Bessel operators.  It is parallel to what in the 
unperturbed case was called the family of homogeneous 
Bessel operators. (In the perturbed case the homogeneity is no longer 
true, therefore the name has to be changed).

Let $m\in\mathbb{C}$, $-1<\Re(m)$. We assume that
    \begin{align*}
&  Q \in \mathscr{L}^{(0)}_\varepsilon, \quad 
             m\neq0,\quad
\varepsilon=\max\big(0,-2\Re(m)\big);
\\
  & Q \in \mathscr{L}^{(0)}_{0,\mathrm{ln}^2} ,\quad m=0.
    \end{align*} 
We can then define
\begin{align*}
\Dom(H_{ m}) & := \Big\{f\in \Dom(L_{m^2}^{\max})\mid
\, \Wr(x^{\frac12+m},f;0)=0\Big\},\\
H_{ m}&:=L_{m^2}\big|_{\Dom(H_{m})},
\end{align*}
which we will call 
the {\em perturbed Bessel operator with pure boundary conditions}.

Using Theorem \ref{thm-bc} we see that the operator
  $H_{m}$ is
  closed,
\begin{align*}
  L^{\mathrm{min}}_{m^2} = H_{m} = L^{\mathrm{max}}_{m^2},\quad \Re(m)\geq1;\\
  L^{\mathrm{min}}_{m^2} \subset H_{m} \subset L^{\mathrm{max}}_{m^2} ,\label{pqo}\quad
 -1< \Re(m)<1;
\end{align*}
and both inclusions are of codimension $1$.

\begin{proposition}\label{prop:Hmres}
Suppose that the assumptions on $Q$ stated at the beginning of this subsection
hold.  
Let $\Re(k)>0$. Then 
$k^2\notin \sigma(H_{m})$
if and only if $\Wr_{m}(k)\neq0$. Moreover,
the operator  
$  G_{m,\bowtie}(-k^2)$ defined in \eqref{green-bowtie}  is
then bounded  
  and 
  \begin{equation*}
  G_{m,\bowtie}(-k^2)=(k^2+  H_{m})^{-1}.
\end{equation*}
\end{proposition}
  
\begin{proof}
 We use \cite[Proposition 7.7]{DeGe19_01} together with the asymptotic
 behavior near $0$ of $u_m(x,k)$ established in
 Propositions \ref{prop:reg1h}.
 \end{proof}

In
the following theorem we fix a perturbation $Q$ and consider the
operator valued family
$   H_{m}$.
In the definition of regularity we use the
concept of a holomorphic family of closed operators
recalled in 
Appendix \ref{Holomorphic families of closed operators}.
 Moreover, the continuity of a family of closed operators should be understood in the weak resolvent sense.

 \begin{theorem}\label{thm:hol_Hma}$ $
    \begin{enumerate}[label=(\roman*)]
   \item 
Let $2>\varepsilon>0$ and suppose that
$  Q \in \mathscr{L}^{(0)}_\varepsilon$.
 Then
 \begin{align}\label{kak1}
  \Big\{-\frac{\varepsilon}{2} \leq \Re(m)\Big\}\ni m&\mapsto
   H_{m}
 \end{align}
 is regular.
\item
  Let
$Q \in \mathscr{L}^{(0)}_0. $
  Then
 \begin{align}
  \Big\{\Re(m)\geq0,\, m\neq0\Big\}\ni m&\mapsto
   H_{m},\label{porto1k}
 \end{align}
 is regular.
If we strengthen the assumption  to
$Q \in \mathscr{L}^{(0)}_{0,\mathrm{ln}^2} $,
then we can include $m=0$ in \eqref{porto1k}.
   \end{enumerate}
   \end{theorem}

\begin{proof}
In view of Proposition \ref{prop:Hmres}, we can proceed as in the proof of Theorem 3.10 in \cite{DeFaNgRi20_01}: It suffices to use Propositions \ref{prop:anal_um2}, \ref{prop:reginfty-an} and \ref{prop:analyticiy_Jost}.
\end{proof}

Note that for $\Re(m)\geq1$ the operator $H_m$ is the unique
 closed realization
of $L_{m^2}$. Theorem \ref{thm:hol_Hma} shows that the holomorphic
function
\begin{equation*}
  \{\Re(m)>1\}\ni m\mapsto H_m
  \end{equation*}
has an analytic continuation to a larger region,
\eqref{kak1} or \eqref{porto1k}, where the width of the additional strip  depends on  the assumption on the potential.

\subsection{The perturbed Bessel operator with mixed boundary 
  conditions I}
\label{subsec:mixedI}

In this subsection we describe closed realizations of perturbed Bessel
operators with  mixed boundary conditions under sufficiently strong conditions on the perturbation,
which guarantee that these realizations are very similar to the
unperturbed case.

Let  $m\neq0$, $|\Re(m)|<1$, $\varepsilon=2|\Re(m)|$ and
$Q \in \mathscr{L}^{(0)}_\varepsilon$.
 For 
$\kappa\in\C\cup\{\infty\}$ we set 
\begin{align*}
  \Dom(H_{ m,\kappa}) & := \Big\{f\in \Dom(L_{m^2}^{\max})\mid
\Wr\big(x^{\frac12+m} + \kappa x^{\frac12-m},f;0
\big)=0\Big\},\quad\kappa\in\C, \nonumber \\
\Dom(H_{ m,\infty}) &: = \Big\{f\in \Dom(L_{m^2}^{\max})\mid
\Wr\big( x^{\frac12-m},f;0
\big)=0\Big\},\\
H_{ m,\kappa}&:=L_{ m^2}\big|_{\Dom(H_{ m,\kappa})}.
\end{align*}

If $m=0$ we assume $Q \in \mathscr{L}^{(0)}_{0,\mathrm{ln}^2} $.
For
$\nu\in\C\cup\{\infty\}$ we set
\begin{align*}
\Dom(H_{0}^{\nu}) &: = \Big\{f\in \Dom(L_{0}^{\max})\mid
                    \Wr\big(\nu x^{\frac12} + p_0
                    ,f;0
\big)=0\Big\},\quad\nu\in\C, \nonumber \\
\Dom(H_{ 0}^{\infty}) & := \Dom(H_{ 0}),\\
H_{ 0}^{\nu}&:=L_{ 0}\big|_{\Dom(H_{0}^{\nu})}.
\end{align*}
Note that if $Q \in \mathscr{L}^{(0)}_\varepsilon$ with $\varepsilon>0$, then 
\begin{align*}
\Dom(H_{0}^{\nu}) &: = \Big\{f\in \Dom(L_{0}^{\max})\mid
                    \Wr\big(\nu x^{\frac12} + x^{\frac12}\mathrm{ln}(x)
                    ,f;0
\big)=0\Big\}.
\end{align*}

Clearly, the operators
 $H_{m,\kappa}$,   $H_{0}^\nu$
  are closed,
\begin{align}
  L^{\mathrm{min}}_{m^2} \subset H_{m,\kappa} \subset 
  L^{\mathrm{max}}_{m^2} ,\label{pqo1a}\\
  L^{\mathrm{min}}_0 \subset H_{0}^\nu \subset 
  L^{\mathrm{max}}_0 ,\label{pqo1a-}
\end{align}
and both inclusions in (\ref{pqo1a}) and (\ref{pqo1a-})
are of codimension $1$.

One can compute the resolvents of $H_{m,\kappa}$ in the same way as
for $H_m$.

\begin{proposition}\label{prop:Hmres_mixed}
Let $\Re(k)>0$.
\begin{enumerate}[label=(\roman*)]
\item
Let  $m\neq0$, $|\Re(m)|<1$, $\varepsilon=2|\Re(m)|$,
$Q \in \mathscr{L}^{(0)}_\varepsilon$ and
$\kappa\in\C\cup\{\infty\}$.
We have $k^2\notin \sigma(H_{m,\kappa})$
if and only if $
\Wr_m(k)+\kappa\frac{\Gamma(1-m)}{\Gamma(1+m)}\frac{k^{2m}}{2^{2m}}\Wr_{-m}(k)\neq0$.
Besides, the operator  
$  G_{m,\bowtie,\kappa}(-k^2)$ defined in \eqref{bowtie-kappa}  is
then bounded  and
  \begin{equation*}
  G_{m,\bowtie,\kappa}(-k^2)=(k^2+  H_{m,\kappa})^{-1}.
\end{equation*}
\item  Let $m=0$,   $Q \in
  \mathscr{L}^{(0)}_{0,\mathrm{ln}^2}$ and $\nu \in \C\cup\{\infty\}$.
We have $k^2\notin \sigma(H_{m}^\nu)$ if and only if
$ \Wr( w_0 , \nu  u_0 + p_{0} ) 
\neq 0$. Besides,  the operator $G^\nu_{0}(-k^2)$ with kernel defined in \eqref {bowtie-zero}
is then bounded 
and  \begin{equation*}
  G^\nu_{0,\bowtie}(-k^2)=(k^2+  H^\nu_{0})^{-1}.
\end{equation*}
\end{enumerate}
\end{proposition}
  
\begin{proof}
The argument is the same as in the proof of Proposition \ref{prop:Hmres}.
 \end{proof}

Let us fix a perturbation $Q$ and consider the operator valued
families with mixed boundary conditions.

\begin{theorem}\label{thm:hol_Hma-mix}$ $
   \begin{enumerate}[label=(\roman*)]
   \item
     Let $2>\varepsilon > 0$ and $  Q \in \mathscr{L}^{(0)}_\varepsilon$.
 Then
 \begin{align}\label{holi}
  \Big\{(m,\kappa)\ |\ |\Re(m)|\leq\frac{\varepsilon}{2},\
   \kappa\in \mathbb{C}\cup\{\infty\}, \ (m,\kappa)\neq(0,-1)\Big\}\ni (m,\kappa)&\mapsto
   H_{m,\kappa}
 \end{align}
 is regular.
 Besides, $H_{m,\kappa}=H_{-m,\kappa^{-1}}$.
\item
  Let $m=0$. Suppose that $Q \in \mathscr{L}^{(0)}_{0,\mathrm{ln}^2} $.  Then
 \begin{align}
\mathbb{C}\cup\{\infty\}\ni\nu&\mapsto
   H_0^\nu,\label{porto1k-mix}
 \end{align}
 is   analytic.
   \end{enumerate}
   \end{theorem}

\begin{proof}
This follows as in Theorem \ref{thm:hol_Hma}.
 \end{proof}

\begin{remark}
  Proposition 3.11(ii) in \cite{DeFaNgRi20_01} (in the case $Q=0$) shows that $(m,\kappa)\mapsto H_{m,\kappa}$ cannot be extended by continuity at $(0,-1)$.
\end{remark}


\subsection{The perturbed Bessel operator with mixed boundary 
  conditions II}\label{subsec:mixedII}

As discussed in Theorem \ref{thm-bc} we can define closed realizations of
$L_{m^2}$
under much weaker conditions on $Q$ than those in the previous
subsection. Let us choose a  
nonnegative integer $n$. A natural method to
describe them is by using the boundary functionals defined by
$ u_{-m}^{0[n]}$.

Let $0\leq\Re(m)<1$, $m\neq0$, $\varepsilon=\frac{2\Re(m)}{n+1}$ and
  $ Q \in \mathscr{L}^{(0)}_\varepsilon$.
For
$\kappa\in\C\cup\{\infty\}$ we set
\begin{align*}
\Dom(H_{ -m,\kappa}^{[n]}) & := \left\{f\in \Dom(L_{m^2}^{\max})\mid
\Wr\Big( u_{-m}^{0[n]} + \kappa \frac{\Gamma(1-m)}{\Gamma(1+m)} x^{\frac12+m},f;0
\Big)=0\right\},\quad\kappa\in\C, \nonumber \\
\Dom(H_{ -m,\infty}^{[n]}) &: = \Big\{f\in \Dom(L_{m^2}^{\max})\mid
\Wr\big( x^{\frac12+m},f;0
\big)=0\Big\},\\
H_{- m,\kappa}^{[n]}&:=L_{ m^2}\big|_{\Dom(H_{-m,\kappa})}.
\end{align*}
Clearly, the operators
 $H_{-m,\kappa}^{[n]}$
  are closed,
\begin{align}
  L^{\mathrm{min}}_{m^2} \subset H_{-m,\kappa} ^{[n]}\subset 
  L^{\mathrm{max}}_{m^2} \label{pqo1a+}
\end{align}
and both inclusions in (\ref{pqo1a+})
are of codimension $1$.

Note that in the particular case $n=0$ we have $H_{ -m,\kappa}^{[0]}=H_{-m,\kappa}$.

\begin{proposition}\label{prop:Hmres_mixed+}
 Let $0\leq\Re(m)<1$, $m\neq0$, $\varepsilon=\frac{2\Re(m)}{n+1}$ and
  $ Q \in \mathscr{L}^{(0)}_\varepsilon$.
We have $k^2\notin \sigma(H_{-m,\kappa}^{[n]})$
  if and only if  $ \Wr( v_m (\cdot,k),  u_{-m}^{[n]}(\cdot ,k)) +
  \kappa \frac{\Gamma(1-m)}{\Gamma(1+m)}\Wr_m(k)  \neq0$. Besides, 
the operator  
$  G_{m,\bowtie,\kappa}^{[n]}(-k^2)$ defined in \eqref{bowtie-kappa-n}
is then bounded and
  \begin{equation*}
  G_{m,\bowtie,\kappa}^{[n]}(-k^2)=(k^2+  H_{-m,\kappa}^{[n]})^{-1}.
\end{equation*}
\end{proposition}
  
\begin{proof}
The argument is the same as in the proof of Proposition \ref{prop:Hmres}.
 \end{proof}

   The following theorem can be compared with Theorem  \ref{thm:hol_Hma-mix}:
\begin{theorem}\label{thm:hol_Hma-mix.}
     Let $1>\frac{\varepsilon(n+1)}{2} > 0$ and $  Q \in \mathscr{L}^{(0)}_\varepsilon$.
 Then
 \begin{align}
  \Big\{-\frac{\varepsilon(n+1)}{2}\geq-\Re(m)\geq0,m\neq0\Big\}\times(\mathbb{C}\cup\{\infty\})\ni (-m,\kappa)&\mapsto
   H_{-m,\kappa}^{[n]}
 \end{align}
 is regular.
   \end{theorem}

 \subsection{Scattering length}\label{scatlen}

Suppose that 
\begin{align*}
&Q\in\mathscr{L}^{(\infty)}_{\delta} , \quad \text{if }\ 0\leq\Re(m),\quad m\neq0,\quad \delta =1+2\Re(m); \\
&Q\in\mathscr{L}^{(\infty)}_{1,\mathrm{ln}^2}, \quad \text{if } m=0.
\end{align*}
(Note that we do not impose conditions on $Q$ near $0$ apart from the
usual 
local integrability)
By Propositions \ref{prop:reginfty+} and
\ref{prop:reginftylog}, under these assumptions the space
$\Ker(L_{m^2})$ possesses a distinguished basis
\begin{align*} q_{-m},\quad q_{m},&\quad m\neq0;\\
q_0,\quad q_{0,\mathrm{ln}},&\quad m=0,
\end{align*}
Therefore, the boundary space $\cB_{m^2}$ has a
basis
\begin{align}\label{wrio1}
  \Wr(q_{-m},\cdot;0),\quad \Wr(q_{m},\cdot;0),&\quad m\neq0;\\
  \label{wrio2}
  \Wr(q_0,\cdot;0),\quad \Wr(q_{0,\mathrm{ln}},\cdot;0),&\quad m=0. 
\end{align}

 Suppose that  $H_\bullet$ is one of the realizations of the 
 Bessel operator such that 
 \begin{equation*}
  L_{m^2}^{\min}\subsetneq H_\bullet\subsetneq L_{m^2}^{\max}. 
\end{equation*} 
As we discussed above, to define $H_\bullet$ we need to fix  a
non-zero 
boundary functional. So far, we tried to express boundary functionals 
in terms of the asymptotics of functions near zero, as  in Subsection \ref{bounda-sec}.

In quantum mechanics one often prefers to describe realizations of
perturbed Bessel 
operators using \eqref{wrio1} and \eqref{wrio2}. We say that the {\em scattering length} of $H_\bullet$ is
$a\in\mathbb{C}$ if
\begin{align}
  \cD(H_\bullet)&=\{f\in\cD(L_{m^2}^{\max})\ |\
                  \Wr(q_m-aq_{-m},f;0)=0\},\quad m\neq0;\\
                    \cD(H_\bullet)&=\{f\in\cD(L_{m^2}^{\max})\ |\
                                    \Wr(q_{0,\mathrm{ln}}-aq_0,f;0)=0\},\quad m=0. 
\end{align}
If 
\begin{align}
  \cD(H_\bullet)&=\{f\in\cD(L_{m^2}^{\max})\ |\
                  \Wr(q_{-m},f;0)=0\},\quad m\neq0;\\
                    \cD(H_\bullet)&=\{f\in\cD(L_{m^2}^{\max})\ |\
                                    \Wr(q_0,f;0)=0\},\quad m=0,
\end{align}
then we say that the scattering length of $H_\bullet$ is $a=\infty$.

\appendix

\section{Integral operators}

\label{volterra}

In this appendix we recall a well-known property of Volterra operators that was used in Section \ref{section:solutions}.
We start with the following easy lemma:

\begin{lemma}
  Suppose that $(x,y)\mapsto K(x,y)$ is a measurable function on $]0,\infty[\times]0,\infty[$ such that
  \begin{equation*}
  \sup_{x}\int_0^\infty |K(x,y)|\mathrm{d}y=:C<\infty.
  \end{equation*}
  Then the operator $K$ defined by
  \begin{equation*}
  Kf(x):=\int_0^\infty K(x,y)f(y)\mathrm{d}y
  \end{equation*}
  is bounded on $L^\infty]0,\infty[$ and $\|K\|\leq C$.
\end{lemma}

Given an integral operator $K$ as in the previous lemma and $a>0$, the operator $K^{(a)}$ is defined as an
operator on  $L^\infty]0,a[$ by the kernel
\begin{equation*}
K^{(a)}(x,y):=\theta(a-x)\theta(a-y)K(x,y).
\end{equation*}
The operator $K^{(a)}$ is called {\em the compression of $K$ to $]0,a[$}.

We will say that the operator $K$ with the kernel $K(x,y)$ is a
{\em forward}, respectively {\em backward  Volterra operator} if $K(x,y)=0$ for $x<y$,
respectively $x>y$. The following proposition can be proven by an induction argument.

\begin{proposition}\label{volterra1}
 Suppose that $K$ is a forward Volterra operator and 
  \begin{equation*}
\int_0^\infty |K(x,y)|  \mathrm{d}y=:C(x)<\infty.\]
 Then for any $a>0$, $K^{(a)}$ is bounded 
 and $\one^{(a)}+K^{(a)}$ is invertible on
 $L^\infty]0,a[$. Besides, 
for all $x>0$,  $n\in\mathbb{N}$,
  \begin{equation*}
 |   ( K^nf)(x)|\leq\frac1{n!}C(x)^n\mathop{\mathrm{ess\, sup}}_{y<x} |f(y)|.
\end{equation*}
so that for  $f\in  L^\infty]0,a[$ the  series
 \begin{equation*}
(\one+K)^{-1}f(x)=\sum_{n=0}^\infty(-K)^n f(x)
\end{equation*}
is convergent.
\end{proposition}

\section{Holomorphic families of closed operators}
\label{Holomorphic families of closed operators}

In this appendix we recall the concept of a holomorphic family of
operators on a complex Banach space $\mathcal{H}$.

Let  $\Theta$ be an open subset of $\mathbb{C}^d$.
We say that a family $\{ B(z) \}_{ z \in \Theta }$ of bounded
operators on $\cH$ is a {\em holomorphic family of bounded
  operators}
if for any $f,g\in\cH$
\begin{equation}\label{rwr}
  \Theta\ni z\mapsto (f|B(z)g)\end{equation}
is holomorphic. Note that this is equivalent to a weaker condition:
$\{ B(z) \}_{ z \in \Theta }$ is locally bounded on $\Theta$ and 
there exists a dense subspace $\mathcal{D} \subset \mathcal{H}$ such
that, for all $f , g \in \mathcal{D}$, the map \eqref{rwr} is holomorphic.

One can also introduce another concept: that of {\em holomorphic  families of
closed operators}. We will not give here its general definition, which
can be found e.g. in \cite{Kato,BuDeGe11_01,DeWr} and will not be
used here.
We will restrict ourselves to defining this concept for families that
have nonempty resolvent set.

More precisely,
suppose that $\{ H(z) \}_{ z \in \Theta }$ is a function with values
in closed operators on $\cH$.
Suppose that  for any $z_0 \in \Theta$, there exist $\lambda  \in
\mathbb{C}$ and a neighborhood $\Theta_0 \subset \Theta$ of $z_0$ such
that, for all $z \in \Theta_0$, $\lambda $ is in the resolvent set of
$H(z)$. Then we say that  $\{ H(z) \}_{ z \in \Theta }$ is holomorphic
if for all such $\Theta_0$ the map $\Theta_0 \ni z \mapsto ( H(z) - \lambda  )^{-1} \in \mathcal{L}( \mathcal{H} )$ is holomorphic  as a family of bounded operators.

\section{Technical lemma}

The following easy lemma was used several times in the main part of the manuscript.
\begin{lemma}\label{lm:integr}
Let $a>0$ or $a=\infty$. Let $f\in L^1_{\mathrm{loc}}]0,a[$  and $h:]0,a[\to\mathbb{R}$ be a positive increasing function such that $h(x)\to0$ as $x\to0$ and
\begin{equation*}
\int_0^a h(x) |f(x)| \mathrm{d}x < \infty.
\end{equation*}
Then
\begin{equation*}
\int_x^a f(y) \mathrm{d}y = o \big ( h(x)^{-1} \big ) , \quad x \to 0.
\end{equation*}
\end{lemma}
\begin{proof}
Let
\begin{equation*}
w(x,y) := h(x) |f(y)| \mathds{1}_{[x,a[}(y).
\end{equation*}
Clearly, for a.e. $y\in]0,a[$, $w(x,y) \to 0$ as $x\to0$. Moreover, since $h$ is increasing, 
\begin{equation*}
w(x,y) \le h(y) |f(y)| ,
\end{equation*}
for all $x,y\in]0,a[$. Since $y\mapsto h(y) |f(y)|$ is integrable on $]0,a[$ by assumption, the dominated convergence theorem implies that
\begin{equation*}
\int_0^a w(x,y) \mathrm{d}y \to 0 , \quad x \to 0.
\end{equation*}
This proves the lemma.
\end{proof}

\renewcommand{\abstractname}{Acknowledgements}

\begin{abstract}
 We thank K. Yajima for useful comments. J.D. acknowledges the support from the National Science Centre (NCN
  Project Nr. 2018/31/G/ST1/01166). 
\end{abstract}

\end{document}